\documentclass[12pt,oneside,reqno]{amsart}
\usepackage[utf8]{inputenc}
\usepackage{amsmath}
\usepackage{amsthm,ifthen, amsfonts, amssymb,
srcltx, amsopn, color, enumerate}
 \usepackage{mathabx}
\usepackage{amssymb}
\usepackage[mathscr]{euscript}
\usepackage{enumerate}
\usepackage{graphicx}
\usepackage{fullpage}
\usepackage{tikz}
\usepackage{color}
\usepackage{multicol}
\usepackage{mathrsfs}
\usepackage{hyperref}
\usepackage[normalem]{ulem}
\usepackage{float}
\usepackage{subcaption}
\usepackage{mathabx}
\usepackage{hieroglf}
\usepackage[T1]{fontenc}
\usepackage[cmtip,arrow]{xy}
\usepackage{pb-diagram, pb-xy}
\usepackage{overpic}
\usepackage{verbatim}
\usepackage{calc}
\usepackage{transparent}

\definecolor{Green}{RGB}{30, 150, 30}

\usetikzlibrary{arrows}
\usepackage{tikz-cd} %Make Commutative diagrams and exact sequences

\tikzcdset{row sep/normal=1cm}
\tikzcdset{column sep/normal=1cm}

\newtheorem{thm}{Theorem}[section]
\newtheorem{prop}[thm]{Proposition}
\newtheorem{lem}[thm]{Lemma}
\newtheorem{cor}[thm]{Corollary}

\theoremstyle{definition}
\newtheorem{defn}[thm]{Definition}
\theoremstyle{definition}

\theoremstyle{definition}

\theoremstyle{definition}
\newtheorem{remark}[thm]{Remark}
\theoremstyle{definition}

\theoremstyle{definition}
\newtheorem{convention}[thm]{Convention}
\theoremstyle{definition}

\newtheorem{claim}[thm]{Claim}

\newcommand*\Z{\mathbb{Z}}

\newcommand*\supp{\operatorname{supp}}

\newcommand*\Stab{\operatorname{Stab}}
\newcommand*\diam{\operatorname{diam}}
\newcommand*\nest{\sqsubseteq}
\newcommand*\propnest{\sqsubsetneq}
\newcommand*\nestneq{\sqsubsetneq}
\newcommand*\mf[1]{\mathfrak{#1}}
\newcommand*\mc[1]{\mathcal{#1}}
\newcommand*\trans{\pitchfork}
\newcommand{\dist}{\textup{\textsf{d}}}
\newcommand*\gate{\mathfrak{g}}
\newcommand*\sub[1]{\langle {#1} \rangle}

\newcommand{\tsh}[1]{\left\{\kern-.7ex\left\{#1\right\}\kern-.7ex\right\}}
\newcommand{\Tsh}[2]{\tsh{#2}_{#1}}
\newcommand{\threshold}[2]{\Tsh{#2}{#1}}
\DeclareMathOperator*{\free}{\raisebox{-0.6ex}{\scalebox{2.5}{$\ast$}}}
\newcommand{\llangle}{\left\langle\!\left\langle}
\newcommand{\rrangle}{\right\rangle\!\right\rangle}

\newcommand{\prefix}{\operatorname{prefix}}
\newcommand{\suffix}{\operatorname{suffix}}
\newcommand*\lk{\operatorname{lk}}
\newcommand*\st{\operatorname{st}}
\newcommand{\syl}[1]{| {#1} |_{syl}}

\newcommand{\Snest}{\nest_{\mf{S}}}
\newcommand{\Sperp}{\perp_{\mf{S}}}
\newcommand{\Strans}{\trans_{\mf{S}}}
\newcommand{\Rnest}{\nest_{\mf{R}}}
\newcommand{\Rperp}{\perp_{\mf{R}}}
\newcommand{\Rtrans}{\trans_{\mf{R}}}
\newcommand{\oldProj}{\pi}
\newcommand{\newProj}{\psi}
\newcommand{\oldRel}{\rho}
\newcommand{\newRel}{\beta}
\title[Graph products and HHGs]{ 
Hierarchical hyperbolicity of graph products}
\author{Daniel Berlyne}

\address{The Graduate Center, City University of New York, U.S.A.}
\email{dberlyne@gradcenter.cuny.edu}
\author{Jacob Russell}
\address{Rice University, Houston, U.S.A.}
\email{jacob.russell@rice.edu}

\begin{document}

\begin{abstract}
We show that any graph product of finitely generated groups is hierarchically hyperbolic relative to its vertex groups. We apply this result to answer two questions of Behrstock, Hagen, and Sisto: we show that the syllable metric on any graph product forms a hierarchically hyperbolic space, and that graph products of hierarchically hyperbolic groups are themselves hierarchically hyperbolic groups. This last result is a strengthening of a result of Berlai and Robbio by removing the need for extra hypotheses on the vertex groups.  We also answer two questions of Genevois about the geometry of the electrification of a graph product of finite groups. 
\end{abstract}

%\begin{classification}
%Primary: 20F65; Secondary: 20F67, 20F55, 20E08.
%\end{classification}

%\begin{keywords}
%Hierarchically hyperbolic, graph product.
%\end{keywords}

\maketitle
\tableofcontents

\section{Introduction}

There have been many attempts to generalise the notion of hyperbolicity of a group since it was first introduced by Gromov \cite{gromov1}. One of these, \emph{hierarchical hyperbolicity}, was developed by Behrstock, Hagen, and Sisto \cite{BHS_HHSI,BHS_HHSII} as a way of describing hyperbolic behaviour in quasi-geodesic metric spaces via hierarchy machinery akin to that constructed for mapping class groups by Masur and Minsky \cite{MMI,MMII}.  The work of Behrstock, Hagen, and Sisto originally focused on developing such machinery for right-angled Artin groups, but also encompasses a wide variety of groups and spaces, such as virtually cocompact special groups \cite{BHS_HHSII}, 3--manifold groups with no Nil or Sol components \cite{BHS_HHSII}, Teichm\"{u}ller space with either the Teichm\"{u}ller or Weil--Petersson metric \cite{BHS_HHSI,MMI,BKMM_Rigidity,Brock_WeilPetersson,Durham_Combinatorial_Teich,Rafi_Combinatorial_Teich,EMR_Teich_Rank}, and graph products of hyperbolic groups \cite{BR_Combination}. Hierarchical hyperbolicity has deep geometric consequences for a space, including a Masur and Minsky style distance formula \cite{BHS_HHSII}, a quadratic isoperimetric inequality \cite{BHS_HHSII}, rank rigidity and Tits alternative theorems \cite{HHS_Boundary,DHS_corrigendum}, control over top-dimensional quasi-flats \cite{BHS_HHS_Quasiflats}, and bounds on the asymptotic dimension \cite{BHS_HHS_AsDim}.

 A hierarchically hyperbolic structure on a quasi-geodesic space $\mc{X}$ is a collection of uniformly hyperbolic spaces $C(W)$ indexed by the elements $W$ of an index set $\mf{S}$. For each $W \in \mf{S}$, there is a projection map from $\mc{X}$ onto the hyperbolic space $C(W)$, and every pair of elements of $\mf{S}$ is related by one of three mutually exclusive relations: orthogonality, nesting, and transversality. This data then satisfies a collection of axioms that allow for the coarse geometry of the entire space to be recovered from the projections to the hyperbolic spaces $C(W)$.

In the present paper, we construct an explicit hierarchy structure for any graph product, using right-angled Artin groups as our motivating example. Given a finite simplicial graph $\Gamma$ with vertex set $V(\Gamma)$ and edge set $E(\Gamma)$, we define the \emph{right-angled Artin group} $A_{\Gamma}$ by 
\[ A_{\Gamma} = \langle V(\Gamma) \,\,|\,\, [v,w]=e \,\,\forall\,\, \{v,w\} \in E(\Gamma) \rangle. \]
More generally, if we associate to each vertex $v$ of $\Gamma$ a finitely-generated group $G_{v}$, then we define the \emph{graph product} $G_{\Gamma}$ by 
\[ G_{\Gamma} = \left.\left(\free_{v \in V(\Gamma)} G_{v}\right) \right/ \llangle [g_{v},g_{w}] \,\, \middle| \,\, g_{v} \in G_{v},\, g_{w} \in G_{w},\, \{v,w\} \in E(\Gamma) \rrangle, \]
so that $A_{\Gamma}$ is obtained as the special case where the vertex groups are $G_{v} = \mathbb{Z}$ for all $v \in V(\Gamma)$.

For right-angled Artin groups $A_{\Gamma}$, a hierarchically hyperbolic structure was constructed by Behrstock, Hagen, and Sisto by considering the collection of induced subgraphs of the defining graph $\Gamma$  \cite{BHS_HHSI}. Each induced subgraph $\Lambda$ of $\Gamma$ generates a new right-angled Artin group $A_\Lambda$, which is realised as a subgroup of $A_{\Gamma}$. The Cayley graph of $A_{\Gamma}$ is the $1$--skeleton of a CAT($0$) cube complex $X$, which comes equipped with a projection to a hyperbolic space $C(X)$ called the \emph{contact graph}. Since each induced subgraph $\Lambda$ of $\Gamma$ generates its own right-angled Artin group with associated cube complex $Y \subseteq X$, the subgroup $A_\Lambda$ has its own associated contact graph $C(Y)$. Since edges of $\Gamma$ correspond to commuting relations in $A_{\Gamma}$, join subgraphs of $\Gamma$ (that is, subgraphs of the form $\Lambda_{1} \sqcup \Lambda_{2}$ where every vertex of $\Lambda_{1}$ is joined by an edge to every vertex of $\Lambda_{2}$) generate direct product subgroups of $A_{\Gamma}$.  This provides us with an intuitive notion of \emph{orthogonality} within our hierarchy. Set containment of subgraphs of $\Gamma$ provides a natural partial order in the hierarchy, which we call \emph{nesting}, and any subgraphs that are not orthogonal or nested are considered \emph{transverse}. Collectively, the hyperbolic spaces $C(Y)$ allow us to recover the entire geometry of $A_{\Gamma}$, via projections to the subcomplexes $Y \subseteq X$ and through the nesting, orthogonality and transversality relations defined above.

Since the nesting and orthogonality relations for a right-angled Artin group are intrinsic to the defining graph $\Gamma$, it is sensible to attempt to generalise this hierarchy structure to arbitrary graph products. It is important to note, however, that arbitrary graph products may \emph{not}  be hierarchically hyperbolic, since we have no control over the vertex groups. For example, the vertex groups could be copies of Out$(F_{3})$, which is known not to be hierarchically hyperbolic \cite{BHS_HHSII}. However, this is the only roadblock. Specifically, we show that graph products are \emph{relatively hierarchically hyperbolic}, that is, graph products admit a structure satisfying all of the axioms of hierarchical hyperbolicity with the exception that the  the spaces associated to the nesting-minimal sets (the vertex groups) are not necessarily hyperbolic. 

\renewcommand*{\thethm}{\Alph{thm}}

\begin{thm}\label{intro_thm:main_thm}
Let $\Gamma$ be a finite simplicial graph, with each vertex $v$ labelled by a finitely-generated group $G_{v}$. The graph product $G_{\Gamma}$ is a hierarchically hyperbolic group relative to the vertex groups.
\end{thm}

The notion of relative hierarchical hyperbolicity was originally developed by Behrstock, Hagen, and Sisto in \cite{BHS_HHSII} and is explored further in \cite{BHS_HHS_AsDim}. 
Despite the lack of hyperbolicity in the nesting-minimal sets, many of the consequences of hierarchical hyperbolicity are preserved in the relatively hierarchically hyperbolic setting. In particular, Theorem \ref{intro_thm:main_thm} implies the graph product $G_\Gamma$ has a Masur and Minsky style distance formula and an acylindrical action on the nesting-maximal hyperbolic space; see Corollaries \ref{cor:distance_formula} and \ref{cor:acyl}.

Another way of asserting control over the vertex groups is by replacing the word metric on $G_{\Gamma}$ with the \emph{syllable metric},  which measures the length of an element $g \in G_{\Gamma}$ by counting the  minimal number of elements needed to express $g$ as a product of vertex group elements. 
This has the effect of making all vertex groups diameter $1$, and therefore hyperbolic. The syllable metric on a right-angled Artin group was studied by Kim and Koberda as an analogue of the Weil--Petersson metric on Teichm\"uller space (the Weil--Petersson metric is quasi-isometric to the space obtained from the mapping class group by coning off all cyclic subgroups generated by Dehn twists)  \cite{KimKoberda}. Kim and Koberda produce several hierarchy-like results for the syllable metric on a right-angled Artin group with triangle- and square-free defining graph, including a Masur and Minsky style distance formula and an acylindrical action on a hyperbolic space. This inspired Behrstock, Hagen, and Sisto to ask if the syllable metric on a right-angled Artin group is a hierarchically hyperbolic space \cite{BHS_HHSII}. We give a positive answer to this question, not just for right-angled Artin groups but for all graph products.

\begin{cor}\label{intro_cor:syllable_is_HHG}
Let $\Gamma$ be a finite simplicial graph, with each vertex $v$ labelled by a group $G_{v}$. The graph product $G_{\Gamma}$ endowed with the syllable metric is a hierarchically hyperbolic space.
\end{cor}

To prove Theorem \ref{intro_thm:main_thm} and Corollary \ref{intro_cor:syllable_is_HHG}, we utilise techniques developed by Genevois and Martin in \cite{Gen_Thesis,Genevois_Martin}, which exploit the cubical-like geometry of a graph product when endowed with the syllable metric. This allows us to adapt proofs from the right-angled Artin group case, which rely heavily on geometric properties of cube complexes. While the syllable metric does not appear in the statement of Theorem \ref{intro_thm:main_thm}, it is an integral part of the proof, acting as a middle ground where geometric computations are performed before projecting to the associated hyperbolic spaces. This also allows Theorem \ref{intro_thm:main_thm} and Corollary \ref{intro_cor:syllable_is_HHG} to be proved essentially simultaneously.

 Our primary application of Theorem \ref{intro_thm:main_thm} is showing that a graph product of hierarchically hyperbolic groups is itself hierarchically hyperbolic. This gives a positive answer to another question of Behrstock, Hagen, and Sisto \cite[Question D]{BHS_HHSII}. 

\begin{thm}\label{intro_thm:Graph_Products_of_HHGs_are_HHGs}
Let $\Gamma$ be a finite simplicial graph, with each vertex $v$ labelled by a group $G_{v}$. If each $G_v$ is a hierarchically hyperbolic group, then the graph product $G_{\Gamma}$ is a hierarchically hyperbolic group.
\end{thm}

Berlai and Robbio have established a combination theorem for graphs of hierarchically hyperbolic groups that they use to prove Theorem \ref{intro_thm:Graph_Products_of_HHGs_are_HHGs} when the vertex groups satisfy some natural, but non-trivial, additional hypotheses \cite{BR_Combination}. For the specific case of graph products, Theorem \ref{intro_thm:Graph_Products_of_HHGs_are_HHGs} improves upon Berlai and Robbio's result by removing the need for these additional hypotheses, as well as providing an explicit description of the hierarchically hyperbolic structure in terms of the defining graph. 

We also use our relatively hierarchically hyperbolic structure for graph products to answer two questions of Genevois about a new quasi-isometry invariant for graph products of finite groups called the \emph{electrification} of $G_{\Gamma}$. Graph products of finite groups form a particularly interesting class, as they include right-angled Coxeter groups and are the only cases where the syllable metric and word metric are quasi-isometric.  Genevois defines the electrification $\mathbb{E}(\Gamma)$  of a graph product of finite groups by taking the syllable metric on $G_\Gamma$ and adding  edges between elements $g,h$  whenever $g^{-1}h \in G_\Lambda \leq G_\Gamma$ and $\Lambda$ is a \emph{minsquare} subgraph of $\Gamma$, that is, a minimal subgraph that contains opposite vertices of a square if and only if it contains the whole square. Motivated by an analogy with relatively hyperbolic groups, Genevois proved that any quasi-isometry between graph products of finite groups induces a quasi-isometry between their electrifications, and used this invariant to distinguish several quasi-isometry classes of right-angled Coxeter groups \cite{Gen_Rigidity}. Geometrically, the electrification  sits between the syllable metric on $G_\Gamma$ and the nesting-maximal hyperbolic space in our hierarchically hyperbolic structure on $G_\Gamma$. We exploit this situation to classify when the electrification  has bounded diameter and when it is a quasi-line, answering Questions 8.3 and 8.4 of \cite{Gen_Rigidity}.

\begin{thm}\label{intro_thm:electrification}
Let $G_\Gamma$ be a graph product of finite groups and let $\mathbb{E}(\Gamma)$ be its electrification.
\begin{enumerate}
    \item $\mathbb{E}(\Gamma)$ has bounded diameter if and only if $\Gamma$ is either a complete graph, a minsquare graph, or the join of minsquare graph and a complete graph.
    \item $\mathbb{E}(\Gamma)$ is a quasi-line if and only if $G_\Gamma$ is virtually cyclic.
\end{enumerate}
\end{thm}

As a final application of Theorem \ref{intro_thm:main_thm}, we give a new proof of Meier's classification of hyperbolicity of graph products \cite{Meier}.

\vspace{3mm}
\noindent\textbf{Outline of the paper.} We begin by introducing the necessary tools from the geometry of graph products in Section \ref{section:backgroup_graph_products} and reviewing the definition of a relatively hierarchically hyperbolic group in Section \ref{section:backgroup_HHS}. In Section \ref{section:proto}, we set up our proof of the relative hierarchical hyperbolicity of graph products by defining the necessary spaces, projections, and relations. In Section \ref{section:rel_HHG}, we show the spaces, projections, and relations defined in Section \ref{section:proto} satisfy the axioms of a relative HHG (or non-relative HHS in the case of the syllable metric). This completes the proofs of Theorem \ref{intro_thm:main_thm} and Corollary \ref{intro_cor:syllable_is_HHG}. Section \ref{section:applications} is devoted to applications. We start by proving graph products of HHGs are HHGs (Theorem \ref{intro_thm:Graph_Products_of_HHGs_are_HHGs}) in Section \ref{section:graph_products_of_HHG}, which requires a technical result that can be found in the appendix of \cite{ABD}. In Section \ref{section:hyperbolic_graph_products}, we record our proof of Meier's hyperbolicity criteria, and in Section \ref{section:electrification}, we classify when Genevois' electrification has infinite diameter and when it is a quasi-line, proving Theorem \ref{intro_thm:electrification}.

\vspace{3mm}
\noindent \textbf{Acknowledgements.} The authors would like to thank Mark Hagen for his useful comments and suggestions for applications, Anthony Genevois for his comments on an earlier draft of the paper, and their advisor Jason Behrstock for many valuable conversations and his support throughout. The authors would also like to thank an  anonymous referee for comments that improved the paper.

\section{Background}

\renewcommand*{\thethm}{\thesection.\arabic{thm}}

\subsection{Graph products}\label{section:backgroup_graph_products}

\begin{defn}[Graph product]
Let $\Gamma$ be a finite simplicial graph with vertex set $V(\Gamma)$ and edge set $E(\Gamma)$, and with each vertex $v \in V(\Gamma)$ labelled by a  group $G_{v}$.  The \emph{graph product} $G_\Gamma$ is the group 
\[ G_{\Gamma} = \left.\left(\free_{v \in V(\Gamma)} G_{v}\right) \right/ \llangle [g_{v},g_{w}] \,\, \middle| \,\, g_{v} \in G_{v},\, g_{w} \in G_{w},\, \{v,w\} \in E(\Gamma) \rrangle. \]
 We call the $G_v$ the \emph{vertex groups} of the graph product $G_\Gamma$. 
 
By deleting any vertices labelled by the trivial group, every graph product is isomorphic to a graph product where each vertex group is non-trivial. We will therefore will operate under the standing assumption that the vertex groups of our graph products are always non-trivial.
\end{defn}

Note that if all vertex groups of $G_\Gamma$ are copies of $\mathbb{Z}$ then $G_{\Gamma}$ is the right-angled Artin group with defining graph $\Gamma$, and if all vertex groups are copies of $\mathbb{Z}/2\mathbb{Z}$ then $G_{\Gamma}$ is the corresponding right-angled Coxeter group.

We wish to study the geometry of $G_{\Gamma}$ by adapting the cubical geometry of right-angled Artin groups. To this end, we will first need to eliminate any badly behaved geometry occurring within vertex groups. We do this by replacing the usual word metric with the \emph{syllable metric}.

\begin{defn}[Syllable metric on a graph product]\label{defn:syllable_metric}
  Let $G_\Gamma$ be a graph product. The graph $S(\Gamma)$ is the metric graph whose vertices are elements of $G_\Gamma$ and where $g,h \in G_\Gamma$ are joined by an edge of length 1 labelled by $g^{-1}h$ if there exists a vertex $v$ of $\Gamma$ such that $g^{-1}h \in G_{v}$. We denote the distance in $S(\Gamma)$ by $\dist_{syl}(\cdot,\cdot)$ and say $\dist_{syl}(g,h)$ is the \emph{syllable distance} between $g$ and $h$. When convenient, we will use $|g|_{syl}$ to denote $\dist_{syl}(e,g)$ and call it the \emph{syllable length} of $g$.
\end{defn}

Notice that all cosets of vertex groups have diameter $1$ under the syllable metric, thus trivialising their geometry. Therefore, when working with $S(\Gamma)$, instead of expressing an element $g \in G_{\Gamma}$ as a word in the generators of $G_{\Gamma}$, it is more geometrically meaningful to express $g$ as a product of \emph{any} elements of vertex groups.

\begin{defn}[Syllable expressions]
 Let $G_\Gamma$ be a graph product and $g \in G_\Gamma$. If $g = s_1 \dots s_n$ where each $s_i \in G_{v_{i}}$ for some $v_i \in V(\Gamma)$, then we say $s_1 \dots s_n$ is a \emph{syllable expression} for $g$. If $s_1 \dots s_n$ is a syllable expression for $g$ and $n = \dist_{syl}(e,g)$, then we say $s_1 \dots s_n$ is a \emph{reduced syllable expression} for $g$. In this case, $n$ is the smallest number of terms possible for any syllable expression of $g$. 
\end{defn}

A foundational fact about graph products is that any syllable expression can be reduced by applying a sequence of canonical moves.

\begin{thm}[Reduction algorithm for graph products; {\cite[Theorem 3.9]{Green}}] \label{thm:graph_product_normal_form}
Let $G_\Gamma$ be a graph product and $g \in G_\Gamma$. If $s_1\dots s_n$ is a reduced syllable expression for $g$ and $t_1 \dots t_m$ is a syllable expression for $g$, then $t_1\dots t_m$ can be transformed into $s_1 \dots s_n$ by applying a sequence of the following three moves.
\begin{itemize}
    \item Remove a term $t_i$ if $t_i = e$.
    \item Replace consecutive terms $t_i$ and  $t_{i+1}$ belonging to the same vertex group $G_v$ with the single term $t_{i}t_{i+1} \in G_v$.
    \item Exchange the position of consecutive terms $t_i$ and $t_{i+1}$ when $t_i \in G_v$ and $t_{i+1} \in G_w$ with $v$ joined to $w$ by an edge in $\Gamma$.
    \end{itemize}
\end{thm}

The next corollary shows that when each of the vertex groups of the graph product is finitely generated,  Theorem \ref{thm:graph_product_normal_form} implies that the word length of any $g \in G_{\Gamma}$ will be the sum of  the word lengths of the terms in any reduced syllable expression for $g$.

\begin{cor}[Reduced syllable expressions minimise word length]\label{cor:reduced_syllable_expressions_minimize_word_length}
Let $G_\Gamma$ be a graph product of finitely generated groups. For each $v \in V(\Gamma)$, let $S_v$ be a finite generating set for the vertex group $G_v$, and  let $|g|$ be the word length of $g \in G_\Gamma$ with respect to the finite generating set $S = \bigcup_{v \in V(\Gamma)} S_v$.  For all $g \in G_\Gamma$, if $s_1\dots s_n$ is a reduced syllable expression for $g$, then \[|g| = \sum_{i=1}^n |s_i|.\]
\end{cor}

\begin{proof}
Let $s_1\dots s_n$ be a reduced syllable expression for $g \in G_\Gamma$. There exist $w_1,\dots, w_m \in S$ such that $|g| = m$ and $g = w_1\dots w_m$. Since every element of $S$ is an element of one of the vertex groups of $G_\Gamma$, the product $w_1\dots w_m$ is also a syllable expression for $g$. Thus, by applying a finite number of the moves from 
Theorem \ref{thm:graph_product_normal_form}, we can transform $w_1\dots w_m$ into $s_1 \dots s_n$. We  can therefore write each $s_i$ as a product  $s_i = w_{\sigma_i({1})} \dots w_{\sigma_i({m_i})}$, where $m_i \leq m$ and $\sigma_i$ is a permutation of $\{1,\dots, m\}$.  Further, if $i \neq k$, then  $\{\sigma_i(1),\dots,\sigma_i(m_i)\} \cap \{\sigma_k(1), \dots, \sigma_k(m_k)\} = \emptyset$. Thus, $ \sum_{i=1}^n |s_i| \leq \sum_{i=1}^n m_i \leq m$. However, $m = |g| \leq \sum_{i=1}^n |s_i| $ by definition, so $|g| = \sum_{i=1}^n |s_i|$.
\end{proof}

Another critical consequence of Theorem \ref{thm:graph_product_normal_form} is that the terms in a reduced syllable expression for an element of a graph product are well-defined up to applying the commutation relation. This ensures that the following notions are well-defined for an element of a graph product.

\begin{defn}[Syllables and support of an element]
Let $G_\Gamma$ be a graph product and let $g \in G_\Gamma$. If $s_1 \dots s_n$ is a reduced syllable expression for $g$, then we call the $s_i$ the \emph{syllables} of $g$ and use $\supp(g)$ to denote the maximal subgraph of $\Gamma$ with vertex set $\{v_1,\dots,v_n\}$, where $s_i \in G_{v_i}$. We call $\supp(g)$ the \emph{support} of $g$. 
\end{defn}

\begin{convention}
Whenever we consider a subgraph  $\Lambda \subseteq \Gamma$, we will assume that $\Lambda$ is both non-empty and an induced subgraph of $\Gamma$. That is, whenever $v,w$ are vertices of $\Lambda$ that are joined by an edge of $\Gamma$, then $v$ and $w$ are joined by an edge of $\Lambda$ as well.
\end{convention}

Another hallmark feature of graph products is their rich collection of subgroups corresponding to subgraphs of the defining graph.

\begin{defn}[Graphical subgroups]
Let $G_\Gamma$ be a graph product with vertex groups $\{G_v:v \in V(\Gamma)\}$ and let $\Lambda \subseteq \Gamma$ be a subgraph.  We use $\sub{\Lambda}$ to denote the subgroup of $G_\Gamma$ generated by $\{G_v : v \in V(\Lambda)\}$. We call such subgroups the \emph{graphical subgroups} of $G_\Gamma$. Note, each subgroup $\sub{\Lambda}$ is isomorphic to the graph product $G_\Lambda$.
\end{defn}

Since the graphical subgroups are themselves graph products, we can also define the syllable metric on them and their cosets. 

\begin{defn}[Syllable metric on graphical subgroups]
Let $G_\Gamma$ be a graph product, $g \in G_\Gamma$, and $\Lambda \subseteq \Gamma$. Let $S(\Lambda)$  be the metric graph defined in Definition \ref{defn:syllable_metric} for the graph product $\sub{\Lambda}$, and let $S(g\Lambda)$ denote the metric graph whose vertices are elements of the coset $g\sub{\Lambda}$ and where $gx$ and $gy$ are joined by an edge of length 1 if $x$ and $y$ are joined by an edge in $S(\Lambda)$. 
\end{defn}

\begin{remark}[Graphical subgroups are convex in $S(\Gamma)$]
Geodesics in $S(\Gamma)$ between two elements $k$ and $h$ are labelled by the reduced syllable forms of $k^{-1}h$.  The induced subgraph of $S(\Gamma)$ with vertex set $g \sub{\Lambda}$ is therefore convex and graphically isomorphic  to $S(g\Lambda)$ via the identity map. In particular, the distance between two vertices $k,h$ of $S(g\Lambda)$ is $\dist_{syl}(k,h)$.
\end{remark}

In order to analyse how the graphical subgroups of $G_\Gamma$ interact, we make extensive use of the following definitions from graph theory.

\begin{defn}[Star, link, and join]
Let $\Gamma$ be a finite simplicial graph and $\Lambda$ a subgraph of $\Gamma$. The \emph{link} of $\Lambda$, denoted $\lk(\Lambda)$, is the subgraph spanned by the vertices of $\Gamma \smallsetminus \Lambda$ that are connected to every vertex of $\Lambda$. The \emph{star} of $\Lambda$, denoted $\st(\Lambda)$, is $\Lambda \cup \lk(\Lambda)$. We say $\Lambda$ is a \emph{join} if the vertices of $\Lambda$ can be expressed as $V(\Lambda) = V(\Lambda_{1}) \cup V(\Lambda_{2})$ where $\Lambda_{1}$ and $\Lambda_{2}$ are disjoint subgraphs of $\Gamma$ and every vertex of $\Lambda_{1}$ is connected to every vertex of $\Lambda_{2}$. We denote the join of $\Lambda_{1}$ and $\Lambda_{2}$ by $\Lambda_{1} \bowtie \Lambda_{2}$. In particular, $\st(\Lambda)$ is the join $\Lambda \bowtie \lk(\Lambda)$.
\end{defn}

\begin{remark}\label{rem:star_elements}
The star, link, and join have important algebraic significance. A join subgraph of $\Gamma$ generates a subgroup of $G_{\Gamma}$ which splits as a direct product, while $\langle\st(\Lambda)\rangle$ is the largest subgroup of $G_{\Gamma}$ which splits as a direct product with $\langle\Lambda\rangle$ as one of the factors: $\langle\st(\Lambda)\rangle  = \langle\Lambda\rangle \times \langle\lk(\Lambda)\rangle$. Moreover, since every element of $\langle\Lambda\rangle$ commutes with every element of $\langle\lk(\Lambda)\rangle$, the reduced syllable form tells us that we can always write an element $g \in \langle\st(\Lambda)\rangle$ in the form $g=\lambda l$, where $\lambda \in \langle\Lambda\rangle$ and $l \in \langle\lk(\Lambda)\rangle$. 
\end{remark}

Genevois observed that the graph $S(\Gamma)$ is almost a cube complex, with the only non-cubical behaviour arising from the vertex groups. More precisely, he showed the following result.

\begin{prop}[{\cite[Lemmas 8.5, 8.8]{Gen_Thesis}}]\label{prop:squares_in_S} Two adjacent edges of $S(\Gamma)$ are  edges of a triangle if and only if they are labelled by elements of the same vertex group. Two adjacent edges of $S(\Gamma)$ are edges of an induced square if and only if they are labelled by elements of adjacent vertex groups. In this case,  opposite edges of the square are labelled by the same vertex groups.
\end{prop}

The above proposition means that while $S(\Gamma)$ is not a cube complex, it is the $1$--skeleton of a complex built from \emph{prisms} glued isometrically along subprisms. Henceforth, we will interchangeably refer to $S(\Gamma)$ and the canonical cell complex of which it is the $1$--skeleton.

\begin{defn}[Prism]
A \emph{prism} $P$ of $S(\Gamma)$ is a subcomplex which can be written as a product of simplices $P = T_{1} \times \dots \times T_{m}$.
\end{defn}

Since a cube is a product of $1$--simplices, prisms generalise the cubes in a cube complex. Genevois used the prisms in $S(\Gamma)$ to build hyperplanes with very similar properties to those in CAT(0) cube complexes. We present a slightly different, but equivalent, construction of these hyperplanes in $S(\Gamma)$.

 In a cube complex, hyperplanes are built from mid-cubes. If we view each cube in a cube complex as a product $\left[-\frac{1}{2},\frac{1}{2}\right]^n$, we obtain a \emph{mid-cube} by restricting one of the intervals $\left[-\frac{1}{2},\frac{1}{2}\right]$ to $0$. In much the same way, we obtain a \emph{mid-prism} from a prism by performing a  modified  barycentric subdivision on one of its simplices. If this simplex is a $1$--simplex, this just gives us the midpoint of the edge.

\begin{defn}[Mid-prism]
Given an $n$--simplex $T$ in $S(\Gamma)$, perform a modified barycentric subdivision as follows. First add a vertex at the barycentre of each sub-simplex of $T$. Then for each $2 \leq k \leq n$, add edges connecting the barycentre of each $k$--simplex in $T$ to the barycentres of each of its $(k-1)$--sub-simplices; see Figure \ref{fig:midprism}. The complex we have added through this procedure is then the $1$--skeleton of a canonical simply connected cell complex, which we denote by $K(T)$.
We call $K(T)$ the \emph{mid-prism} of $T$. More generally, we define a \emph{mid-prism} $K_{i}$ of a prism $P = T_{1} \times \dots \times T_{m}$ to be the product $K_{i} = T_{1} \times \dots \times T_{i-1} \times K(T_{i}) \times T_{i+1} \times \dots \times T_{m}$.
\end{defn}

\begin{figure}[ht]
\centering
\def\svgscale{1}
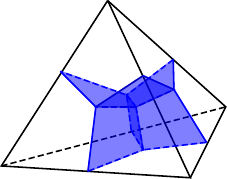 \hspace{.1in} 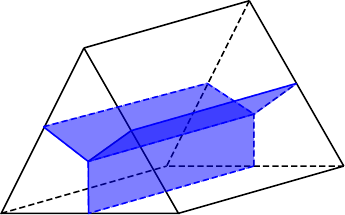
\caption{The mid-prism of a $3$--simplex and a mid-prism of the product of a $2$--simplex and a $1$--simplex.}\label{fig:midprism}
\end{figure}

Note that the simplices in $S(\Gamma)$ that arise from infinite vertex groups have infinitely many vertices. A simplex with infinitely many vertices may still be assigned a mid-prism, by constructing mid-prisms for each of its finite sub-simplices. The inductivity of the barycentric subdivision procedure ensures that these mid-prisms all agree with each other.

A \emph{hyperplane} of a cube complex is defined to be a maximal connected  union of mid-cubes. In the same way, we can construct hyperplanes in $S(\Gamma)$ by taking maximal connected  unions of mid-prisms.

\begin{defn}[Hyperplane, carrier]\label{defn:hyperplane}
Construct an equivalence relation $\sim$ on the edges of $S(\Gamma)$ by  defining $E_{1} \sim E_{2}$ if $E_{1}$ and $E_{2}$ are either opposite sides of a square or two sides of a triangle, and then extending transitively. We say the \emph{hyperplane} dual to the equivalence class $[E]$ is the  union of all mid-prisms that intersect edges of $[E]$; see Figure \ref{fig:hyperplane_eg}. The \emph{carrier} of the hyperplane dual to $[E]$ is the union of all prisms that  contain edges of $[E]$.

If a geodesic $\gamma$ or a coset $g \sub{\Lambda}$ contains an edge that is dual to a hyperplane $H$, then we say $H$ \emph{crosses} $\gamma$ or $g \sub{\Lambda}$. We say a hyperplane $H$ \emph{separates} two subsets $X$ and $Y$ of $S(\Gamma)$ if $X$ and $Y$ are each entirely contained in different connected components of $S(\Gamma) \smallsetminus H$. 
\end{defn}

\begin{figure}[ht]
\centering
\def\svgscale{1}
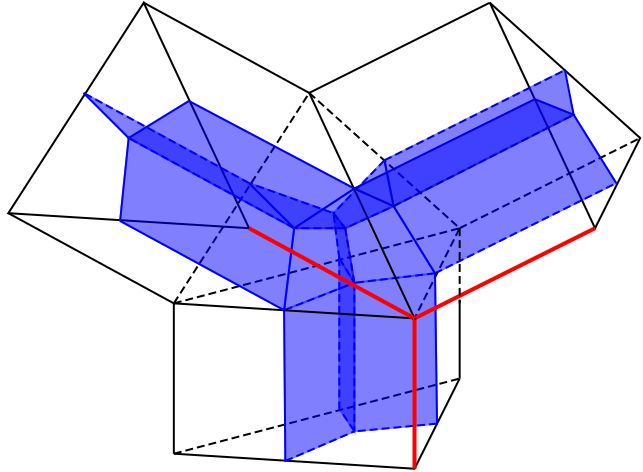 
\caption{A hyperplane (blue) inside its carrier, and an associated combinatorial hyperplane (red).}\label{fig:hyperplane_eg}
\end{figure}

Each hyperplane of a cube complex comes with two corresponding \emph{combinatorial hyperplanes}, obtained by restricting intervals to $-\frac{1}{2}$ or $\frac{1}{2}$ instead of $0$ when constructing mid-cubes. The advantage of these combinatorial hyperplanes is that they form subcomplexes of the cube complex. In $S(\Gamma)$, we obtain combinatorial hyperplanes by restricting a simplex to a vertex instead of performing barycentric subdivision when constructing mid-prisms.

\begin{defn}[Combinatorial hyperplane]\label{defn:comb_hyperplane}
Let $P = T_{1} \times \dots \times T_{m}$ be a prism, where each $T_{i}$ is an $n_{i}$--simplex. Each mid-prism $K_{i}$ splits $P$ into $n_{i}+1$ sectors, each containing a subcomplex $T_{1} \times \dots \times \{v_{k}\} \times \dots \times T_{m}$, where $v_{k}$ is a vertex of $T_{i}$. Given a hyperplane $H$ of $S(\Gamma)$, consider the union of all such subcomplexes obtained from the mid-prisms of $H$. We call each connected component of this union a \emph{combinatorial hyperplane} associated to $H$; see Figure \ref{fig:hyperplane_eg}.
\end{defn}

\begin{remark}[Labelling hyperplanes]\label{rem:labeling_hyperplanes}
Proposition \ref{prop:squares_in_S} tells us that if two edges $E_{1}$ and $E_{2}$ of $S(\Gamma)$ are sides of a common triangle or opposite sides of a square, then they are labelled by elements of the same vertex group. It follows that all edges that a hyperplane $H$ intersects are labelled by elements of the same vertex group $G_{v}$. We therefore label $H$ with the  vertex group $G_v$. Moreover, the edges of the associated combinatorial hyperplanes will then be labelled by elements of $\langle\lk(v)\rangle$. This fact will be exploited repeatedly in our proofs.
\end{remark}

Genevois established that the hyperplanes of $S(\Gamma)$ maintain many of the fundamental properties from the cubical setting.

\begin{prop}[{Properties of hyperplanes; \cite[Section 2]{Gen_Thesis}}]\label{prop:hyperplanes_properties} \quad
\begin{enumerate}
    \item Every hyperplane of $S(\Gamma)$ separates $S(\Gamma)$ into at least two connected components.
    \item If $H$ is a hyperplane of $S(\Gamma)$, then any combinatorial hyperplane for $H$ is convex in $S(\Gamma)$.
    \item If $H$ is a hyperplane of $S(\Gamma)$, then any connected component of $S(\Gamma) \smallsetminus H$ is convex in $S(\Gamma)$.
    \item A continuous path $\gamma$ in $S(\Gamma)$ is a geodesic if and only if $\gamma$ intersects each hyperplane at most once. \label{hyperplanes:geodesic_iff_cross_once}
    \item If two hyperplanes cross, then they are labelled by adjacent vertex groups.\label{hyperplanes:cross_iff_adjacent_labels} 
\end{enumerate}
\end{prop}

\begin{remark}\label{remark:hyperplanes_cross_loops_twice}
Item (\ref{hyperplanes:geodesic_iff_cross_once}) implies that a hyperplane $H$ of $S(\Gamma)$ crosses a geodesic connecting a pair of points $x,y $ if and only if $H$ separates $x$ and $y$. Thus, if $\gamma_1,\dots,\gamma_n$ is a collection of geodesics in $S(\Gamma)$ such that $\gamma_1\cup \dots \cup\gamma_n$ forms a loop and $H$ is a hyperplane that crosses $\gamma_i$, then $H$ must also cross $\gamma_j$ for some $j \neq i$.
\end{remark}

It is important to note that while we still use the terms `hyperplane' and `combinatorial hyperplane' here, they differ from those of cube complexes in a critical way: the complement of a hyperplane $H$ in $S(\Gamma)$ may have more than two connected components, and thus $H$ may have more than two associated combinatorial hyperplanes.

Genevois and Martin use the convexity of  the cosets $g \sub{\Lambda}$ to construct a nearest point projection onto $g\langle\Lambda\rangle$, which we call a \emph{gate map}.
The map and its properties are given below, and will be essential tools throughout this paper.

\begin{prop}[Gate onto graphical subgroups; {\cite[Section 2]{Genevois_Martin}}]\label{prop:gates_to_subgroups}
Let $G_\Gamma$ be a graph product. For all $\Lambda \subseteq \Gamma$ and $g\in G_\Gamma$, there exists a map $\gate_{g\Lambda} \colon G_\Gamma \to g \sub{\Lambda}$ satisfying the following properties.
\begin{enumerate}
    \item \label{gate:distance_non-increasing} For all $k,h \in G_\Gamma$, $\dist_{syl}(\gate_{g\Lambda}(h),\gate_{g\Lambda}(k)) \leq \dist_{syl}(h,k)$.
    \item \label{gate:equivariance} For all $x,h \in G_\Gamma$, $h \cdot \gate_{g\Lambda}(x) = \gate_{hg\Lambda}(hx)$. In particular, $\gate_{g\Lambda}(x) = g \cdot \gate_\Lambda( g^{-1} x)$.
    \item  \label{gate:unique_closest_point} For all $x \in G_\Gamma$, $\gate_{g\Lambda}(x)$ is the unique element of $g\sub{\Lambda}$ such that $\dist_{syl}(x,\gate_{g\Lambda}(x)) = \dist_{syl}(x, g\sub{\Lambda}).$
    \item \label{gate:hyperplanes_separating_image} Any hyperplane in $S(\Gamma)$ that separates $x$ from $\gate_{g\Lambda}(x)$ separates $x$ from $g\sub{\Lambda}$.
    \item \label{gate:hyperplanes_separating_pairs} If $x,y \in G_\Gamma$ and $H$ is a hyperplane in $S(\Gamma)$ separating $\gate_{g\Lambda}(x)$ and $\gate_{g\Lambda}(y)$, then $H$ separates $x$ and $y$, so that $x$ and $\gate_{g\Lambda}(x)$ (resp. $y$ and $\gate_{g\Lambda}(y)$) are contained in the same connected component of  $S(\Gamma) \smallsetminus H$.
\end{enumerate}
\end{prop}

We also obtain a convenient algebraic formulation for the gate map of an element $g$ onto a graphical subgroup $\langle\Lambda\rangle$ by considering the collection of all possible initial subwords of $g$ that are contained in $\langle\Lambda\rangle$.

\begin{defn}[Prefixes and suffixes]
Let $g \in G_{\Gamma}$. If there exist $p, s \in G_\Gamma$ so that $g = p s$ and $\syl{g} = \syl{p}+\syl{s}$,  we call $p$ a \emph{prefix} of $g$ and $s$ a \emph{suffix} of $g$. We shall use $\prefix(g)$ and $\suffix(g)$  to respectively denote the collections of all prefixes and suffixes of $g$.
\end{defn}

\begin{lem}[Algebraic description of the gate map]\label{lem:head_lemma}
For all $\Lambda \subseteq \Gamma$ and $g \in G_\Gamma$, there exists $p \in \prefix(g) \cap \sub{\Lambda}$ so that $\gate_{\Lambda}(g) = p$. Further, $p$ is the element of $\prefix(g) \cap \sub{\Lambda}$ with the largest syllable length.
\end{lem}

\begin{proof}
Since $\prefix(g) \cap \sub{\Lambda}$ is a finite set, there exists $p \in \prefix(g) \cap \sub{\Lambda}$ so that $|p'|_{syl} \leq |p|_{syl}$ for all $p' \in \prefix(g) \cap \sub{\Lambda}$. Let $x = \gate_\Lambda(g)$ and let $s$ be the suffix of $g$ corresponding to $p$. If there exists a non-identity element $y \in \prefix(s) \cap \sub{\Lambda}$, then $py$ would be an element of $\prefix(g) \cap \sub{\Lambda}$ with syllable length strictly larger than $p$. Since this is impossible by choice of $p$, we have $\prefix(s) \cap \sub{\Lambda} = \{e\}$. This implies $\syl{x^{-1}ps} \geq \syl{s}$ since $x^{-1}p \in \sub{\Lambda}$, and we have the following calculation:
 \[\dist_{syl}(x,g) =|x^{-1} ps|_{syl} \geq |s|_{syl} = |p^{-1}g|_{syl} = \dist_{syl}(p,g).\] Since $p \in \sub{\Lambda}$, this implies $x=p$, as $x$ is the unique element of $\sub{\Lambda}$ that minimises the syllable distance of $g$ to $\sub{\Lambda}$ (Proposition \ref{prop:gates_to_subgroups}(\ref{gate:unique_closest_point})).
\end{proof}

\begin{defn}
Denote the element $p$ of $\prefix(g) \cap \langle \Lambda \rangle$ with largest syllable length by $\prefix_{\Lambda}(g)$, and define $\suffix_{\Lambda}(g) = (\prefix_{\Lambda}(g^{-1}))^{-1}$.
\end{defn}

\subsection{Relatively hierarchically hyperbolic groups}\label{section:backgroup_HHS} 

We break the definition of a relative HHG given by Behrstock, Hagen, and Sisto in \cite{BHS_HHSII}  into three parts in order to more clearly organise the structure of our arguments. First we define what we call the \emph{proto-hierarchy structure}, which sets up the defining information (relations and projections) for the HHG structure.  
We then give the more advanced geometric properties that we need to impose for the group to be a relatively hierarchically hyperbolic space.
 We then define a \emph{relatively hierarchically hyperbolic group} to be a group whose Cayley graph is a relative HHS in such a way that the relative HHS structure agrees with the group structure.

\begin{defn}[Proto-hierarchy structure]\label{defn:HHS}
Let $\mc{X}$ be a quasi-geodesic space and $E>0$. An \emph{$E$--proto-hierarchy structure} on $\mc{X}$ is an index set $\mathfrak S$ and a set $\{ C(W) : W\in\mathfrak S\}$ of geodesic spaces $( C (W),\dist_W)$  such that the following axioms are satisfied. \begin{enumerate}
\item\textbf{(Projections.)}\label{axiom:projections} For each $W \in \mf{S}$, there exists a \emph{projection} $\pi_W \colon \mc{X} \rightarrow 2^{C(W)}$ such that for all $x \in \mc{X}$, $\pi_W(x) \neq \emptyset$ and $\diam(\pi_W(x)) \leq E$. Moreover, each $\pi_W$ is $(E, E)$--coarsely
Lipschitz and $C(W) \subseteq \mc{N}_E(\pi_{W}(\mc{X}))$ for all $W \in \mf{S}$.

 \item \textbf{(Nesting.)} \label{axiom:nesting} If $\mathfrak S \neq \emptyset$, then $\mf{S}$ is equipped with a  partial order $\nest$ and contains a unique $\nest$--maximal element. When $V\nest W$, we say $V$ is \emph{nested} in $W$.  For each
 $W\in\mathfrak S$, we denote by $\mathfrak S_W$ the set of all $V\in\mathfrak S$ with $V\nest W$.  Moreover, for all $V,W\in\mathfrak S$ with $V\propnest W$ there is a specified non-empty subset $\rho^V_W\subseteq C(W)$ with $\diam(\rho^V_W)\leq E$.

 \item \textbf{(Orthogonality.)} 
 \label{axiom:orthogonal} $\mathfrak S$ has a symmetric relation called \emph{orthogonality}. If $V$ and $W$ are orthogonal, we write $V\perp
 W$ and require that $V$ and $W$ are not $\nest$--comparable. Further, whenever $V\nest W$ and $W\perp
 U$, we require that $V\perp U$. We denote by $\mf{S}_W^\perp$ the set of all $V\in \mf{S}$ with $V\perp W$.

 \item \textbf{(Transversality.)}
 \label{axiom:transversality} If $V,W\in\mathfrak S$ are not
 orthogonal and neither is nested in the other, then we say $V,W$ are
 \emph{transverse}, denoted $V\trans W$.  Moreover, for all $V,W \in \mathfrak{S}$ with $V\trans W$ there are non-empty
  sets $\rho^V_W\subseteq C(W)$ and
 $\rho^W_V\subseteq C(V)$ each of diameter at most $E$.
\end{enumerate}
We use $\mf{S}$ to denote the entire proto-hierarchy structure, including the index set $\mf{S}$, spaces $\{C(W) : W \in \mf{S}\}$, projections $\{\pi_W : W \in \mf{S}\}$, and relations $\nest$, $\perp$, $\trans$. We call the elements of $\mf{S}$ the \emph{domains} of $\mf{S}$ and call the set $\rho_W^V$ the \emph{relative projection} from $V$ to $W$. The number $E$ is called the \emph{hierarchy constant} for $\mf{S}$.
\end{defn}

\begin{defn}[Relatively hierarchically hyperbolic space]\label{defn:relativ_HHS}
An $E$--proto-hierarchy structure $\mf{S}$ on a quasi-geodesic space $\mc{X}$ is a \emph{relatively $E$--hierarchically hyperbolic space structure} (relative $E$--HHS structure) on $\mc{X}$ if it satisfies the following additional axioms.
\begin{enumerate}

\item \textbf{(Hyperbolicity.)} \label{axiom:hyperbolicity} For each $W \in \mf{S}$, either $W$ is $\nest$--minimal or $C(W)$ is $E$--hyperbolic.

\item \textbf{(Finite complexity.)} \label{axiom:finite_complexity} Any set of pairwise $\nest$--comparable elements has cardinality at most $E$.
 
\item \textbf{(Containers.)} \label{axiom:containers}  For each $W \in \mf{S}$ and $U \in \mf{S}_W$ with $ \mf{S}_W\cap \mf{S}_U^\perp \neq \emptyset$, there exists $Q \in\mf{S}_W$ such that $V \nest Q$ whenever $V \in\mf{S}_W \cap \mf{S}_U^\perp$.  We call $Q$ a \emph{container of $U$ in $W$}.

\item\textbf{(Uniqueness.)} There exists a function 
$\theta \colon [0,\infty) \to [0,\infty)$ so that for all $r \geq 0$, if $x,y\in\mc X$ and
$\dist_\mc{X}(x,y)\geq\theta(r)$, then there exists $W\in\mathfrak S$ such
that $\dist_W(\pi_W(x),\pi_W(y))\geq r$. We call $\theta$ the \emph{uniqueness function of $\mf{S}$}.\label{axiom:uniqueness}

 \item \textbf{(Bounded geodesic image.)} \label{axiom:bounded_geodesic_image} For all $x,y \in \mc{X}$ and $V,W\in\mathfrak S$ with $V\propnest W$, if $\dist_V(\pi_V(x),\pi_V(y))\geq E$, then every $C(W)$--geodesic from $\pi_W(x)$ to $\pi_W(y)$ must intersect the $E$--neighbourhood of $\rho_W^V$.

 \item \textbf{(Large links.)}  \label{axiom:large_link_lemma} 
For all $W\in\mathfrak S$ and  $x,y\in\mc X$, there exists $\{V_1,\dots,V_m\}\subseteq\mathfrak S_W\smallsetminus\{W\}$ such that $m \leq E \dist_{W}(\pi_W(x),\pi_W(y))+E$, and for all $U\in\mathfrak
S_W\smallsetminus\{W\}$, either $U\in\mathfrak S_{V_i}$ for some $i$, or $\dist_{U}(\pi_U(x),\pi_U(y)) \leq E$.  

\item \textbf{(Consistency.)}
 \label{axiom:consistency}  If $V\trans W$, then \[\min\left\{\dist_{
 W}(\pi_W(x),\rho^V_W),\dist_{
 V}(\pi_V(x),\rho^W_V)\right\}\leq E\] for all $x\in\mc X$. Further, if $U\nest V$ and either $V\propnest W$ or $V\trans W$ and $W\not\perp U$, then $\dist_W(\rho^U_W,\rho^V_W)\leq E$.
 
  \item \textbf{(Partial realisation.)} \label{axiom:partial_realisation}  If $\{V_i\}$ is a finite collection of pairwise orthogonal elements of $\mathfrak S$ and $p_i\in  C(V_i)$ for each $i$, then there exists $x\in \mc X$ so that:
 \begin{itemize}
 \item $\dist_{V_i}(\pi_{V_{i}}(x),p_i)\leq E$ for all $i$;
 \item for each $i$ and 
 each $W\in\mathfrak S$, if $V_i\propnest W$ or $W\trans V_i$, we have 
 $\dist_{W}(\pi_W(x),\rho^{V_i}_W)\leq E$.
  \end{itemize}
  \end{enumerate}

If $C(W)$ is $E$--hyperbolic for all $W \in \mf{S}$, then $\mf{S}$ is an \emph{$E$--hierarchically hyperbolic space structure} on $\mc{X}$. We call a quasi-geodesic space $\mc{X}$ a  \emph{(relatively) $E$--hierarchically hyperbolic space} if there exists a (relatively) $E$--hierarchically hyperbolic structure on $\mc{X}$. We use the pair $(\mc{X},\mf{S})$ to denote a (relatively) hierarchically hyperbolic space equipped with the specific (relative) HHS structure $\mf{S}$.
\end{defn}

\begin{defn}[Relatively hierarchically hyperbolic group]\label{defn:hierarchically hyperbolic groups}
  
     Let $G$ be a finitely generated group and let $X$ be the Cayley graph of $G$ with respect to some finite generating set.  We say $G$ is a (relatively) \emph{$E$--hierarchically hyperbolic group} (HHG) if:

   \begin{enumerate}[(1)]
       \item \label{item:HHG_cayleygraph} The space $X$ admits a (relative) $E$--HHS structure $\mf{S}$.
        \item \label{item:HHG_index_action}There is a $\nest$--, $\perp$--, and $\trans$--preserving action of $G$ on $\mf{S}$ by bijections such that $\mf{S}$ contains finitely many $G$--orbits.
        \item \label{item:HHG conditions} For each $W \in \mf{S}$ and $g\in G$, there exists an isometry $g_W \colon C(W) \rightarrow C(gW)$ satisfying the following for all $V,W \in \mf{S}$ and $g,h \in G$.
      \begin{itemize}
            \item The map $(gh)_W \colon C(W) \to C(ghW)$ is equal to the map $g_{hW} \circ h_W \colon C(W) \to C(ghW)$.
            \item For each $x \in X$, $g_W(\pi_W(x))$ and $\pi_{gW}(g \cdot x)$ are at most $E$--far apart in $C(gW)$.
            \item If $V \trans W$ or $V \propnest W$, then $g_W(\rho_W^V)$  and $\rho_{gW}^{gV}$ are at most $E$--far apart in $C(gW)$.
        \end{itemize}
        \end{enumerate}

The structure $\mf{S}$ satisfying (\ref{item:HHG_cayleygraph})--(\ref{item:HHG conditions})  is called a \emph{(relatively) $E$--hierarchically hyperbolic group} (HHG) structure on $G$. We use $(G,\mf{S})$ to denote a group $G$ equipped with a specific (relative) HHG structure $\mf{S}$. 
\end{defn} 

 We build the proto-hierarchy structure for a graph product of finitely generated groups in Section \ref{section:proto} and spend Section \ref{section:rel_HHG} verifying this structure satisfies the axioms of a relatively hierarchically hyperbolic space and respects the group structure.

\section{The proto-hierarchy structure on a graph product}\label{section:proto}
For this section $G_\Gamma$ will be a graph product of finitely generated groups. For each vertex group $G_v$, let $S_v$ be a finite generating set for $G_v$, then define $S$ to be $\bigcup_{v \in V(\Gamma)} S_v$. Throughout this section, $\dist$ will denote the word metric on $G_\Gamma$ with respect to $S$. We now begin to explicitly construct the HHS structure on $G_\Gamma$. We first define the index set, associated spaces, and projection maps in Section \ref{section:index_set} and then define the relations and relative projections in Section \ref{section:relations}.

\subsection{The index set, associated spaces, and projections.}\label{section:index_set}

The index set for our relative HHS structure on $G_\Gamma$ is the set of parallelism classes of graphical subgroups. This mirrors the case of right-angled Artin groups studied in \cite{BHS_HHSI}.

\begin{defn}[Parallelism and the index set for a graph product]\label{defn:index_set}
  Let $G_\Gamma$ be a graph product. For a subgraph $\Lambda \subseteq \Gamma$, we shall use $g \Lambda$ to denote the coset $g\sub{\Lambda}$ for ease of notation. We say $g\Lambda$ and $h\Lambda$ are \emph{parallel} if  $g^{-1}h \in \sub{\st(\Lambda)}$ and write $g\Lambda \parallel h \Lambda$. Let $[g\Lambda]$ denote the equivalence class of $g\Lambda$ under the parallelism relation $\parallel$. Define the index set $\mf{S}_\Gamma = \{ [g\Lambda] : g \in G_\Gamma, \ \Lambda \subseteq \Gamma\}$.
\end{defn}

The geometric intuition for the definition of parallelism comes from the fact that if two cosets $g\langle\Lambda\rangle$ and $h\langle\Lambda\rangle$ satisfy $g^{-1}h \in \langle\st(\Lambda)\rangle$, then they are each crossed by precisely the same set of hyperplanes of $S(\Gamma)$. Again, it is important to note that these hyperplanes, introduced in Definition \ref{defn:hyperplane}, are generalisations of those in cube complexes.

\begin{prop}[Parallel cosets have the same hyperplanes]\label{prop:hyperplanes_cross_parallel}
Let $\Lambda \subseteq \Gamma$ and $g,h \in G_\Gamma$. If $g\langle\Lambda\rangle \parallel h\langle\Lambda\rangle$, then every hyperplane of $S(\Gamma)$ crossing $g\langle\Lambda\rangle$ must also cross $h\langle\Lambda\rangle$.
\end{prop}

\begin{proof}
Since  $g\langle\Lambda\rangle \parallel h\langle\Lambda\rangle$, we have $g^{-1}h \in \langle\st(\Lambda)\rangle$ and there exist $\lambda \in \sub{\Lambda}$ and $l \in \sub{\lk(\Lambda)}$ such that $g^{-1}h= \lambda l$ (Remark \ref{rem:star_elements}). Since $\lambda$ and $l$ commute, $g^{-1}h \sub{\Lambda} = l \sub{\Lambda}$.

Let $H$ be a hyperplane in $S(\Gamma)$ crossing $g\langle\Lambda\rangle$. In particular, $H$ separates two adjacent points $ga$ and $gb$ in $g\langle\Lambda\rangle$. Translating by $g^{-1}$, we have that $g^{-1}H$ separates $a$ and $b$ in $\langle\Lambda\rangle$. Let $s_1 \dots s_n$ be a reduced syllable expression for $l$. Thus, there is a geodesic from $a$ to $l a $ and a geodesic from $b$ to $l b$ each labelled by $s_1\dots s_n$ where each $s_i \in \sub{\lk(\Lambda)}$. Since $b^{-1}a$ labels an edge of $\sub{\Lambda}$, $b^{-1}a$ and $s_i$ span a square for each $i \in \{1,\dots,n\}$. Thus we have a strip of squares joining the edge between $a$ and $b$ to the edge between $l a$ and $l b$ with the hyperplane $g^{-1} H$ running through the middle.  Hence $g^{-1} H$ crosses $l \sub{\Lambda} = g^{-1}h \sub{\Lambda}$ and by translating by $g$, $H$ crosses $h \sub{\Lambda}$.
\end{proof}

The  hierarchy structure on a graph product on $n$ vertices can be thought of as being built up in $n$ levels, with level $k$ consisting of the subgraphs with $k$ vertices. Whenever we build up to the next level in the hierarchy, we need to record precisely the geometry we have just added; any less will violate the uniqueness axiom, while any more  may violate hyperbolicity. When defining our spaces $C(g\Lambda)$, we therefore do not want to record any distance travelled in strict subgraphs of $\Lambda$. This leads us to the \emph{subgraph metric}:

\begin{defn}[Subgraph metric on a graph product]\label{defn:subgraph_metric}
Let $G_\Gamma$ be a graph product. Define $C(\Gamma)$ to be the graph whose vertices are elements of $G_\Gamma$ and where $g,h \in G_\Gamma$ are joined by an edge if there exists a proper subgraph $\Lambda \subsetneq \Gamma$ such that $g^{-1}h \in \sub{\Lambda}$, or if $g^{-1}h$ is an element of the generating set $S$ defined at the beginning of the section. We denote the distance in $C(\Gamma)$ by $\dist_\Gamma(\cdot,\cdot)$ and say $\dist_\Gamma(g,h)$ is the \emph{subgraph distance} between $g$ and $h$.  When $\Gamma$ is a single vertex $v$,  $C(\Gamma) = C(v)$ is the Cayley graph of the vertex group $G_v$ with respect to the finite generating set $S$. Otherwise, $\dist_\Gamma(e,g)$ is equal to the smallest $n$ such that  $g=\lambda_1 \dots \lambda_n $ with $\supp(\lambda_i)$ a proper subgraph of $\Gamma$ for each $i \in \{1,\dots,n\}$.
  
If $g = \lambda_1 \dots \lambda_n$ where $\supp(\lambda_i)$ is a proper subgraph of $\Gamma$ for each $i \in \{1,\dots,n\}$, then we call $\lambda_1 \dots \lambda_n$ a \emph{subgraph expression} for $g$. If $n = \dist_\Gamma(e,g)$, then  $\lambda_1 \dots \lambda_n$ is a \emph{reduced subgraph expression} for $g$. Note that when $\Gamma$ is a single vertex, there are no subgraph expressions.
\end{defn}

\begin{remark}
When $\Gamma$ has at least $2$ vertices, $S(\Gamma)$ is obtained from Cay($G_{\Gamma},S$)  by adding extra edges, where $S$ is the generating set defined at the beginning of the section. Likewise $C(\Gamma)$ is then obtained from $S(\Gamma)$ by adding even more edges. It therefore follows that $\dist_{\Gamma} \leq \dist_{syl} \leq \dist$, where $\dist$ is the word metric on $G_{\Gamma}$ induced by $S$. 
\end{remark}

In a reduced subgraph expression $g=\lambda_1 \dots \lambda_n$ we may assume   $\suffix_{\Lambda_{i+1}}(\lambda_{1}\dots\lambda_{i}) = e$ for each $i \in \{1,\dots,n-1\}$ by removing  any non-trivial suffix from the end of $\lambda_{1}\dots\lambda_{i}$ and attaching it to the beginning of $\lambda_{i+1}$. By repeating this procedure for each $i$ in ascending order and then writing reduced syllable expressions for each $\lambda_{i}$, we then obtain a reduced syllable expression for $g$.

\begin{lem}\label{lem:normal_subgraph_form}
If $\Gamma$ contains at least 2 vertices, then for each  $g \in G_\Gamma$, there exist $\lambda_1,\dots,\lambda_n \in G_{\Gamma}$ with $\supp(\lambda_i) = \Lambda_i \subsetneq \Gamma$ such that the following hold.
\begin{enumerate}
    \item $\lambda_1\dots\lambda_n$ is a reduced subgraph expression for $g$. \label{item:reduced_subgraph_expression}
    \item For each $i \in \{1,\dots, n-1\}$, $\suffix_{\Lambda_{i+1}}(\lambda_1\dots\lambda_i) = e$. \label{item:trivial_suffix}
    \item $|g|_{syl} = |\lambda_1\dots \lambda_n|_{syl}= \sum_{j=1}^{n} |\lambda_j|_{syl}$. \label{item:subgraph+syllable_expression}
\end{enumerate}
In particular, for each $x,y \in G_\Gamma$, there exists an $S(\Gamma)$--geodesic $\gamma$ connecting $x$ and $y$  such that if $\lambda_1\dots \lambda_n$ is the above reduced subgraph expression for $x^{-1}y$, then the element $x\lambda_1\dots \lambda_i$ is a vertex of $\gamma$ for each $i \in \{1,\dots,n\}$.
\end{lem}

\begin{proof}
We begin by noting how the final conclusion of the lemma follows from the main conclusion. Let  $\lambda_1\dots\lambda_n$ be a reduced subgraph expression for $x^{-1}y$ that satisfies (\ref{item:subgraph+syllable_expression}). For each $i \in \{1,\dots,n\}$, let $s_1^i\dots s_{m_i}^i$ be a reduced syllable expression for $\lambda_i$. Since $|x^{-1}y|_{syl}=|\lambda_1\dots \lambda_n|_{syl} = \sum_{j=1}^{n} |\lambda_j|_{syl}$, it follows that $(s_1^1 \dots s_{m_1}^1) \dots (s_1^n\dots s_{m_n}^n)$ is a reduced syllable expression for $x^{-1}y$. Hence, there exists an $S(\Gamma)$--geodesic $\eta$ from $e$ to $x^{-1}y$ whose edges are labelled by $(s_1^1 \dots s_{m_1}^1) \dots (s_1^n\dots s_{m_n}^n)$, and this implies the element $\lambda_1\dots \lambda_i$ appears as a vertex of $\eta$  for each $i \in \{1,\dots,n\}$. Translating by $x$ gives $\gamma = x\eta$ as the desired geodesic.

We now prove we can find a reduced subgraph expression satisfying (\ref{item:trivial_suffix}) and (\ref{item:subgraph+syllable_expression}) for any element of $G_\Gamma$. Our proof proceeds by induction on $n = \dist_{\Gamma}(e,g)$. If $n = 1$, then $\supp(g)$ is a proper subgraph of $\Gamma$ and the conclusion is  trivially true.

Assume the lemma holds for all $h \in G_\Gamma$ with $\dist_\Gamma(e,h) \leq n-1$ and let $g \in G_\Gamma$ with $\dist_{\Gamma}(e,g)=n$.   Let $\omega_1\dots \omega_n$ be a reduced subgraph expression for $g$. Let $\Omega_i = \supp(\omega_i)$ for each $i \in \{1,\dots,n\}$.  By the induction hypothesis, we can assume $g_{0}=\omega_1\dots\omega_{n-1}$ satisfies the conclusion of the lemma. Hence,  $|\omega_1\dots \omega_{n-1}|_{syl}= \sum_{j=1}^{n-1} |\omega_j|_{syl}$ and $\suffix_{\Omega_{i+1}}(\omega_1\dots\omega_i) = e$  for  $i \in \{1,\dots, n-2\}$. 

Let $\sigma = \suffix_{\Omega_{n}}(\omega_1\dots\omega_{n-1})$. For each $i \in \{1,\dots,n-1\}$, let $s_1^i\dots s_{m_i}^i$ be a reduced syllable expression for $\omega_i$.  Now,  $(s_1^1 \dots s_{m_1}^1) \dots (s_1^{n}\dots s_{m_{n-1}}^{n-1})$ is a reduced syllable expression for $\omega_1\dots \omega_{n-1}$ as  $|\omega_1\dots \omega_{n-1}|_{syl}= \sum_{j=1}^{{n-1}} |\omega_j|_{syl}$. 
Thus, each syllable of $\sigma$ is a syllable of one of $\omega_1, \dots,\omega_{n-1}$. 
For each $i \in \{1,\dots, n-1\}$, let $j_1 < \dots < j_i$ be the elements of $\{1,\dots, m_i\}$ such that $s_{j_1}^i, \dots, s_{j_i}^i$ are the syllables of $\omega_i$ that are {not} syllables of $\sigma$. For $i \in \{1,\dots, n-1\}$, let $\omega'_i = s^i_{j_1} \dots s^i_{j_i}$. Thus, we have $\omega_1\dots\omega_{n-1} = \omega'_1\dots\omega'_{n-1} \sigma$ where $\suffix_{\Omega_n}(\omega'_1\dots\omega'_{n-1}) = e$.

Let $\omega'_n = \sigma \omega_n$. Then $\omega'_1\dots \omega'_{n-1}\omega'_n$ is a reduced subgraph expression for $g$ with $\supp(\omega'_n) = \Omega_n$ and $\suffix_{\Omega_n}(\omega'_1\dots\omega'_{n-1}) = e$. Let $g' = \omega'_1\dots\omega'_{n-1}$. 
Since $\omega'_1\dots \omega'_n$ is a reduced subgraph expression for $g$, then $\omega'_1\dots\omega'_{n-1}$ is a reduced subgraph expression for $g'$. 
Hence, $\dist_\Gamma(e,g') = n-1$ and the induction hypothesis says that there exists a reduced subgraph expression $\lambda_1\dots\lambda_{n-1}$ for $g'$ such that $\suffix_{\supp(\lambda_{i+1})}(\lambda_1\dots\lambda_{i}) = e$ for $i \in \{1,\dots, n-2\}$ and $|\lambda_1\dots\lambda_{n-1}|_{syl} = \sum_{j=1}^{n-1} |\lambda_j|_{syl}$. 
Further, $\suffix_{\Omega_n}(\lambda_1\dots\lambda_{n-1}) = e$ as $\lambda_1\dots\lambda_{n-1} = g' = \omega'_1\dots\omega'_{n-1}$.

Now let $\lambda_n = \omega'_n$ and $\Lambda_i = \supp(\lambda_i)$ for each $i \in \{1,\dots,n\}$. We verify that $\lambda_1,\dots,\lambda_n$ satisfies the conclusion of the lemma for $g$.
\begin{enumerate}
    \item $\lambda_1\dots \lambda_{n}$ is a reduced subgraph expression for $g$ as each $\Lambda_i = \supp(\lambda_i)$ is a proper subgraph of $\Gamma$ and $\dist_{\Gamma}(e,g) = n$.
    \item For each $i \in \{1,\dots,n-1\}$, the above shows $\suffix_{\Lambda_{i+1}}(\lambda_1\dots\lambda_{i}) = e$.
    \item We prove that writing each $\lambda_i$ in a reduced syllable form produces a reduced syllable form for the product $\lambda_1 \dots \lambda_n$. For each $i \in \{1,\dots,n\}$, let $t_1^i \dots t_{k_i}^i$ be a reduced syllable expression for $\lambda_i$. Since  $|\lambda_1\dots\lambda_{n-1}|_{syl} = \sum_{j=1}^{n-1} |\lambda_j|_{syl}$, we know $(t_1^1 \dots t_{k_1}^1) \dots (t_1^{n-1} \dots t_{k_{n-1}}^{n-1})$ is a reduced syllable expression for $\lambda_1\dots\lambda_{n-1}$. Therefore, if  $$(t_1^1 \dots t_{k_1}^1) \dots (t_1^{n}\dots t_{k_{n}}^{n})$$ is not a reduced syllable expression for $\lambda_1 \dots \lambda_n$, then Theorem \ref{thm:graph_product_normal_form} implies there must exist syllables $t_j^i$ of $\lambda_1\dots\lambda_{n-1}$ and $t_\ell^n$ of $\lambda_n$ such that $\supp(t_j^i) = \supp(t_\ell^n)$ and $t_j^i$ can be moved to be adjacent to $t_\ell^n$ using a number of commutation relations. However, this implies $t_j^i$ is a suffix for $ \lambda_1 \dots \lambda_{n-1}$ with support in $\Lambda_n$. This is impossible as $\suffix_{\Lambda_n}(\lambda_1\dots \lambda_{n-1}) = e$. Therefore, $(t_1^1 \dots t_{k_1}^1) \dots (t_1^{n}\dots t_{k_{n}}^{n})$ must be a reduced syllable expression for $\lambda_1\dots\lambda_{n}$ and hence $|\lambda_1\dots\lambda_{n}|_{syl} = |\lambda_1|_{syl} + \dots + |\lambda_{n}|_{syl}$ as desired.\qedhere
\end{enumerate}
\end{proof}

We can now define the geodesic spaces associated to elements of the index set. In the next section, we will show that they are hyperbolic.

\begin{defn}
 Let $G_\Gamma$ be a graph product. For each $g \in G_\Gamma$ and $\Lambda \subseteq \Gamma$, let $C(g\Lambda)$ denote the graph whose vertices are elements of the coset $g\sub{\Lambda}$ and where $gx$ and $gy$ are joined by an edge if $x$ and $y$ are joined by an edge in $C(\Lambda)$. The metric on $C(g\Lambda)$ is denoted $\dist_{g\Lambda}(\cdot,\cdot)$.
\end{defn}

\begin{remark}\label{rem:join_implies_bounded_diameter}
If $\Lambda \subseteq \Gamma$ is a join $\Lambda = \Lambda_1 \bowtie \Lambda_2$, then every element $\lambda \in \sub{\Lambda}$ can be written as $\lambda = \lambda_1 \lambda_2$ where $\lambda_1 \in \sub{\Lambda_1}$ and $\lambda_2 \in \sub{\Lambda_2}$. Since $\Lambda_1$ and $\Lambda_2$ are proper subgraphs of $\Lambda$, this implies $C(\Lambda)$, and therefore $C(g\Lambda)$, has diameter at most $2$ whenever $\Lambda$ splits as a join.
\end{remark}

We now wish to use our gate map from Proposition \ref{prop:gates_to_subgroups} to define projections for our hierarchy structure. Since $\mf{S}_\Gamma$ is the set of parallelism classes of cosets of graphical subgroups, we must verify that the gate map is  well-behaved  under parallelism.

\begin{lem}[{Gates to parallelism classes are well defined}]\label{lem:C_spaces_are_well_defined}
If $g\Lambda \parallel h\Lambda$, then for all $x \in G_\Gamma$, $\gate_{h\Lambda}(x) = \gate_{h\Lambda} \circ \gate_{g\Lambda}(x)$. In particular, if $g\Lambda \parallel h\Lambda$, then $\gate_{h\Lambda}\vert_{g\sub{\Lambda}} \colon g\sub{\Lambda} \to h\sub{\Lambda}$ agrees with the isometry of $S(\Gamma)$ induced by the element $hpg^{-1}$, where $p = \prefix_{\Lambda}(h^{-1}g)$.
\end{lem}

\begin{proof}
Suppose that $\gate_{h\Lambda} (x) \neq \gate_{h\Lambda}(\gate_{g\Lambda}(x))$. There must then exist a hyperplane $H$ separating $\gate_{h\Lambda} (x)$ and $\gate_{h\Lambda}(\gate_{g\Lambda}(x))$ in $S(\Gamma)$.  By (\ref{gate:hyperplanes_separating_image}) and (\ref{gate:hyperplanes_separating_pairs}) of  Proposition \ref{prop:gates_to_subgroups}, $H$ separates $x$ and $\gate_{g\Lambda}(x)$ and thus cannot cross $g\sub{\Lambda}$. However, $H$ crosses $h\langle\Lambda\rangle$, and so must
cross $g\sub{\Lambda}$  by Proposition \ref{prop:hyperplanes_cross_parallel}. As this is a contradiction, we must have that $\gate_{h\Lambda} (x) = \gate_{h\Lambda}(\gate_{g\Lambda}(x))$.

Note, if $g\lambda \in g\langle\Lambda\rangle$, then  equivariance (Proposition \ref{prop:gates_to_subgroups}(\ref{gate:equivariance})) plus the prefix description of the gate map (Lemma \ref{lem:head_lemma}) imply
\[\gate_{h\Lambda}(g\lambda) = h \cdot \gate_{\Lambda}(h^{-1}g\lambda) = h \cdot \prefix_{\Lambda}(h^{-1}g\lambda).\]
Since $h^{-1}g \in \langle\st(\Lambda)\rangle$, we can write $h^{-1}g = pl$, where $p \in \langle\Lambda\rangle$ and $l \in \langle\lk(\Lambda)\rangle$. Therefore $\gate_{h\Lambda}(g\lambda) = h\cdot\prefix_{\Lambda}(pl\lambda) = hp\lambda$, that is, $\gate_{h\Lambda}\vert_{g\sub{\Lambda}}$ agrees with the isometry induced by $hpg^{-1}$.
\end{proof}

Since Cay($G_{\Gamma},S$), $S(\Gamma)$ and $C(\Gamma)$ differ only in that the latter two have extra edges, we can easily promote our gate map to a projection map.

\begin{defn}\label{defn:projections}
  For all $\Lambda \subseteq \Gamma$ and $g \in G_\Gamma$, define $\pi_{g\Lambda} \colon G_\Gamma \to C(g\Lambda)$ by $i_{g\Lambda} \circ \gate_{g\Lambda}$ where $i_{g\Lambda}$ is the inclusion map from $g\sub{\Lambda}$ into $C(g\Lambda)$. 
\end{defn}

\begin{remark}\label{rem:prefix_description_of_projection}
Combining the prefix description of the gate map (Lemma \ref{lem:head_lemma})  with equivariance (Proposition \ref{prop:gates_to_subgroups}.(\ref{gate:equivariance})), we have that $\gate_{g\Lambda}(x) = g \cdot \prefix_\Lambda(g^{-1}x)$ for all $x \in G_\Gamma$. Since the only difference between $\pi_{g\Lambda}$ and $\gate_{g\Lambda}$ is the metric on the image, this means $\pi_{g\Lambda}(x) =g \cdot \prefix_\Lambda(g^{-1}x)$ as well. 
\end{remark}

Note that any coset of $\langle\Lambda\rangle$ can be expressed in the form $g\langle\Lambda\rangle$ where $\suffix_{\Lambda}(g) = e$ (and thus $\prefix_{\Lambda}(g^{-1})=e$). Indeed, let $h\langle\Lambda\rangle$ be a coset of $\sub{\Lambda}$, and suppose $\suffix_{\Lambda}(h) = \lambda$. Then we can write $h = g\lambda$, where $\suffix_{\Lambda}(g) = e$. It therefore follows that $h\langle\Lambda\rangle = g\lambda\langle\Lambda\rangle = g\langle\Lambda\rangle$. 
The next proposition shows that choosing the representative of $g \sub{\Lambda}$ in this way ensures that $\prefix_{\Lambda}(g^{-1}x)$ contains only syllables of $x$. This is particularly helpful when considering the prefix description of $\pi_{g\Lambda}(x)$. 

\begin{prop}\label{prop:syllables_of_gate_ are_syllables_of_image}
Let $\Lambda \subseteq \Gamma$ and let $g\in G_\Gamma$. Then for all $x,y\in G_\Gamma$, every syllable of $(\gate_{g\Lambda}(x))^{-1} \cdot \gate_{g\Lambda}(y)$ is a syllable of $x^{-1}y$. In particular, if $g$ is the representative of $g \sub{\Lambda}$ with $\suffix_\Lambda(g) = e$ and $h \in G_\Gamma$, then every syllable of $\prefix_\Lambda(g^{-1}h) = \gate_{\Lambda}(g^{-1}h)$ is a syllable of $h$.
\end{prop}

\begin{proof}
Let $x,y \in G_\Gamma$, then let $p_x = \gate_{g\Lambda}(x)$ and $p_y = \gate_{g\Lambda}(y)$.  Let $\eta$ be an $S(\Gamma)$--geodesic connecting $p_x$ and $p_y$ and let $\gamma$ be an $S(\Gamma)$--geodesic connecting $x$ and $y$.  Let $s_1,\dots,s_n$ be the elements of the vertex groups of $G_\Gamma$ that label the edges of $\eta$. This means $s_1,\dots,s_n$ are the syllables of $p_x^{-1}p_y$. For each $i \in \{1,\dots,n\}$, let $H_i$ be the hyperplane  dual to the edge of $\eta$ that is labelled by $s_i$ and let $v_{i}$ be the vertex of $\Gamma$ such that $s_i \in G_{v_{i}}$.

Since each $H_i$ separates $\gate_{g\Lambda}(x)$ and $\gate_{g\Lambda}(y)$, each $H_i$ must also cross $\gamma$  by Proposition \ref{prop:hyperplanes_properties}(\ref{hyperplanes:geodesic_iff_cross_once}) and Proposition \ref{prop:gates_to_subgroups}(\ref{gate:hyperplanes_separating_pairs}). For $i \in \{1,\dots,n\}$, let $E_i$ be the edge of $\gamma$ dual to $H_i$.  Note, every edge dual to $H_i$ is labelled by an element of the vertex group $G_{v_i}$, but not necessarily by the same element of $G_{v_i}$. 

If $E_i$  is not labelled by $s_i \in G_{v_i}$, then the hyperplane $H_i$ must encounter a triangle of $S(\Gamma)$ between $\eta$ and $\gamma$. This creates a branch of the hyperplane $H_i$ that cannot cross either $\eta$ or $\gamma$ by Proposition \ref{prop:hyperplanes_properties}(\ref{hyperplanes:geodesic_iff_cross_once}). Thus, this branch must cross either an $S(\Gamma)$--geodesic connecting $x$ and $p_x$ or an $S(\Gamma)$--geodesic connecting $y$ and $p_y$; see Figure \ref{fig:branching_hyperplane}. 
Without loss of generality, assume $H_i$ crosses an $S(\Gamma)$--geodesic connecting $x$ and $p_x = \gate_{g\Lambda}(x)$.
\begin{figure}[ht]
     \centering
     \def\svgscale{.7}
     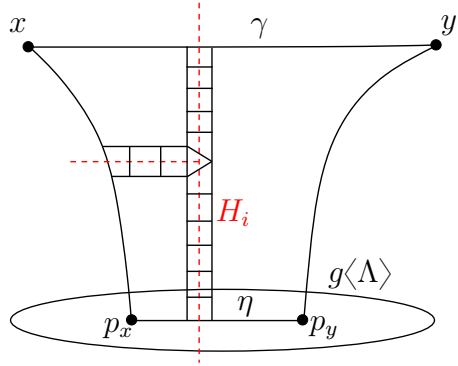
     \caption{If the the hyperplane $H_i$ encounters a triangle of $S(\Gamma)$ between $\eta$ and $\gamma$, then a branch of $H_i$ must cross an $S(\Gamma)$--geodesic from $x$ to $p_x$ (shown) or from $y$ to $p_y$. }
     \label{fig:branching_hyperplane}
 \end{figure}
This means $H_i$ separates $x$ from $\gate_{g\Lambda}(x)$, and thus $H_i$ must separate $x$ from all of $g \sub{\Lambda}$ (Proposition \ref{prop:hyperplanes_properties}(\ref{gate:hyperplanes_separating_image})). However, this is impossible as $H_i$ crosses $g\sub{\Lambda}$. Therefore $H_i$ cannot encounter a triangle between $\eta$ and $\gamma$, and $E_i$ must therefore be labelled by the element $s_i$. Since the elements labelling the edges of $\gamma$ are the syllables of $x^{-1}y$, this implies every syllable of $p_x^{-1} p_y$ is also a syllable of $x^{-1} y$.

For the final clause of the proposition, note that $\suffix_\Lambda(g) = e$ implies $\gate_\Lambda(g^{-1}) = \prefix_{\Lambda}(g^{-1}) = e$. Thus, we can apply the above with $x = g^{-1}$ and $y= g^{-1}h$ to conclude that every syllable of $(\gate_{\Lambda}(g^{-1}))^{-1}\gate_{\Lambda}(g^{-1}h) = \gate_{\Lambda}(g^{-1}h)$ is also a syllable of $(g^{-1})^{-1} g^{-1} h = h$.
\end{proof} 

Given $h,k \in G_{\Gamma}$, we shall employ a common abuse of notation by using $\dist_{g\Lambda}(h,k)$ to denote $\dist_{g\Lambda}(\pi_{g\Lambda}(h),\pi_{g\Lambda}(k))$. We can now prove our first HHS axiom.

\begin{lem}[Projections]\label{lem:projections}
For each  $g \in G_\Gamma$ and $\Lambda \subseteq \Gamma$, the projection $\pi_{g\Lambda}$ is $(1,0)$--coarsely Lipschitz.
\end{lem}

\begin{proof}
We want to show that $\dist_{g\Lambda}(x,y) \leq \dist(x,y)$ for all $x,y \in G_{\Gamma}$. First assume $ \Lambda$ consists of a single vertex $v$.  Let $p_x$ and $p_y$ be $ \gate_{g\Lambda}(x) = \pi_{g\Lambda}(x)$ and  $\gate_{g\Lambda}(y) = \pi_{g\Lambda}(y)$ respectively. Since $\Lambda$ is the single vertex $v$, $C(\Lambda)$ is the Cayley graph of $G_v$ with respect to our fixed finite generating set, and $C(g\Lambda)$ is a coset of $C(\Lambda)$. Thus, it suffices to prove  $|p_x^{-1}p_y| $ is bounded above by $|x^{-1}y|$, where $| \cdot |$ is the word length on $G_\Gamma$ with respect to the generating set $S$ defined at the beginning of the section.

Let $s = p_x^{-1}p_y \in G_{v}$. By Proposition \ref{prop:syllables_of_gate_ are_syllables_of_image}, $s$ must be a syllable of $x^{-1}y$, that is, $s$ appears in a reduced syllable expression for $x^{-1}y$. Recall, if $s_1 \dots s_n$ is a reduced syllable expression for $x^{-1}y$, then $|x^{-1}y| = \sum_{i=1}^n |s_i|$ (Corollary \ref{cor:reduced_syllable_expressions_minimize_word_length}). Thus $|x^{-1}y| \geq |s| = |p_x^{-1}p_y|$.

Now assume $\Lambda$ contains at least $2$ vertices. By Proposition \ref{prop:gates_to_subgroups}(\ref{gate:distance_non-increasing}), we have \[\dist_{syl}(\gate_{g\Lambda}(x),\gate_{g\Lambda}(y)) \leq \dist_{syl}(x,y) \leq \dist(x,y).\] Furthermore, 
$C(g\Lambda)$ is obtained from $S(g\Lambda)$ by adding edges as $\Lambda$ contains at least two vertices. Thus we have 
\[\dist_{g\Lambda}(x,y) \leq \dist_{syl}(\gate_{g\Lambda}(x),\gate_{g\Lambda}(y)) \leq \dist_{syl}(x,y) \leq \dist(x,y). \qedhere\]
\end{proof}

Given an $S(\Gamma)$--geodesic $\gamma$, there is a natural order on its vertices which arises from orienting $\gamma$. The distances between the vertices of $\gamma$ under the projection $\pi_{g\Lambda}$ then satisfy the following monotonicity property with respect to this order.

\begin{lem}[Subgraph distance along $S(\Gamma)$--geodesics]\label{lem:subgraph_distance_along_geodesics}
Let $\gamma$ be an $S(\Gamma)$--geodesic connecting two elements $x,y \in G_\Gamma$. For each vertex $q$ of $\gamma$, each element $g \in G_{\Gamma}$, and each subgraph $\Lambda \subseteq \Gamma$, we have \[\dist_{g\Lambda}(x,q) \leq \dist_{g\Lambda}(x,y) \,\,\text{ and }\,\,\, \dist_{g\Lambda}(q,y) \leq \dist_{g\Lambda}(x,y).\]
\end{lem}

\begin{proof}
Fix $g \in G_\Gamma$ and a subgraph $\Lambda \subseteq \Gamma$. Let $p_x = \gate_{g\Lambda}(x)$, $p_y = \gate_{g\Lambda}(y)$, and $p_q = \gate_{g\Lambda}(q)$.

First suppose $\Lambda$ consists of a single vertex of $\Gamma$. Then the $S(\Gamma)$--diameter of $g \sub{\Lambda}$ is $1$ and there exists a single hyperplane $H$ so that every edge of $g\sub{\Lambda}$ is dual to $H$.  If $p_q \neq p_x$ and $p_q \neq p_y$, then $H$ must separate $p_q$ from both $p_x$ and $p_y$. Therefore, $H$ must  cross $\gamma$ between $x$ and $q$ and again between $q$ and $y$ by Proposition \ref{prop:gates_to_subgroups}(\ref{gate:hyperplanes_separating_pairs}). However, this is impossible as $H$ cannot cross $\gamma$ twice (Proposition \ref{prop:hyperplanes_properties}(\ref{hyperplanes:geodesic_iff_cross_once})). Thus we must have either $p_q= p_x$ or $p_q = p_y$. The conclusion of the lemma then automatically holds as $\pi_{g\Lambda}(q) = \pi_{g\Lambda}(x)$ or $\pi_{g\Lambda}(q) = \pi_{g\Lambda}(y)$.

Now assume $\Lambda$ has at least two vertices and $p_q \neq p_x$ and $p_q \neq p_y$. Let $\lambda_1\dots\lambda_m$ be a reduced subgraph expression for $p_{x}^{-1}p_{y}$ of the form provided by Lemma \ref{lem:normal_subgraph_form}. In particular, there exists an $S(\Gamma)$--geodesic $\eta$ connecting $p_x$ and $p_y$ whose vertices include $p_x\lambda_1\dots\lambda_i$  for each $i \in \{1,\dots,m\}$.

Let $\alpha$ and $\beta$ be $S(\Gamma)$--geodesics connecting $p_x$ to $p_q$ and $p_q$ to $p_y$ respectively. Any hyperplane that crosses $\alpha$ must also cross $\gamma$ and separate $x$ and $q$ by Proposition \ref{prop:gates_to_subgroups}(\ref{gate:hyperplanes_separating_pairs}). Similarly, any hyperplane that crosses $\beta$ must also cross $\gamma$ and separate $y$ and $q$. Thus, a hyperplane that crosses both $\alpha$ and $\beta$ would cross the $S(\Gamma)$--geodesic $\gamma$ twice. Since no hyperplane of $S(\Gamma)$ can cross the same geodesic twice (Proposition \ref{prop:hyperplanes_properties}(\ref{hyperplanes:geodesic_iff_cross_once})), it follows that any hyperplane that crosses $\alpha$ (resp. $\beta$) cannot cross $\beta$ (resp. $\alpha$). By Remark \ref{remark:hyperplanes_cross_loops_twice},  any hyperplane  that crosses either $\alpha$ or $\beta$ must therefore cross $\eta$ as $\alpha \cup \beta \cup \eta$ forms a loop in $S(\Gamma)$.

We now prove $\dist_{g\Lambda}(x,q) \leq \dist_{g\Lambda}(x,y)$. The proof for $\dist_{g\Lambda}(q,y) \leq \dist_{g\Lambda}(x,y)$ is nearly identical with $\beta$ replacing $\alpha$.  Let $E_1,\dots,E_k$ be the edges of $\alpha$ and let  $H_j$ be the hyperplane that crosses $E_j$ for $j \in \{1,\dots,k\}$.  We say that two hyperplanes $H_j$ and $H_\ell$ \emph{cross between $\alpha$ and $\eta$} if there exists a vertex $a$ of $\alpha$ such that for each vertex $b$ of $\eta$,  either $H_j$ or $H_\ell$ separates $a$ from $b$; see Figure \ref{fig:cross_between}.
 
 \begin{figure}[ht]
     \centering
     \def\svgscale{.7}
     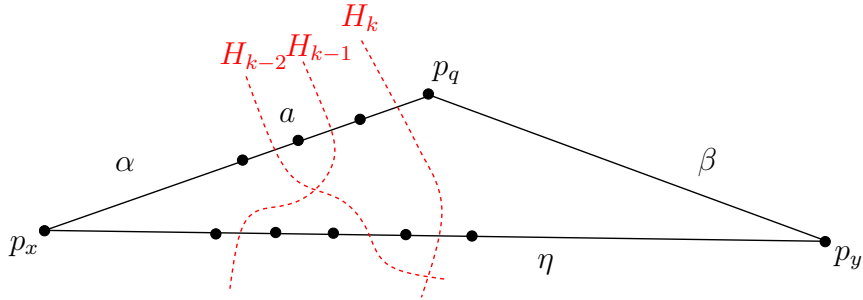
     \caption{The hyperplanes $H_{k-2}$ and $H_{k-1}$ cross between $\alpha$ and $\eta$ because the vertex $a$ is separated from every vertex of $\eta$ by either $H_{k-2}$ or $H_{k-1}$. Even though $H_{k-2}$ and $H_k$ cross, they do not cross \emph{between} $\alpha$ and $\eta$.}
     \label{fig:cross_between}
 \end{figure}
 
 \begin{claim}\label{claim:crossing_reduction}
 There exists an $S(\Gamma)$--geodesic $\alpha'$ that connects $p_x$ and $p_q$ such that no two of $H_1,\dots,H_k$ cross between $\alpha'$ and $\eta$.
 \end{claim}
 
 \begin{proof}
 Let $\alpha_1 = \alpha$ and let $K_i$ be the number of times two of $H_1,\dots,H_k$ cross between $\alpha_i$ and $\eta$.  Note, $K_1 \leq \frac{k(k-1)}{2}$. If $K_1 =0$ we are done. Otherwise, there exists $j \in \{1,\dots,k\}$ such that $H_j$ is the first hyperplane where $H_{j-1}$ and $H_{j}$ cross between $\alpha_1$ and $\eta$. Since $H_{j-1}$ and $H_j$ cross, Proposition \ref{prop:hyperplanes_properties}(\ref{hyperplanes:cross_iff_adjacent_labels}) tells us the edges $E_{j-1}$ and $E_j$ are labelled by elements of adjacent vertex groups. By Proposition \ref{prop:squares_in_S}, $E_{j-1}$ and $E_j$ are two sides of a square $S$ of $S(\Gamma)$ inside which $H_{j-1}$ and $H_j$ cross.  Let $\alpha_2$ be the $S(\Gamma)$--geodesic obtained from $\alpha_1$ by replacing the edges $E_{j-1}$ and $E_j$ with the other two sides of the square $S$; see Figure \ref{fig:flip_over_square}.
 
  \begin{figure}[ht]
     \centering
     \def\svgscale{.7}
     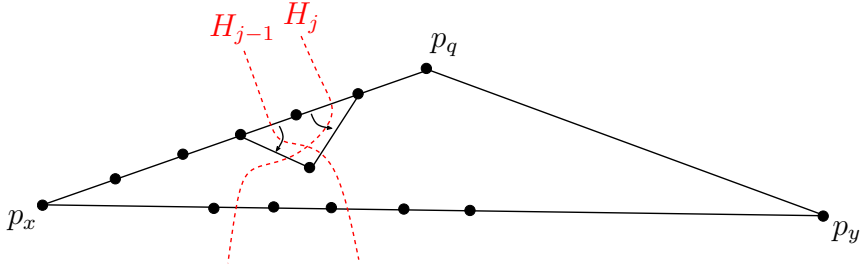
     \caption{The edges $E_{j-1}$ and $E_j$ can be replaced with the other two edges of the square $S$ to obtain a new $S(\Gamma)$--geodesic with $K_2 = K_1 -1$.}
     \label{fig:flip_over_square}
 \end{figure}

Since $H_{j-1}$ and $H_j$ crossed between $\alpha_1$ and $\eta$, we now have $K_2 = K_1 -1$, that is, that the number of times two of $H_1,\dots,H_k$ cross between $\alpha_2$ and $\eta$ is one less than the number of times two of $H_1,\dots,H_k$ crossed between $\alpha_1$ and $\eta$. Reindex $H_1,\dots, H_k$ such that $H_j$ crosses the $j$th edge of $\alpha_2$.
 
If $K_2 =0$, we are done, with $\alpha' = \alpha_2$. Otherwise, can repeat this argument at most $\frac{k(k-1)}{2}$ times to construct a sequence of geodesics $\alpha_1,\alpha_2,\dots, \alpha_r$ where $K_{i+1} = K_{i}-1$ and $K_r = 0$. Then, $\alpha' = \alpha_r$.
\end{proof}

Let $\alpha'$ be as in Claim \ref{claim:crossing_reduction} and reindex $H_1,\dots,H_k$ so that $H_j$ crosses the $j$th edge of $\alpha'$ for each $j \in \{1,\dots,k\}$. Since $H_{j}$ crosses $\eta$ for each $j \in \{1,\dots, k\}$,  the labels for the edges of $\alpha'$ are a subset of the labels of $\eta$. Further, since no two of $H_1,\dots, H_k$ cross between $\alpha'$ and $\eta$, the order in which the labels of edges appear along $\alpha'$ is the same as the order in which they appear along $\eta$. Since the vertices of $\eta$ include $p_{x}\lambda_{1}\dots\lambda_{i}$ for each $i \in \{1,\dots,m\}$, this implies that we can write $p_{x}^{-1}p_{q} = \lambda'_{1}\dots\lambda'_{m}$, where $\supp(\lambda'_{i}) \subseteq \supp(\lambda_{i})$ for each $i \in \{1,\dots,m\}$. It therefore follows that the $C(g\Lambda)$--distance between $p_x$ and $p_q$ is bounded above by the $C(g\Lambda)$--distance between $p_x$ and $p_y$, and so we have $\dist_{g\Lambda}(x,q) \leq \dist_{g\Lambda}(x,y)$.
\end{proof}

\subsection{The relations}\label{section:relations}
Here we define the nesting, orthogonality,  and transversality relations in the proto-hierarchy structure, and prove they have the desired properties. We tackle the nesting relation first.

\begin{defn}[Nesting]\label{defn:nesting}
  Let $G_\Gamma$ be a graph product and let $\mf{S}_\Gamma$ be the index set of parallelism classes of cosets of graphical subgroups described in Definition \ref{defn:index_set}. We say $[g\Lambda] \nest [h\Omega]$ if $\Lambda \subseteq \Omega$ and there exists $k \in G_\Gamma$ such that $[k\Lambda] = [g\Lambda]$ and $[k\Omega]=[h\Omega]$.   
\end{defn}

\begin{lem}\label{lem:nesting_partial}
The relation $\nest$ is a partial order.
\end{lem}

\begin{proof}
The only property that requires checking is transitivity, that is, if   $[g_{1}\Lambda_{1}] \nest [g_{2}\Lambda_{2}] \nest [g_{3}\Lambda_{3}]$, then $[g_{1}\Lambda_{1}] \nest [g_{3}\Lambda_{3}]$.

 Since $\subseteq$ is transitive, we have $\Lambda_1  \subseteq \Lambda_3$. Furthermore, there exist $a,b \in G_{\Gamma}$ such that $[g_{1}\Lambda_{1}] = [a\Lambda_{1}]$, $[a\Lambda_{2}] = [g_{2}\Lambda_{2}] = [b\Lambda_{2}]$, $[g_{3}\Lambda_{3}] = [b\Lambda_{3}]$, that is, $g_{1}^{-1}a \in \sub{\st(\Lambda_{1})}$, $g_{2}^{-1}a,g_{2}^{-1}b \in \sub{\st(\Lambda_{2})}$, $g_{3}^{-1}b \in \sub{\st(\Lambda_{3})}$.
 Thus $g_{1}^{-1}a = l_{1}\lambda_{1}$, $g_{2}^{-1}a = l_{2}\lambda_{2}$, $g_{2}^{-1}b = l'_{2}\lambda'_{2}$, $g_{3}^{-1}b = l_{3}\lambda_{3}$ where $\lambda_i, \lambda_i' \in \sub{\Lambda_i}$ and $l_i,l_i' \in \sub{\lk(\Lambda_i)}$ for each $i$. Let $c = b(\lambda'_{2})^{-1}\lambda_{2}$. Then $g_{3}^{-1}c = g_{3}^{-1}b(\lambda'_{2})^{-1}\lambda_{2} \in \langle\st(\Lambda_{3})\rangle$ since $\Lambda_{2} \subseteq \Lambda_{3}$. Moreover, since $\lk(\Lambda_{2}) \subseteq \lk(\Lambda_{1})$,
\begin{align*}
g_{1}^{-1}c &= g_{1}^{-1}aa^{-1}g_{2}g_{2}^{-1}bb^{-1}c \\
&= l_{1}\lambda_{1}\lambda_{2}^{-1}l_{2}^{-1}l'_{2}\lambda'_{2}(\lambda'_{2})^{-1}\lambda_{2} \\
&= l_{1}l_{2}^{-1}l'_{2}\lambda_{1} \in \langle\st(\Lambda_{1})\rangle.
\end{align*}
Thus $[g_1 \Lambda_1] = [c \Lambda_1]$ and $[g_3\Lambda_3] = [c\Lambda_3]$, verifying that $[g_{1}\Lambda_{1}] \nest [g_{3}\Lambda_{3}]$.
\end{proof}

\begin{defn}[Upwards relative projection]\label{defn:upward_projection}
If $[g\Lambda] \propnest [h\Omega]$, for any choice of representatives $g\Lambda \in [ g\Lambda]$ and $h\Omega \in [h\Omega]$, define $\rho_{h\Omega}^{g\Lambda} \subseteq C(h\Omega)$ to be  \[\rho_{h\Omega}^{g\Lambda} =\bigcup \limits_{k\Lambda \parallel g\Lambda} \pi_{h\Omega}\bigl( k\sub{\Lambda} \bigr) = \pi_{h\Omega}(g\sub{\st(\Lambda)}).\]
The equality between $\bigcup_{k \Lambda \parallel h \Lambda} \pi_{h\Omega} \bigl( k \sub{\Lambda}\bigr)$ and $\pi_{h\Omega}\left( g\sub{\st(\Lambda)}\right)$ is a consequence of the definition that $k \Lambda \parallel g \Lambda$ if and only if $g^{-1}k \in \sub{\st(\Lambda)}$. Indeed, $g\langle\st(\Lambda)\rangle = gg^{-1}k\langle\st(\Lambda)\rangle = k\langle\st(\Lambda)\rangle \supseteq k\langle\Lambda\rangle$ for all $k\Lambda \parallel g\Lambda$. Conversely, each element of $g\langle\st(\Lambda)\rangle$ can be written as $gl\lambda$ where $l \in \langle\lk(\Lambda)\rangle$ and $\lambda\in\langle\Lambda\rangle$, so that $gl\lambda \in gl\langle\Lambda\rangle$ where $g^{-1}gl = l \in \langle\st(\Lambda)\rangle$ and hence $g\Lambda \parallel gl\Lambda$.
\end{defn}

\begin{lem}[Upwards relative projections have bounded diameter]\label{lem:nesting}
If $[g\Lambda] \propnest [h\Omega]$, then for any choice of representatives $g\Lambda \in [g\Lambda]$ and $h \Omega \in [h\Omega]$, we have $\diam\bigl(\rho_{h\Omega}^{g\Lambda} \bigr) \leq 2$. 
\end{lem}

\begin{proof}
Let $g\Lambda$ and $h\Omega$ be fixed representatives of $[g\Lambda]$ and $[h\Omega]$ respectively. Suppose first that $\Omega$ splits as a join. Then $\diam(C(h\Omega)) = 2$ by Remark \ref{rem:join_implies_bounded_diameter}, and  hence $\diam\bigl(\rho_{h\Omega}^{g\Lambda} \bigr) \leq 2$. For the remainder of the proof we will therefore assume that $\Omega$ does not split as a join. Note that this implies that $\st(\Lambda)\cap\Omega \subsetneq \Omega$. Indeed, suppose $\st(\Lambda)\cap\Omega = \Omega$. Then $\Omega \subseteq \st(\Lambda)$, so either $\Omega \subseteq \Lambda$, $\Omega \subseteq \lk(\Lambda)$, or $\Omega$ splits as a join. The first two cases are impossible as $\Lambda \subsetneq \Omega$, and the last case is ruled out by assumption.

Let $a \in G_\Gamma$ be such that $[a\Lambda] = [g\Lambda]$ and $[a\Omega] = [h\Omega]$.  Since $[a\Lambda] = [g\Lambda]$, we have $g^{-1}a \in \sub{\st(\Lambda)}$, so $g\sub{\st(\Lambda)} = gg^{-1}a\sub{\st(\Lambda)} = a\sub{\st(\Lambda)}$. Thus $\rho^{g\Lambda}_{h\Omega} = \pi_{h\Omega}(g\sub{\st(\Lambda)}) = \pi_{h\Omega}(a\sub{\st(\Lambda)})$. Note that any element of $a\sub{\st(\Lambda)}$ can be expressed in the form $a\lambda l$ where $\lambda \in \sub{\Lambda}$ and $l \in \sub{\lk(\Lambda)}$. Using equivariance (Proposition \ref{prop:gates_to_subgroups}(\ref{gate:equivariance})) and the prefix description of the gate map (Lemma \ref{lem:head_lemma}), we have  \[\gate_{a \Omega}(a \lambda l) = a \cdot \gate_{\Omega}(a^{-1} a \lambda l) = a \cdot \prefix_\Omega(\lambda l) = a\lambda \cdot \prefix_{\Omega}(l). \]
This implies $\gate_{a \Omega}(a\lambda l) = a\lambda  l_{0}$, where $l_{0} = \prefix_{\Omega}(l) \in  \sub{\lk(\Lambda) \cap \Omega}$ and so $\supp(\lambda l_{0}) \subseteq \Lambda \cup (\lk(\Lambda) \cap \Omega) = \st(\Lambda) \cap \Omega \subsetneq \Omega$. Moreover, by Lemma \ref{lem:C_spaces_are_well_defined}, $\gate_{h \Omega}(a\lambda l) = \gate_{h \Omega}(\gate_{a\Omega}(a\lambda l)) = \gate_{h \Omega}(a\lambda l_{0})$.

Since $a \Omega \parallel h \Omega$, the gate map from $a\sub{\Omega}$ to $h \sub{\Omega}$ agrees with the isometry of $S(\Gamma)$ induced by the element $hpa^{-1}$ where $p = \prefix_\Omega(h^{-1}a)$ (Lemma \ref{lem:C_spaces_are_well_defined}). 
Since $\supp(\lambda l_{0}) \subsetneq \Omega$,  this implies $$\gate_{h\Omega}(a\lambda l_{0}) = hpa^{-1} \cdot a\lambda l_{0} = hp\lambda l_{0}.$$ 
Therefore, given two arbitrary elements $a\lambda l, a\lambda' l' \in a\sub{\st(\Lambda)}$, we have $(\gate_{h\Omega}(a\lambda l))^{-1}\gate_{h\Omega}(a\lambda'l') = l_{0}^{-1}\lambda^{-1}\lambda'l_{0}'$, where  $$\supp(l_{0}^{-1}\lambda^{-1}\lambda'l_{0}') \subseteq \st(\Lambda)\cap\Omega \subsetneq \Omega.$$ This implies the $C(h\Omega)$--diameter of  $\pi_{h\Omega}(g\sub{\st(\Lambda)}) = \rho_{h\Omega}^{g\Lambda}$ is at most $1$ in this case.
\end{proof}

Next we deal with the orthogonality relation.

\begin{defn}[Orthogonality]\label{defn:orthogonality}
  Let $G_\Gamma$ be a graph product and let $\mf{S}_\Gamma$ be the index set of parallelism classes of cosets of graphical subgroups described in Definition \ref{defn:index_set}. We say $[g\Lambda]\perp [h\Omega]$ if $\Lambda \subseteq \lk(\Omega)$ and there exists $k \in G_\Gamma$ such that $[k\Lambda] = [g\Lambda]$ and $[k\Omega]=[h\Omega]$.
\end{defn}

\begin{lem}[Orthogonality axiom]\label{lem:orthogonality}
The relation $\bot$ has the following properties:
\begin{enumerate}
    \item $\bot$ is symmetric; \label{orthogonality:symetric}
    \item If $[g\Lambda] \bot [h\Omega]$, then $[g\Lambda]$ and $[h\Omega]$ are not $\nest$--comparable; \label{orthogonality:nest_incomparable}
    \item If $[g\Lambda] \nest [h\Omega]$ and $[h\Omega] \bot [k\Pi]$, then $[g\Lambda] \bot [k\Pi]$. \label{orthogonality:orthogonal_to_nest}
\end{enumerate}
\end{lem}

\begin{proof}
(\ref{orthogonality:symetric}) If $\Lambda \subseteq \lk(\Omega)$, then all vertices of $\Lambda$ are connected to all vertices of $\Omega$, hence $\Omega \subseteq \lk(\Lambda)$ too. Thus the relation $\bot$ is symmetric.

(\ref{orthogonality:nest_incomparable}) Any graph is disjoint from its own link, hence if $[g\Lambda] \bot [h\Omega]$ then $[g\Lambda]$ and $[h\Omega]$ cannot be $\nest$--comparable.

(\ref{orthogonality:orthogonal_to_nest}) Suppose $[g\Lambda] \nest [h\Omega]$ and $[h\Omega] \bot [k\Pi]$. Then $\Lambda \subseteq \Omega \subseteq \lk(\Pi)$, and there exist $a,b \in G_{\Gamma}$ such that  \[[a\Lambda] = [g\Lambda], \ [a\Omega] = [h\Omega] = [b\Omega], \text{ and } [b\Pi] = [k\Pi].\] In particular, this means that $b^{-1}a \in \langle\st(\Omega)\rangle$, hence we can write $b^{-1}a = \omega l$ where $\omega \in \langle\Omega\rangle$ and $l \in \langle\lk(\Omega)\rangle$. Then $\omega^{-1}b^{-1}a = l \in \langle\lk(\Omega)\rangle \subseteq \langle\lk(\Lambda)\rangle \subseteq \langle\st(\Lambda)\rangle$, and so $[a\Lambda] = [b\omega\Lambda]$. On the other hand, $\omega^{-1}b^{-1}b = \omega^{-1} \in \langle\Omega\rangle \subseteq \langle\lk(\Pi)\rangle \subseteq \langle\st(\Pi)\rangle$, and so $[b\Pi] = [b\omega\Pi]$. Therefore $[g\Lambda] \perp [k\Pi]$, because $\Lambda \subseteq \lk(\Pi)$ and $[g\Lambda] = [b\omega\Lambda]$, $[k\Pi] = [b\omega\Pi]$.
\end{proof}

Our final relation is transversality, which is a little more nuanced, since our $[g\Lambda]$ and $[h\Omega]$ need not have a common representative $k$ in this case.

\begin{defn}[Transversality and lateral relative projections]\label{defn:transversality}
If $[g\Lambda], [h\Omega] \in \mathfrak{S}_{\Gamma}$ are not orthogonal and neither is nested in the other, then we say $[g\Lambda]$ and $[h\Omega]$ are transverse, denoted $[g\Lambda] \trans [h\Omega]$. When $[g\Lambda]\trans[h\Omega]$, for each choice of representatives $g\Lambda \in [g\Lambda]$ and $h\Omega \in [h\Omega]$, define $\rho_{g\Lambda}^{h\Omega} \subseteq C(g\Lambda)$ by \[ \rho_{g\Lambda}^{h\Omega} = \bigcup_{k \Omega \parallel h \Omega} \pi_{g\Lambda} \bigl( k \sub{\Omega}\bigr) = \pi_{g\Lambda}\left( h\sub{\st(\Omega)}\right). \] 
 The next lemma verifies that $\rho_{g\Lambda}^{h\Omega}$ has diameter at most $2$.
\end{defn}

\begin{lem}\label{lem:side_ways_projections_are_bounded}
If $[g\Lambda]\trans[h\Omega]$, then for any  choice of representatives $g \Lambda \in [g\Lambda]$ and $h \Omega \in [h\Omega]$, we have $\diam \bigl(\pi_{g\Lambda}({h} \sub{\st(\Omega)}) \bigr) \leq 2$ and $\diam\bigl(\pi_{h\Omega}(g\sub{\st(\Lambda)}) \bigr) \leq 2$.
\end{lem}

\begin{proof}
We provide the proof for $\diam \bigl(\pi_{g\Lambda}({h} \sub{\st(\Omega)}) \bigr) \leq 2$. The other case is identical.

Let $x,y \in h \sub{\st(\Omega)}$. Define $p_x = \pi_{g\Lambda}(x) = \gate_{g\Lambda}(x)$  and $p_y = \pi_{g\Lambda}(y) = \gate_{g\Lambda}(y)$. If $\Lambda$ splits as a join $\Lambda_1 \bowtie \Lambda_2$, then $\dist_{g\Lambda}(p_x,p_y) \leq \diam(C(g\Lambda)) \leq 2$  by Remark \ref{rem:join_implies_bounded_diameter}.

Now suppose $\Lambda$ does not split as a join.
Since $p_x,p_y \in g\sub{\Lambda}$, we have $\supp(p_x^{-1}p_y) \subseteq \Lambda$. If $\supp(p_x^{-1}p_y)$ is a proper subgraph of $\Lambda$, then the $C(g\Lambda)$--distance between $p_x$ and $p_y$ will be at most $1$. Thus, it suffices to prove $\supp(p_x^{-1} p_y) \neq \Lambda$.

Since $[g\Lambda]\trans[h\Omega]$ we have that $[g\Lambda] \not \perp [h\Omega]$, $[g\Lambda] \not \nest [h\Omega]$, and $[h\Omega] \not \nest [g \Lambda]$. This can occur in two different ways; either  
  $\Lambda \not\subseteq \lk(\Omega)$, $\Omega \not\subseteq \Lambda$ and $\Lambda \not\subseteq \Omega$, or there does not exist $k \in G_\Gamma$ so that $[g\Lambda] = [k \Lambda]$ and $[h \Omega] = [ k \Omega]$.

First assume $\Lambda \not\subseteq \lk(\Omega)$ and $\Lambda \not\subseteq \Omega$.   Then $\Lambda \not\subseteq \st(\Omega)$, as $\Lambda$ also does not split as a join. This implies that $\st(\Omega)\cap\Lambda \neq \Lambda$. By Proposition \ref{prop:syllables_of_gate_ are_syllables_of_image}, every syllable of $p_x^{-1}p_y$ is a syllable of $x^{-1}y$. Since   $x^{-1}y \in \sub{\st(\Omega)}$, this implies $\supp(p_x^{-1}p_y) \subseteq \st(\Omega) \cap \Lambda \neq \Lambda$ as desired. 

Now assume $\Lambda \subseteq \lk(\Omega)$ or $\Lambda \subseteq \Omega$. Thus, there does not exist $k \in G_\Gamma$ so that $[g\Lambda] = [k \Lambda]$ and $[h \Omega] = [ k \Omega]$. For the purposes of contradiction, suppose $\supp(p_x^{-1}p_y) = \Lambda$.

Let $s_x$ and $s_y$ be the suffixes of $x$ and $y$ respectively such that $x= p_xs_x$ and $y = p_y s_y$. Select the following $S(\Gamma)$--geodesics: $\alpha_x$  connecting $x$  and $p_x$, $\alpha_y$ connecting $y$ and $p_y$, $\eta$ connecting $p_x$ and $p_y$, $\gamma$ connecting $x$ and $y$; see Figure \ref{fig:link_of_Lambda_case}.

\begin{figure}[ht]
\centering
\def\svgscale{.7}
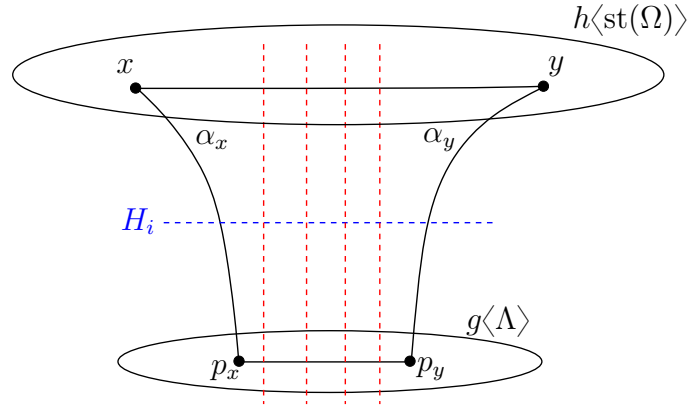
\caption{Any hyperplane that crosses $\alpha_x$ and $\alpha_y$ must cross all of the hyperplanes separating $p_x$ and $p_y$.}\label{fig:link_of_Lambda_case}
\end{figure} 

Let $t_1\dots t_n$ be the reduced syllable expression for $s_x$ corresponding to the geodesic $\alpha_x$. For each $i \in \{1,\dots,n\}$, let $H_i$ be the hyperplane crossing the edge of $\alpha_x$ labelled by $t_i$. Recall, a hyperplane in $S(\Gamma)$ crosses a geodesic segment if  and only if it separates the end points of the segment (Proposition \ref{prop:hyperplanes_properties}(\ref{hyperplanes:geodesic_iff_cross_once})). Each $H_i$ therefore separates $x$ and $p_x = \gate_{g\Lambda}(x)$,  so each $H_i$ must separate $x$ from all of $g\sub{\Lambda}$ by Proposition \ref{prop:gates_to_subgroups}(\ref{gate:hyperplanes_separating_image}). In particular, no $H_i$  crosses $\eta$. 
Thus, by Remark \ref{remark:hyperplanes_cross_loops_twice}, each $H_i$ must cross either  $\gamma$ or  $\alpha_y$. If $H_i$ crosses $\gamma$, then $t_i \in \sub{\st(\Omega)}$. On the other hand, if $H_i$ crosses $\alpha_y$, then $H_i$ must cross every hyperplane that separates $p_x$ and $p_y$; see Figure \ref{fig:link_of_Lambda_case}. 
Because $\supp(p_x^{-1} p_y) = \Lambda$, it follows that for every vertex $v$ of $\Lambda$ there exists a hyperplane that separates $p_x$ and $p_y$ and is labelled by $v$. Hence, if $H_i$ crosses $\alpha_y$, then $H_i$ crosses at least one hyperplane that is labelled by each vertex of $\Lambda$. By Proposition \ref{prop:hyperplanes_properties}(\ref{hyperplanes:cross_iff_adjacent_labels}), if two hyperplanes cross then they are labelled by adjacent vertices in $\Gamma$. Thus, the vertex labelling $H_i$ must be in the link of $\Lambda$. In particular, $t_i \in \sub{\lk(\Lambda)}$.

The above shows that $t_i \in \sub{\st(\Omega)}$ or $t_i \in \sub{\lk(\Lambda)}$ for each $i \in\{1,\dots,n\}$. Further, $t_i \in \sub{\st(\Omega)}$ if $H_i$ crosses $\gamma$ and $t_i \in \sub{\lk(\Lambda)}$ if $H_i$ crosses $ \alpha_y$. Now suppose $i < j$  and that $H_i$ crosses $\gamma$, but $H_j$ crosses $\alpha_y$. As shown in Figure \ref{fig:link_in_front}, this forces $H_i$ to cross $H_j$, which implies that $t_i$ and $t_j$ commute by Proposition \ref{prop:hyperplanes_properties}(\ref{hyperplanes:cross_iff_adjacent_labels}).
Thus, by commuting the syllables of $s_x$, we have $s_x =  l_x \omega_x$ where $\omega_x \in \sub{\st(\Omega)}$ and $l_x \in \sub{\lk(\Lambda)}$.

\begin{figure}[ht]
\centering
\def\svgscale{.7}
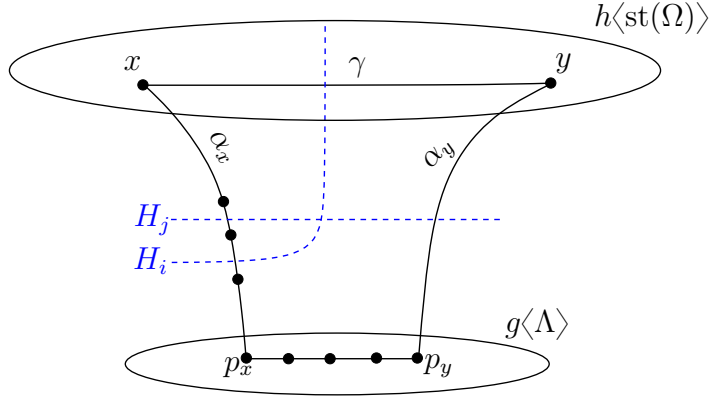
\caption{The hyperplane $H_i$ crosses $\alpha_x$ and $\gamma$ while $H_j$ crosses $\alpha_x$ and $\alpha_y$. Since $H_i$ appears before $H_j$ along $\alpha_x$, $H_i$ must cross $H_j$.}\label{fig:link_in_front}
\end{figure}

Now, since $x \in h\sub{\st(\Omega)}$, we have  $h^{-1}x \in \sub{\st(\Omega)}$, which implies $[h\Omega] = [x\Omega]$. Since $x = p_x s_x = p_x l_x \omega_x$, we have $[x\Omega] = [p_x  l_x \omega_x \Omega ] = [p_x l_x \Omega]$. Similarly, $p_{x} \in g\sub{\Lambda}$, so $g^{-1}p_{x} \in \sub{\Lambda}$, which implies $[g \Lambda ]= [p_x \Lambda]$. Now, $[p_x \Lambda] = [p_x l_x \Lambda]$ as $p_x^{-1} (p_x l_x) = l_x \in \sub{\lk(\Lambda)} \subseteq \sub{\st(\Lambda)}$. Thus we have \[[h\Omega] = [p_x l_x \Omega] \text{ and } [g\Lambda] = [p_x l_x \Lambda].\] However, this contradicts our assumption that there is no $k \in G_{\Gamma}$ such that $[h\Omega] = [k\Omega]$ and $[g\Lambda] = [k\Lambda]$, proving we must have $\supp(p_x^{-1} p_y) \neq \Lambda$ as desired.
\end{proof}

\subsection{The proto-hierarchy structure}
We now combine the work in this section to give a proto-hierarchy structure for $G_\Gamma$.

\begin{thm}\label{thm:proto-structure}
Let $G_\Gamma$ be a graph product of finitely generated groups. For each parallelism class $[g\Lambda] \in \mf{S}_\Gamma$, fix a representative $g\Lambda \in [g\Lambda]$.  The following is a  $2$--proto-hierarchy structure for $(G_\Gamma,\dist)$.
\begin{itemize}
    \item The index set is the set of parallelism classes $\mf{S}_\Gamma$ defined in Definition \ref{defn:index_set}.
    \item The space $C([g\Lambda])$ associated to $[g\Lambda]$ is the space $C(g\Lambda)$ from Definition \ref{defn:subgraph_metric}, where $g\Lambda$ is the fixed representative of $[g\Lambda]$.
    \item The projection map $\pi_{[g\Lambda]} \colon G_\Gamma \to C([g\Lambda])$ is the map $\pi_{g\Lambda} \colon G_\Gamma \to C(g\Lambda)$ from Definition \ref{defn:projections} for the fixed representative $g\Lambda \in [g\Lambda]$.
    \item $[g\Lambda] \nest [h\Omega]$ if $\Lambda \subseteq \Omega$ and there exists $k \in G_\Gamma$ such that $[k\Lambda] = [g\Lambda]$ and $[k \Omega] = [h\Omega]$. 
    \item The upwards relative projection $\rho_{[h\Omega]}^{[g\Lambda]}$ when $[g\Lambda]\propnest [h\Omega]$ is the set $\rho_{h\Omega}^{g\Lambda}$ from Definition \ref{defn:upward_projection}, where $g\Lambda$ and $h \Omega$ are the fixed representatives for $[h\Omega]$ and $[g\Lambda]$.
    \item $[g\Lambda] \perp [h\Omega]$ if $\Lambda \subseteq \lk(\Omega)$ and there exists $k \in G_\Gamma$ such that $[k\Lambda] = [g\Lambda]$ and $[k \Omega] = [h\Omega]$. 
    \item $[g\Lambda] \trans [h\Omega]$ whenever $[g\Lambda]$ and $[h\Omega]$ are not orthogonal and neither is nested into the other.
    \item The lateral relative projection $\rho_{[h\Omega]}^{[g\Lambda]}$ when $[g\Lambda] \trans [h\Omega]$ is the set $\rho_{h\Omega}^{g\Lambda}$ from Definition \ref{defn:transversality}, where $g\Lambda$ and $h \Omega$ are the  fixed representatives for $[h\Omega]$ and $[g\Lambda]$.
\end{itemize}
\end{thm}

\begin{proof}
The projection map $\pi_{[g\Lambda]}$ is shown to to be $(1,0)$--coarsely Lipschitz in Lemma \ref{lem:projections}. Nesting is shown to be a partial order in  Lemma \ref{lem:nesting_partial}. The upward relative projection has diameter at most $2$ by Lemma \ref{lem:nesting}. Lemma \ref{lem:orthogonality} shows that orthogonality is symmetric and mutually exclusive of nesting, and that 
nested domains inherit orthogonality. The lateral relative projections have diameter at most $2$ by Lemma \ref{lem:side_ways_projections_are_bounded}.
\end{proof}

\section{Graph products are relative HHGs}\label{section:rel_HHG}

In this section, we complete our proof that graph products of finitely generated groups are relative HHGs (Theorem \ref{thm:graph_products_are_rel_HHGs}) by proving the eight remaining HHS axioms and
showing that the group structure is compatible with our hierarchy structure. 
In Section \ref{section:hyperbolicity}, we prove hyperbolicity of $C(g\Lambda)$ whenever $\Lambda$ contains at least two vertices. Section \ref{section:finite_complexity_containers} is devoted to proving the finite complexity and containers axioms. Section \ref{section:uniqueness} deals with the uniqueness axiom, and in Section \ref{section:bgi_large_links}, we the prove the bounded geodesic image and large links axioms. In Section \ref{section:partial_realisation}, we verify partial realisation, and Section \ref{section:consistency} deals with the consistency axiom. Finally, in Section \ref{section:compatibility}, compatibility of the relative HHS structure with the group structure is checked.

We also obtain some auxiliary results along the way: in Section \ref{section:hyperbolicity}, we show that not only are the spaces $C(g\Lambda)$ hyperbolic whenever $\Lambda$ contains at least 2 vertices, but they are also quasi-trees; and in Section \ref{section:uniqueness}, we use uniqueness to give a classification of when $C(g\Lambda)$ has infinite diameter.

We conclude the section by remarking that the syllable metric on $G_{\Gamma}$ is a hierarchically hyperbolic space. This is true even when the vertex groups are not finitely generated. However, until then we will continue to assume $G_\Gamma$ is a graph product of finitely generated groups and that $\dist$ is the word metric on $G_\Gamma$, where the generating set for $G_{\Gamma}$ is given by taking a union of finite generating sets for each vertex group.

\subsection{Hyperbolicity}\label{section:hyperbolicity} 

\begin{lem}[Hyperbolicity]\label{lem:hyperbolicity}
For each $[g\Lambda] \in \mathfrak{S}_{\Gamma}$, either $[g\Lambda]$ is $\nest$--minimal or $C(g\Lambda)$ is $\frac{7}{2}$--hyperbolic.
\end{lem}

\begin{remark}
The hyperbolicity of $C(g\Lambda)$ can also be deduced from Proposition 6.4 of \cite{GenNeg}. The proof presented below uses a different argument that produces the explicit hyperbolicity constant of $\frac{7}{2}$.
\end{remark}

\begin{proof}
Take $[g\Lambda] \in \mathfrak{S}_{\Gamma}$ and suppose it is not $\nest$--minimal, i.e., $\Lambda$ contains at least two vertices. Let $x,y,z \in C(g\Lambda)$ be three distinct points and let $\gamma_{1},\gamma_{2},\gamma_{3}$ be three $C(g\Lambda)$--geodesics connecting the pairs $\{y,z\},\{z,x\},\{x,y\}$ respectively. We wish to show this triangle is $\frac{7}{2}$--slim, that is, we will show that $\gamma_1$ is contained in the $\frac{7}{2}$--neighbourhood of $\gamma_2 \cup \gamma_3$. Since $C(g\Lambda)$ is a metric graph whose edges have length $1$, it suffices to show that any vertex of $\gamma_{1}$ is at distance at most $3$ from $\gamma_{2}\cup\gamma_{3}$.

Let $p_1^i,\dots,p_{m_i}^i$ be the vertices of $\gamma_i$, and let $\gamma'_i$ be the path in $S(g\Lambda)$ obtained by connecting each pair of consecutive vertices $p_{j}^{i}$ and $p_{j+1}^i$ with an $S(g\Lambda)$--geodesic $\alpha_j^i$. Since $\alpha^{i}_{j}$ is labelled by vertices of $\supp((p_j^i)^{-1}p_{j+1}^i)$, which is a proper subgraph of $\Lambda$, the $C(g\Lambda)$--distance between any vertex of $\alpha_j^i$ and $p_j^i$ or $p^{i}_{j+1}$ is at most $1$. It therefore suffices to show that given any vertex $p^{1}_{j}$ of $\gamma_{1}$, either $\alpha^{1}_{j-1}$ or $\alpha^{1}_{j}$ is $C(g\Lambda)$--distance $1$ from some $\alpha^{i}_{t}$ with $i=2$ or $3$. See Figure \ref{fig:hyp_triangle}.

\begin{figure}[ht]
\centering
\def\svgwidth{\columnwidth}
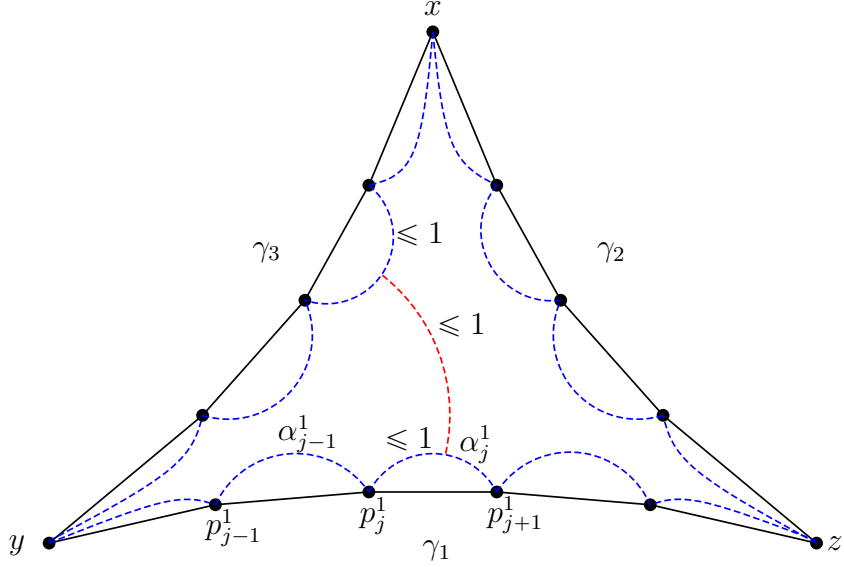
\caption{For each edge of the $C(g\Lambda)$--geodesic triangle, we construct an $S(g\Lambda)$--geodesic segment $\alpha^{i}_{j}$ between its endpoints (shown in blue). To show the triangle is $\frac{7}{2}$--slim, it then suffices to  show that for each $j$, $\alpha^{1}_{j-1}\cup\alpha^{1}_{j}$ is $C(g\Lambda)$--distance $1$ from some $\alpha^{i}_{t}$ with $i \neq 1$.}\label{fig:hyp_triangle}
\end{figure}

If $\Lambda$ has no edges, then $\langle\Lambda\rangle$ is the free product of the vertex groups, hence $S(g\Lambda)$ is a tree of simplices, that is, any  cycle in $S(g\Lambda)$ is contained in a single simplex (a coset of a vertex group). Therefore any two paths in $S(g\Lambda)$ with the same endpoints are contained in the $1$--neighbourhood of each other, and in particular $\gamma'_{1}$ is contained in the $1$--neighbourhood of $\gamma'_{2}\cup\gamma'_{3}$. Thus, any vertex of $\gamma_1$ is at  distance at most $3$ from $\gamma_{2}\cup\gamma_{3}$ in $C(g\Lambda)$.

Now suppose $\Lambda$ has at least one edge, so that it has a vertex $w$ with non-empty link. We may also assume that $\Lambda$ does not split as a join; otherwise, $C(g\Lambda)$ has diameter $2$ by Remark \ref{rem:join_implies_bounded_diameter} and hence is clearly $\frac{7}{2}$--hyperbolic. 
Take a vertex $p_j^1$ of $\gamma_1$. If $p^{1}_{j}$ is one of the first or last $4$ vertices of $\gamma_{1}$, then it is at distance at most $3$ from $\gamma_{2}$ or $\gamma_{3}$ in $C(g\Lambda)$. 
Otherwise $p^{1}_{j}$ is an endpoint of two consecutive edges $L_{j-1}$ and $L_{j}$ of $\gamma_{1}$ labelled by strict subgraphs $\Lambda_{j-1}$ and $\Lambda_{j}$ of $\Lambda$. We must have $\Lambda_{j-1} \cup \Lambda_{j} = \Lambda$, as otherwise we could replace these two edges with a single edge labelled by $\Lambda_{j-1}\cup\Lambda_{j}$, contradicting $\gamma_{1}$ being a $C(g\Lambda)$--geodesic. It follows that all vertices of $\Lambda$ appear as labels on the edges of the geodesic segments $\alpha^{1}_{j-1}$ and $\alpha^{1}_{j}$ of $\gamma'_{1}$ corresponding to $L_{j-1}$ and $L_{j}$. Consider the collection $\mathcal{E}_{w}$ of edges of $\alpha^{1}_{j-1} \cup \alpha^{1}_{j}$  labelled by the fixed vertex $w$ with $\lk(w) \cap \Lambda \neq \emptyset$, and consider the collection $\mathcal{H}_{w}$ of hyperplanes in $S(g\Lambda)$ dual to the edges in $\mathcal{E}_{w}$. We proceed to construct an $S(g\Lambda)$--path from an edge of $\mathcal{E}_{w}$ to some $\alpha^{i}_{t}$ with $i=2$ or $3$, either by travelling through the carrier of a single hyperplane, labelled by $\st(w)\cap\Lambda \subsetneq \Lambda$, or by following a sequence of combinatorial hyperplanes labelled by $\lk(w) \cap \Lambda \subsetneq \Lambda$. Since this path will be labelled by a proper subgraph of $\Lambda$, the $C(g\Lambda)$--distance between its endpoints will be $1$.

Suppose some hyperplane $H \in \mathcal{H}_{w}$ also crosses a geodesic segment $\alpha^{i}_{t}$ of $\gamma'_{2}\cup\gamma'_{3}$. Since the carrier of $H$ is labelled by vertices of  $\st(w) \cap \Lambda$, and $\st(w) \cap \Lambda$ is a strict subgraph of $\Lambda$ because $\Lambda$ does not split as a join, it follows that $p^{1}_{j}$ is at most $C(g\Lambda)$--distance $3$ from either $\gamma_{2}$ or $\gamma_{3}$, as desired.

Suppose therefore that no hyperplane of $\mathcal{H}_{w}$ crosses $\gamma'_{2}\cup\gamma'_{3}$. This means that each $H \in \mathcal{H}_{w}$ must cross $\gamma'_{1}$ a second time (Remark \ref{remark:hyperplanes_cross_loops_twice}). 
 Further, Proposition \ref{prop:hyperplanes_properties}(\ref{hyperplanes:cross_iff_adjacent_labels}) tells us that no two hyperplanes labelled by the same vertex may cross each other. It follows that there exists an outermost hyperplane $H_{0}$ of $\mathcal{H}_{w}$; that is, no hyperplane of $\mathcal{H}_{w}$ crosses edges of $\gamma'_{1}$ both earlier and later than $H_{0}$ does. Moreover, $H_{0}$ has an outermost combinatorial hyperplane $H_{0}'$; see Figure \ref{fig:outermost}. Note that since this combinatorial hyperplane is labelled by vertices of $\lk(w) \cap \Lambda \subsetneq \Lambda$, the $C(g\Lambda)$--distance between any two points on $H_{0}'$ is $1$. In particular, since $\gamma_{1}$ is a $C(g\Lambda)$--geodesic, it follows that the segments $\alpha^{1}_{r}$ and $\alpha^{1}_{k}$ that $H_{0}'$ intersects must satisfy $|k-r|\leq 2$. As we know that $H_{0}$ crosses $\alpha^{1}_{j-1}\cup\alpha^{1}_{j}$, this implies $H_{0}'$ must intersect $\alpha^{1}_{j-1}\cup\alpha^{1}_{j}$ too. Recalling that a hyperplane may not cross the same geodesic twice (Proposition \ref{prop:hyperplanes_properties}(\ref{hyperplanes:geodesic_iff_cross_once})), we may therefore suppose without loss of generality that $r=j$ and $j < k \leq j+2$ (the cases where $j-2 \leq k < j$ or $r = j-1$ proceed similarly).
 
 \begin{figure}[ht]
\centering
\def\svgscale{1.8}
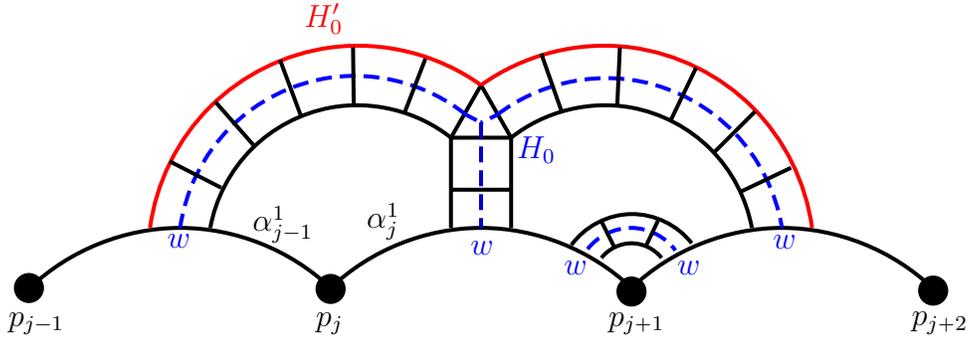
\caption{The outermost hyperplane $H_{0}$ of $\mc{H}_{w}$ and its outermost combinatorial hyperplane $H_{0}'$.}\label{fig:outermost}
\end{figure}
Let $E_{0}$ be the edge of $\mathcal{E}_{w}$ on $\alpha^{1}_{j}$ that $H_{0}$ crosses, and let $e_{1}$ and $e_{2}$ denote its endpoints. Let $F_{0}$ be the edge of $\alpha^{1}_{k}$ labelled by $w$ that $H_{0}$ crosses, and denote its endpoints by $f_{1}$ and $f_{2}$.  Then there is a path $\eta$ connecting $e_{1}$ and $f_{2}$  that is contained in the combinatorial hyperplane $H_{0}'$  labelled by vertices of $\lk(w) \cap \Lambda \subsetneq \Lambda$. Furthermore, if $w$ does not appear as a label of an earlier edge of $\alpha^{1}_{j}$ or a later edge of $\alpha^{1}_{k}$, then $\dist_{g\Lambda}(p_j^{1}, p_{k+1}^{1}) = 1$ as the path obtained by travelling from $p^{1}_{j}$ to $e_{1}$ along $\alpha^{1}_{j}$, then from $e_{1}$ to $f_{2}$ along $\eta$, then from  $f_{2}$ to $p^{1}_{k+1}$ along $\alpha^{1}_{k}$ is labelled by the proper subgraph $\Lambda \smallsetminus w$. This contradicts the assumption that $\gamma_{1}$ is a $C(g\Lambda)$--geodesic.  On the other hand, if $w$ appears as a label of an earlier edge $E_{-1}$ of $\alpha^{1}_{j}$ (take the closest one to $E_{0}$) but not a later edge of $\alpha^{1}_{k}$, then the corresponding hyperplane $H_{-1}$ must cross a segment $\alpha^{1}_{l}$ with $l<j$ (since $H_{0}$ is outermost), and there exists an $S(g\Lambda)$--path $\xi$ labelled by $\Lambda \smallsetminus w$ connecting $e_{1}$ and $\alpha^{1}_{l}$. Then the $C(g\Lambda)$--distance between the endpoints of the path $\xi \cup \eta$ is $1$ and so we obtain $\dist_{g\Lambda}(p^1_{l},p^1_{k+1}) \leq 2$, a contradiction. There therefore exists some edge labelled by $w$ which appears after  $F_{0}$ on $\alpha^{1}_{k}$. Let $E_{1}$ be the closest such edge to $H_{0}$, and consider the hyperplane $H_{1}$ dual to $E_{1}$.

If $H_{1}$ crosses $\alpha^{1}_{s}$ with $|s-j| \geq 3$, then we obtain a contradiction since we have a path in $ C(g\Lambda)$ from $p^{1}_{j}$ to $p^{1}_{s+1}$ (or $p^{1}_{j+1}$ to $p^{1}_{s}$ if $s<j$) of length at most $3$. If $H_{1}$ crosses $\alpha^{1}_{s}$ with $|s-k| \geq 3$, then similarly we obtain a contradiction. Assume therefore that $|s-j| \leq 2$ and $|s-k| \leq 2$. 
Note that since $H_{0}$ and $H_{1}$ cannot cross, we must have $s<j$ or $s>k$. 

If $s<j$ then we must have $k=j+1$ and $s=j-1$. In this case, $H_{1}$ crosses $\alpha^{1}_{j-1}$, which contradicts our assumption that $H_{0}$ is an outermost hyperplane of $\mathcal{H}_{w}$.Thus $H_{1}$ cannot cross any $\alpha^{1}_{s}$ with $s<k$. This implies that if $H_{1}$ crosses a segment $\alpha^{i}_{s}$ with $i=2$ or $3$, then we can conclude that $p^{1}_{j}$ is at most $C(g\Lambda)$--distance $3$ from either $\gamma_{2}$ or $\gamma_{3}$, by following a sequence of geodesics labelled by vertices of $\lk(w) \cap \Lambda$ and contained in combinatorial hyperplanes associated to $H_{0}$ and $H_{1}$; see Figure \ref{fig:follow_comb_hyps}.

\begin{figure}[ht]
\centering
\def\svgscale{1.8}
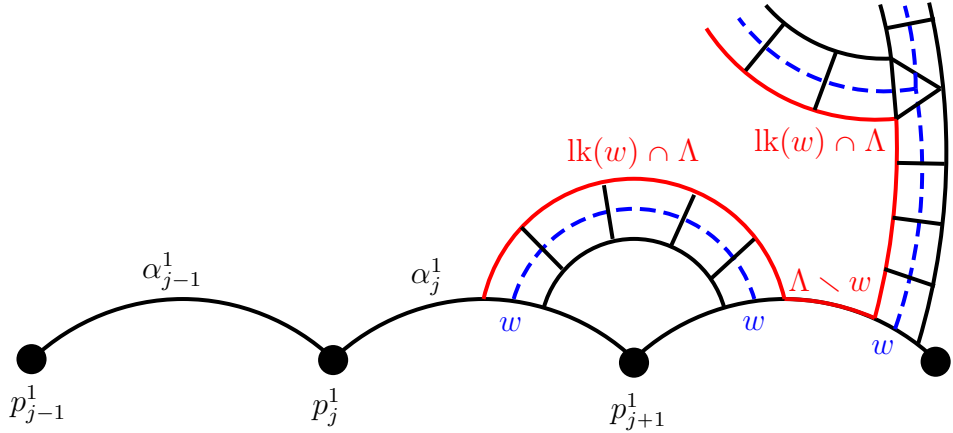
\caption{By following a sequence of combinatorial hyperplanes, we obtain a path labelled by $\Lambda\smallsetminus w$ (shown in red) that must eventually leave $\gamma_{1}'$ and cross $\gamma_{2}'\cup\gamma_{3}'$.}\label{fig:follow_comb_hyps}
\end{figure}

On the other hand, if $H_{1}$ crosses $\alpha^{1}_{s}$ with $s>k$, then $k=j+1$ and $s=j+2$. Repeating the same process, there must exist a later edge of $\alpha^{1}_{s}$ labelled by $w$. Let $H_{2}$ be the hyperplane dual to the closest such edge to $H_{1}$. If $H_{2}$ also crosses $\alpha^{1}_{t}$ where $t \neq s$, then we must have $t<j=s-2$ or $t>s=j+2$, as $H_{2}$ cannot cross the previous hyperplanes. 
However, the first case results in $|t-s| \geq 3$, and the second case gives $|t-j| \geq 3$, both of which give a contradiction. Therefore $H_{2}$ must cross $\alpha^{i}_{t}$ where $i=2$ or $3$. Following the sequence of geodesics labelled by vertices of $\Lambda \smallsetminus w$, we again see that $p^{1}_{j}$ is at most  $C(g\Lambda)$--distance $3$ from either $\gamma_{2}$ or $\gamma_{3}$.
\end{proof}

A similar technique can moreover show that the spaces $C(g\Lambda)$ are quasi-trees, by applying Manning's bottleneck criterion.

\begin{thm}[Bottleneck criterion \cite{ManningBottleneck}]
Let $Y$ be a geodesic metric space. The following are equivalent:
\begin{itemize}
    \item[(1)] $Y$ is quasi-isometric to some simplicial tree $T$;
    \item[(2)] There is some $\Delta>0$ so that for all $y,z\in Y$ there is a midpoint $m = m(y,z)$ with $\dist(y,m) = \dist(z,m) = \frac{1}{2}\dist(y,z)$ and the property that any path from $y$ to $z$ must pass within a distance $\Delta$ of $m$.
\end{itemize}
\end{thm}

\begin{thm}\label{thm:quasi-trees}
For each $[g\Lambda] \in \mathfrak{S}_{\Gamma}$, either $[g\Lambda]$ is $\nest$--minimal or $C(g\Lambda)$ is a quasi-tree.
\end{thm}

The proof of Theorem \ref{thm:quasi-trees} proceeds similarly to the proof of Lemma \ref{lem:hyperbolicity}, with the role of $\gamma_{1}$ being played by a geodesic from $y$ to $z$ containing the midpoint $m(y,z)$, and replacing $\gamma_{2}\cup\gamma_{3}$ with an arbitrary path from $y$ to $z$.

\begin{proof}
Suppose $[g\Lambda]$ is not $\nest$--minimal. Let $x,y \in C(g\Lambda)$, let $\gamma$ be a $C(g\Lambda)$--geodesic connecting $x$ and $y$, and let $\beta$ be another $C(g\Lambda)$--path from $x$ to $y$. From $\gamma$ and $\beta$ we may obtain paths $\gamma'$ and $\beta'$ in $S(g\Lambda)$ by replacing each edge with a geodesic segment in $S(g\Lambda)$. Note that any point on such a segment is $C(g\Lambda)$--distance $1$ from the endpoints of the segment. Let $m$ be the midpoint of $\gamma$, so that $m$ is either a vertex of $\gamma$ or a midpoint of an edge.

If $\Lambda$ has no edges, then $S(g\Lambda)$ is a tree of simplices in the same manner as in the previous proof, and in particular any two paths in $S(g\Lambda)$ between $x$ and $y$ are contained in the $1$--neighbourhood of each other. Applying this to $\gamma'$ and $\beta'$ shows that $m$ is at distance at most $\Delta = \frac{7}{2}$ from $\beta$.

Now suppose $\Lambda$ has at least one edge, and let $L_{1}$ and $L_{2}$ be two edges of $\gamma$ adjacent to $m$ (if $m$ is the midpoint of an edge $L$, pick $L$ and one edge adjacent to it). Then $L_{1}$ and $L_{2}$ are labelled by strict subgraphs $\Lambda_{1}$ and $\Lambda_{2}$ of $\Lambda$ such that $\Lambda_{1}\cup\Lambda_{2} = \Lambda$. Thus either $\Lambda_{1}$ or $\Lambda_{2}$ contains a vertex $w$ with non-empty link, and $w$ therefore appears as a label of a hyperplane crossing an edge of the corresponding geodesic segments $\alpha_{1}$ and $\alpha_{2}$ of $\gamma'$.

We can now repeat the argument in the proof of Lemma \ref{lem:hyperbolicity} to find a path   connecting $\alpha_{1}\cup\alpha_{2}$ to $\beta'$ that is labelled by a proper subgraph of $\Lambda$. It follows that $m$ is at most $C(g\Lambda)$--distance $\Delta = \frac{7}{2}$ from $\beta$.
\end{proof}

\subsection{Finite complexity and containers}\label{section:finite_complexity_containers}

\begin{lem}[Finite complexity]\label{lem:finite_complexity}
Any set of pairwise $\nest$--comparable elements has cardinality at most $|V(\Gamma)|$.
\end{lem}

\begin{proof}
If $[g\Lambda] \nest [h\Omega]$ and $\Lambda$ and $\Omega$ have the same number of vertices, then we must have $\Lambda = \Omega$ and $[g\Lambda] = [k\Lambda] = [k\Omega] = [h\Omega]$ for some $k \in G_\Gamma$. Therefore, any two distinct $\nest$--comparable elements must have different numbers of vertices. Thus any set of pairwise $\nest$--comparable elements has cardinality at most $|V(\Gamma)|$.
\end{proof}

\begin{lem}[Containers]\label{lem:containers}
Let $[h\Omega] \propnest [g\Lambda]$ be elements of $\mf{S}_{\Gamma}$.  If there exists $[k\Pi] \in \mathfrak{S}_{\Gamma}$ such that $[k\Pi] \nest [g\Lambda]$ and $[k\Pi] \bot [h\Omega]$, then $[k\Pi] \nest [a(\lk(\Omega)\cap\Lambda)] \propnest [a \Lambda]$ where $a \in G_\Gamma$ satisfies $[a\Lambda] = [g\Lambda]$ and $[a\Omega] = [h\Omega]$.
\end{lem}

\begin{proof}
First, since $[k\Pi] \nest [g\Lambda]$ and $[k\Pi] \bot [h\Omega]$, we have $\Pi \subseteq \Lambda$ and $\Pi \subseteq \lk(\Omega)$, hence $\Pi \subseteq \lk(\Omega)\cap\Lambda \subsetneq \Lambda$. Next, let $b \in G_{\Gamma}$ be such that $[b\Pi] = [k\Pi]$ and $[b\Omega] = [h\Omega]$, and let $c \in G_{\Gamma}$ be such that $[c\Pi] = [k\Pi]$ and $[c\Lambda] = [g\Lambda]$. We claim that there exists $d \in G_{\Gamma}$ such that $[k\Pi] = [d\Pi]$ and $[a(\lk(\Omega)\cap\Lambda)] = [d(\lk(\Omega)\cap\Lambda)]$, which would complete our proof.

Indeed, $k^{-1}a = k^{-1}bb^{-1}a = k^{-1}cc^{-1}a$, and we know that $$\supp(k^{-1}b) \subseteq \st(\Pi),\,\, 
\supp(b^{-1}a) \subseteq \st(\Omega)$$
and    $$\supp(k^{-1}c) \subseteq \st(\Pi),\,\, \supp(c^{-1}a) \subseteq \st(\Lambda).$$

Writing $p = \prefix_{\st(\Pi)}(k^{-1}a)$, we have $p^{-1}k^{-1}a = s$, where $s$ satisfies $\prefix_{\st(\Pi)}(s) = e$. That is, $\prefix_{\st(\Pi)}(p^{-1}k^{-1}bb^{-1}a) = e$. Since $p^{-1}k^{-1}b \in \langle\st(\Pi)\rangle$ and $b^{-1}a \in \langle\st(\Omega)\rangle$, this implies $p^{-1}k^{-1}a \in \langle\st(\Omega)\rangle$. Similarly, writing $k^{-1}a = k^{-1}cc^{-1}a$ shows us that $p^{-1}k^{-1}a \in \langle\st(\Lambda)\rangle$. 

That is to say, we can write $k^{-1}a = ps$ where   $p \in \langle\st(\Pi)\rangle$ and $s \in \langle\st(\Omega)\cap\st(\Lambda)\rangle$.
But $\Omega \subseteq \Lambda$ and $\lk(\Lambda) \subseteq \lk(\Omega)$, hence $$\st(\Omega) \cap \st(\Lambda) = \Omega \cup \lk(\Lambda)\cup (\lk(\Omega)\cap\Lambda).$$
Moreover, $\Omega\cup\lk(\Lambda) \subseteq \lk(\lk(\Omega)\cap\Lambda)$, hence $s \in \langle\st(\lk(\Omega)\cap\Lambda)\rangle$. Thus $k^{-1}as^{-1} = p \in \langle\st(\Pi)\rangle$ and $a^{-1}as^{-1} \in \langle\st(\lk(\Omega)\cap\Lambda)\rangle$. Letting $d = as^{-1}$, we have $[k\Pi] = [d\Pi]$ and $[a(\lk(\Omega)\cap\Lambda)] = [d(\lk(\Omega)\cap\Lambda)]$ as desired.
\end{proof}

\subsection{Uniqueness}\label{section:uniqueness} Here we prove the uniqueness axiom, which tells us that all geometry of $G_{\Gamma}$ is witnessed by some associated space $C(g\Lambda)$. This means we do not lose any geometric information through our projections. We also use this axiom to classify boundedness of the hyperbolic spaces $C(g\Lambda)$. In what follows, $|\cdot|_{G_{\Gamma}}$ denotes the word length on $G_{\Gamma}$ with respect to the generating set $S$ defined at the beginning of Section \ref{section:proto}.

\begin{lem}[Uniqueness]\label{lem:uniqueness}
Let $G_\Gamma$ be a graph product of finitely generated groups.  For all $g \in G_\Gamma$, if $\dist_{h\Lambda}(e,g) \leq r$ for all $h \in G_\Gamma$ and subgraphs $\Lambda \subseteq \Gamma$, then $|g|_{G_\Gamma} \leq (2^{|V(\Gamma)|} r+2)^{|V(\Gamma)|}$. 
\end{lem}

\begin{proof}
 Let $r \geq 0$. If $\Gamma$ is a single vertex, then the conclusion is immediate as the only subgraph is $\Gamma$ and $d_{\Gamma}(e,g) = |g|_{G_\Gamma} = r$.  Suppose $\Gamma$ contains $n+1$ vertices and assume the lemma holds for any graph product of finitely generated groups whose defining graph contains at most $n$ vertices. Suppose  $g \in G_\Gamma$ with $\dist_{h\Lambda}(e,g) \leq r$ for all $h \in G_\Gamma$ and subgraphs $\Lambda \subseteq \Gamma$.
 
  Since $\dist_\Gamma(e,g) \leq r$, there exist proper subgraphs $\Lambda_i \subsetneq \Gamma$ and elements $\lambda_i$ with $\supp(\lambda_i) = \Lambda_i$ so that $g = \lambda_1 \dots \lambda_m$ and $\dist_\Gamma(e,g) = m \leq r$. We shall see that $\dist_{h\Omega}(e,g) \leq r$ implies $\dist_{h\Omega}(e,\lambda_i)$ is uniformly bounded for each $\Omega \subseteq \Lambda_i$ and $h \in \sub{\Lambda_i}$. Since each $\sub{\Lambda_i}$ is a graph product on at most $n$ vertices, induction will imply the word length of each $\lambda_i$ is bounded, which in turn will bound the word length of $g$.

  If $\Gamma$ splits as a join $\Gamma = \Lambda_{1} \bowtie \Lambda_{2}$, then any element $g \in G_{\Gamma}$ can be written in the form $g = \lambda_{1}\lambda_{2}$ where $\lambda_{i} \in \langle\Lambda_{i}\rangle$ for $i=1,2$ and $|g|_{G_{\Gamma}} = |\lambda_{1}|_{G_{\Gamma}}+|\lambda_{2}|_{G_{\Gamma}}$. Moreover, if $h \in \langle\Lambda_{i}\rangle$ and $\Omega \subseteq \Lambda_{i}$, then $\gate_{h\Omega}(g) = h \cdot \prefix_{\Omega}(h^{-1}g) = h \cdot \prefix_{\Omega}(h^{-1}\lambda_{i}) = \gate_{h\Omega}(\lambda_{i})$.
  Therefore $\dist_{h\Omega}(e,\lambda_i) = \dist_{h\Omega}(e,g) \leq r$ and by induction  $|\lambda_i|_{G_\Gamma} \leq (2^nr+2)^n$ for $i=1,2$. Thus, $|g|_{G_{\Gamma}} \leq 2(2^nr+2)^n \leq (2^{n+1}r+2)^{n+1}$.

  Suppose $\Gamma$ does not split as a join, and define $p_0 =e$ and $p_i = \lambda_1\cdots \lambda_i$ for $i \in \{1,\dots,m\}$. Note that the $p_{i}$ are the vertices of the $C(\Gamma)$--geodesic connecting $e$ and $g$ with edges labelled by the $\lambda_{i}$. By Lemma \ref{lem:normal_subgraph_form}, we can assume that $ \suffix_{\Lambda_i}(p_{i-1}) = e$ for each $i \in \{2,\dots,m\}$ and that there exists an $S(\Gamma)$--geodesic connecting $e$ to $g$ that contains each $p_i$ as a vertex. Fix $i \in \{1,\dots m\}$, $h \in \langle\Lambda_i\rangle$, and $\Omega \subseteq \Lambda_i$. 

 As stated above, we wish to show $\dist_{h\Omega}(e,\lambda_i)$ is bounded uniformly in terms of $r$ so that we can apply the induction hypothesis.  Since $\dist_{h\Omega}(e,\lambda_i)$ is independent of the choice of representative of the coset $h \sub{\Omega}$, we can assume $\suffix_{\Omega}(h) = e$.  To achieve the bound on $\dist_{h\Omega}(e,\lambda_i)$, we  use the following two claims plus the assumption that $\dist_{h\Omega}(e,g) \leq r$. 
  
 \begin{claim}\label{claim:p_i-1 = e}
 $\pi_{p_{i-1}h\Omega}(p_{i-1}) = \pi_{p_{i-1}h\Omega}(e)$.
 \end{claim}
 
 \begin{proof}
 By equivariance of the gate map and the prefix description of the gate map (Lemma \ref{lem:head_lemma}),
 \[ \gate_{p_{i-1}h\Omega}(p_{i-1}) = p_{i-1}h \cdot \prefix_\Omega(h^{-1})\]  and \[\,\, \gate_{p_{i-1}h\Omega}(e) = p_{i-1}h \cdot \prefix_\Omega(h^{-1}p_{i-1}^{-1}).\]
Since $\prefix_{\Lambda_i}(p_{i-1}^{-1}) = e$, we have $\prefix_{\Omega}(p_{i-1}^{-1}) = e$ too.   Since $h \in \sub{\Lambda_i}$ and  $\prefix_{\Omega}(p_{i-1}^{-1}) = e$, we have $\prefix_\Omega(h^{-1}p_{i-1}^{-1}) = \prefix_\Omega(h^{-1})$ and so  $\gate_{p_{i-1}h\Omega}(p_{i-1}) = \gate_{p_{i-1}h\Omega}(e)$. This implies $\pi_{p_{i-1}h\Omega}(p_{i-1}) = \pi_{p_{i-1}h\Omega}(e)$.
  \end{proof}
  
\begin{claim}\label{claim:p_i = g}
$\dist_{p_{i-1}h\Omega}(p_i,g) \leq r$. 
\end{claim}

\begin{proof}[Proof of Claim \ref{claim:p_i = g}]
Recall, we can write each $\lambda_i$ in reduced syllable form to produce an $S(\Gamma)$--geodesic connecting $e$ and $g$ and containing each $p_i$ as a vertex (Lemma \ref{lem:normal_subgraph_form}). Thus, Lemma \ref{lem:subgraph_distance_along_geodesics} says $\dist_{p_{i-1}h\Omega} (p_i,g) \leq \dist_{p_{i-1}h\Omega} (e, g)$,  and $\dist_{p_{i-1}h\Omega} (e,g)\leq r$ by assumption.
\end{proof}

By the equivariance of the gate map (Proposition \ref{prop:gates_to_subgroups}(\ref{gate:equivariance})), $\dist_{h\Omega}(e,\lambda_i) = \dist_{p_{i-1}h\Omega}(p_{i-1},p_{i})$. Claim \ref{claim:p_i-1 = e}  then implies  \[\dist_{p_{i-1}h\Omega}(p_{i-1},p_{i}) =  \dist_{p_{i-1}h\Omega}(e,p_i) \leq \dist_{p_{i-1}h\Omega}(e,g) + \dist_{p_{i-1}h \Omega}(g,p_i).\] 
Since $\dist_{p_{i-1}h\Omega}(e,g) \leq r$ by assumption and $\dist_{p_{i-1}h \Omega}(g,p_i) \leq r$ by  Claim \ref{claim:p_i = g}, we have $\dist_{h\Omega}(e,\lambda_i) = \dist_{p_{i-1}h\Omega}(p_{i-1},p_{i}) \leq 2r$ for each $h \in \sub{\Lambda_i}$ and $\Omega \subseteq \Lambda_i$. 
The induction hypothesis now implies  the word length of $\lambda_i$ in $\sub{\Lambda_i}$ is at most  $(2^{n}(2r)+2)^n$. Thus we have $$|g|_{G_{\Gamma}} \leq r(2^{n+1}r+2)^{n} \leq (2^{n+1}r+2)^{n+1}$$ because each graphical subgroup is convexly embedded in the  word metric $\dist$ on $G_\Gamma$.   \qedhere
\end{proof}

The uniqueness axiom allows us to classify boundedness of the hyperbolic spaces $C(g\Lambda)$.

\begin{thm}\label{thm:diameter}
 For any $g \in G_\Gamma$ and any subgraph $\Lambda$ of $\Gamma$ containing at least two vertices, the space $C(g\Lambda)$ has infinite diameter if and only if $\Lambda$ does not split as a  join.
\end{thm}

\begin{proof}
Recall, if $\Lambda$ splits as a join, then $\diam(C(g\Lambda)) \leq 2$ by Remark \ref{rem:join_implies_bounded_diameter}. Suppose therefore that $\Lambda$ does not split as a  join and let $v_1,\dots,v_k$ be the vertices of $\Lambda$. For each $i \in \{1,\dots,k\}$, pick $s_i \in S_{v_i}$, where $S_{v_i}$  is the finite generating set for $G_{v_i}$ that  we fixed at the beginning of Section \ref{section:proto}. Define $\lambda = s_1\dots s_k$.
For each $i \in \{1,\dots,k\}$ and $j \in \{1,\dots, n\}$, let $s_i^j$ be the $j$th copy of $s_i$ in the product $(s_1\dots s_k)^n = \lambda^n$, that is, $\lambda^n =(s_1^1\dots s_k^1)(s_1^2 \dots s_k^2) \dots (s_1^n\dots s_k^n)$.

We claim that for each $n \in \mathbb{N}$, $(s_1^1\dots s_k^1) \dots (s_1^n\dots s_k^n)$ is a reduced syllable expression for $\lambda^{n}$. 
Indeed, if  $(s_1^1\dots s_k^1) \dots (s_1^n\dots s_k^n)$ is not reduced, then there exists $s_i^{j}$ that is combined with some $s_i^{\ell}$ ($j \neq \ell$) after applying some number of commutation relations (Theorem \ref{thm:graph_product_normal_form}). However,  if $s_i^{\ell}$ were to be combined with $s_i^{j}$, then $s_i$ would need to commute with each of  $s_1,\dots,s_{i-1},s_{i+1},\dots,s_k$. 
This only happens if the vertex $v_i$ is connected to every other vertex of $\Lambda$, but this does not happen as $\Lambda$ does not split as a join.  Therefore  $(s_1^1\dots s_k^1)\dots (s_1^n\dots s_k^n)$ is a reduced syllable expression for $\lambda^{n}$, and we have $|\lambda^n|_{syl}= kn$ for all $n \in \mathbb{N}$. 

To prove $C(g\Lambda)$ has infinite diameter, we use the following claim plus the uniqueness axiom to show that $\dist_\Lambda(e,\lambda^{n})$ can be made as large as desired by increasing $n$.

\begin{claim}\label{claim:Prefix_of_full_support}
 For all $\Omega \subsetneq \Lambda$, $h \in \sub{\Lambda}$ and $n \geq 2$, $\dist_{h\Omega}(e,\lambda^{n}) \leq 3$.
\end{claim}

For now we accept Claim \ref{claim:Prefix_of_full_support}, deferring its proof until after we have proved $C(g\Lambda)$ has infinite diameter.

For the purposes of contradiction, assume there exists $R>0$ such that $\dist_\Lambda(e,\lambda^n) \leq R$ for all $n \in \mathbb{N}$. By Claim \ref{claim:Prefix_of_full_support}, for every proper subgraph $\Omega \subsetneq \Lambda$ and $h \in \sub{\Lambda}$, we have $\dist_{h\Omega}(e,\lambda^{n}) \leq 3$.  Applying the uniqueness axiom (Lemma \ref{lem:uniqueness}) to the graph product $\sub{\Lambda} = G_\Lambda$, this implies there exists  $D=D(R, |V(\Lambda)|)>0$  such that $|\lambda^n|_{G_\Lambda} =|\lambda^n|_{G_\Gamma} \leq D$ for all $n \in \mathbb{N}$. However, this is a contradiction as $|\lambda^n|_{G_\Gamma} \geq |\lambda^{n}|_{syl} = kn$ for all $n \in \mathbb{N}$. Thus, for each $R >0$, there exists $n_R$ such that $\dist_\Lambda(e,\lambda^{n_R}) > R$. Therefore $C(\Lambda)$, and hence $C(g\Lambda)$, has infinite diameter.
\end{proof}

\begin{proof}[Proof of Claim \ref{claim:Prefix_of_full_support}]
Let $\Omega \subsetneq \Lambda$ be a proper subgraph and $h \in \sub{\Lambda}$. Since $\dist_{h\Omega}(e,\lambda^{n})$ does not depend on the choice of representative of the coset $h\sub{\Omega}$, we can assume $\suffix_\Omega(h) = e$, and thus $\prefix_\Omega(h^{-1}) =e$.

Recall, $\pi_{h\Omega} (e) = h \cdot \prefix_{\Omega}(h^{-1})$ and $\pi_{h\Omega}(\lambda^n) = h \cdot \prefix_\Omega(h^{-1}\lambda^n)$ (Remark \ref{rem:prefix_description_of_projection}).  Since $\prefix_\Omega(h^{-1}) =e$,  it suffices to prove that $\dist_\Omega(e,h^{-1}\lambda^{n}) \leq 3$. We can also assume that $ \prefix_\Omega(h^{-1}\lambda^{n}) \neq e$. 

By Proposition \ref{prop:syllables_of_gate_ are_syllables_of_image}, all syllables of $ \prefix_\Omega(h^{-1}\lambda^{n})$ are syllables of $\lambda^{n}$. As $ \prefix_\Omega(h^{-1}\lambda^{n}) \neq e$, there must exist $i \in \{1,\dots,k\}$ and $j \in \{1,\dots,n\}$ such that $s_i^j$ is the first syllable of $(s_1^1\dots s_k^1)(s_1^2 \dots s_k^2) \dots (s_1^n\dots s_k^n)$ that is also a syllable of $ \prefix_\Omega(h^{-1}\lambda^{n})$.

Let   $\ell,m \in \{1,\dots,k\}$ be such that   $v_{\ell} \in V(\Lambda \smallsetminus \st(\Omega))$ and $v_m\in V(\Omega)$ is not joined to $v_\ell$ by an edge. These vertices exist since $\Lambda$ does not split as a join and thus $\Lambda \not\subseteq \st(\Omega)$.
We will show that $\prefix_\Omega(h^{-1}\lambda^n)$ can be written as a product $p_1p_2p_3$ where $\supp(p_2)$ is a single vertex $v$ of $\Omega$ and $\supp(p_1),\supp(p_3) \subseteq \Omega \smallsetminus v$. This implies the   $C(\Omega)$--distance between $e$ and $\prefix_\Omega(h^{-1}\lambda^n)$ is at most $3$, which in turn says $\dist_{h\Omega}(e,\lambda^{n}) \leq 3$.

Suppose $i<\ell$. Since $v_\ell \not \in V(\Omega)$, every syllable of $\prefix_\Omega(h^{-1}\lambda^{n})$ must either be one of $s_i^j,s_{i+1}^j,\dots, s_{\ell-1}^j$ or must commute with $s^j_{\ell}$. As $s_{m}$ does not commute with $s_{\ell}$, it follows that no $s_m^J$ is a syllable of $\prefix_\Omega(h^{-1}\lambda^{n})$ for $J>j$. Therefore $\prefix_\Omega(h^{-1}\lambda^{n})$ can contain at most one syllable with support $v_m$, namely $s_m^j$. Thus $\prefix_\Omega(h^{-1}\lambda^{n}) = p_1p_2p_3$ with $\supp(p_1)\subseteq \Omega \smallsetminus v_m$, $\supp(p_2) \subseteq v_m$, and $\supp(p_3) \subseteq \Omega \smallsetminus v_m$. Note, if $\Omega = v_m$, then $\prefix_{\Omega}(h^{-1}\lambda^n) = p_2 = s_m^j$ and $\dist_{h\Omega}(e,\lambda^n) =\dist_{\Omega}(e,s_m^j) = 1$ because $s_m^j \in S_{v_m}$.

The case $i>\ell$ proceeds similarly since every syllable of $\prefix_\Omega(h^{-1}\lambda^{n})$ must either be one of $s_i^j,s_{i+1}^j,\dots, s_{k}^j, s_1^{j+1},\dots, s_{\ell-1}^{j+1}$ or must commute with $s^{j+1}_{\ell}$.
\end{proof}

In Section \ref{section:applications}, we use our characterisation of when $C(g\Lambda)$ has infinite diameter to answer two questions of Genevois \cite{Gen_Rigidity} (Theorems \ref{thm:electrification} and \ref{thm:quasi-line}).

\subsection{Bounded geodesic image and large links}\label{section:bgi_large_links} As the bounded geodesic image axiom is used to prove large links, we include both in this section. 

\begin{lem}[Bounded geodesic image]\label{lem:bounded_geodesic_image} 
Let $x,y \in G_{\Gamma}$ and $[h\Omega] \propnest [g\Lambda]$.  For any choice of representatives $h\Omega \in [h\Omega]$ and $g\Lambda \in [g\Lambda]$,  if $\dist_{h\Omega}(x,y) >0$, then  every $C(g\Lambda)$--geodesic $\gamma$ from $\pi_{g\Lambda}(x)$ to $\pi_{g\Lambda}(y)$ intersects the closed $2$--neighbourhood of $\rho^{h\Omega}_{g\Lambda}$.
\end{lem}

\begin{proof}
We first need to establish that when $[h \Omega] \nest [g\Lambda]$, gating onto $h \sub{\Omega}$ is the same as first gating onto $g\sub{\Lambda}$ and then gating onto $h\sub{\Omega}$. This will allow us to relate $\pi_{g\Lambda}(x)$ and $\pi_{h\Omega}(x)$.
\begin{claim}
 If $[h\Omega] \nest [g\Lambda]$, then $\gate_{h\Omega}(\gate_{g\Lambda}(x)) = \gate_{h\Omega}(x)$ for all $x \in G_{\Gamma}$ and for all representatives $g\Lambda \in [g\Lambda]$ and $h\Omega \in [h\Omega]$.
\end{claim}

\begin{proof}
Let $k \in G_\Gamma$ so that $[k\Omega] = [h \Omega]$ and $[k \Lambda] = [g\Lambda]$. Without loss of generality, we can assume $x \not\in g\sub{\Lambda}$. 

Suppose $\gate_{h\Omega}(\gate_{g\Lambda}(x)) \neq \gate_{h\Omega}(x)$. Then there is a hyperplane $H$ separating $\gate_{h\Omega}(\gate_{g\Lambda}(x))$ and $\gate_{h\Omega}(x)$. By Proposition \ref{prop:gates_to_subgroups}, $H$ also separates $\gate_{g\Lambda}(x)$ and $x$ and cannot cross $g \sub{\Lambda}$. However, we know that $H$ crosses  $h\langle\Omega\rangle \subseteq h\langle\Lambda\rangle$ and  by parallelism (Proposition \ref{prop:hyperplanes_cross_parallel}) $H$ must also cross $k\langle\Omega\rangle \subseteq k\langle\Lambda\rangle$. But $k \Lambda \parallel g\Lambda$, so $H$ must also cross $g \sub{\Lambda}$. This contradiction means  we must have $\gate_{h\Omega}(\gate_{g\Lambda}(x)) = \gate_{h\Omega}(x)$.
\end{proof}

Let $\gamma$ be a $C(g\Lambda)$--geodesic from  $\pi_{g\Lambda}(x)$ to $\pi_{g\Lambda}(y)$ and let $p_1,\dots,p_n \in \sub{\Lambda}$ so that $\pi_{g\Lambda}(x) = gp_1,gp_2,\dots,gp_n=\pi_{g\Lambda}(y)$ are the vertices of $\gamma$. Let $\alpha_i$ be an $S(g\Lambda)$--geodesic from $gp_i$ to $gp_{i+1}$ for each $i \in \{1,\dots,n-1\}$. Let $\gamma'$ be the path in $S(g\Lambda)$ that is the union of all the $\alpha_i$.

Suppose $\dist_{h\Omega}(x,y) > 0$. Then $\dist_{syl}(\gate_{h\Omega}(x),\gate_{h\Omega}(y)) > 0$ and so there is a hyperplane $H$ separating $\gate_{h\Omega}(x) = \gate_{h\Omega}(\gate_{g\Lambda}(x))$ and $\gate_{h\Omega}(y) = \gate_{h\Omega}(\gate_{g\Lambda}(y))$ that is labelled by a vertex   $w \in V(\Omega)$.  
The hyperplane $H$ then also separates $\gate_{g\Lambda}(x)$ and $\gate_{g\Lambda}(y)$ by Proposition \ref{prop:gates_to_subgroups}. 
Thus, $H$ must cross one of the segments $\alpha_i$ that make up $\gamma'$. Since $H$ crosses both $h\sub{\Omega}$ and $\alpha_i$ and $H$ cannot separate $\gate_{g\Lambda}(x)$ from $\gate_{h\Omega}(x)=\gate_{h\Omega}(\gate_{g\Lambda}(x))$ nor $\gate_{g\Lambda}(y)$ from $\gate_{h\Omega}(y)=\gate_{h\Omega}(\gate_{g\Lambda}(y))$ (Proposition \ref{prop:gates_to_subgroups}(\ref{gate:hyperplanes_separating_image})), there exists an $S(\Gamma)$--geodesic, $\eta$, from an element $b_0 \in h\sub{\Omega}$ to $a_0 \in \alpha_i$ that is labelled by vertices in $\lk(w)$; see Figure \ref{fig:BGI}.

\begin{figure}[ht]
     \centering
     \def\svgscale{.7}
     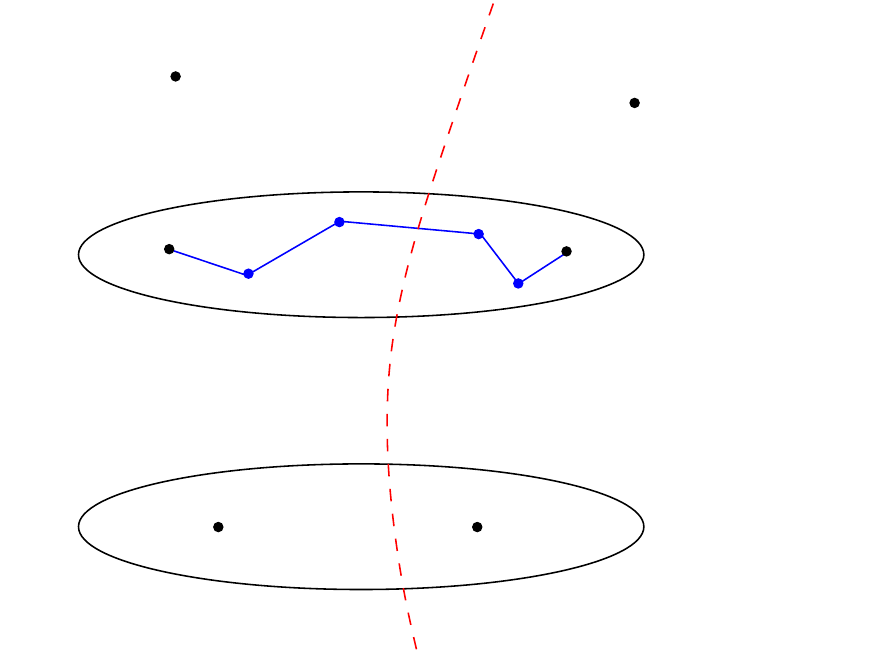
     \caption{The $S(\Gamma)$--geodesic $\eta$ connecting $b_0 \in h \sub{\Omega}$ and $a_0 \in \alpha_i$ when $d_{h\Omega}(x,y)$ is larger than $0$.}
     \label{fig:BGI}
 \end{figure}
 
Let  $a_1 = \pi_{g\Lambda}(a_0)$ and $b_1 = \pi_{g\Lambda}(b_0)$. 
Since $\eta$ was labelled by vertices in $\lk(w)$, Proposition \ref{prop:syllables_of_gate_ are_syllables_of_image} tells us we have $\supp(a_1^{-1}b_1) \subseteq \lk(w) \cap \Lambda$, which is a proper subgraph of  $\Lambda$.  Thus, in the subgraph metric, $\dist_{g\Lambda}(\alpha_i, \rho^{h\Omega}_{g\Lambda}) \leq 1$ as $a_1 \in \pi_{g\Lambda}(\alpha_i)$ and $b_1 \in \pi_{g\Lambda}(h\sub{\Omega}) \subseteq \rho_{g\Lambda}^{h\Omega}$.  
As $\alpha_{i}$ is labelled by a proper subgraph of $\Lambda$, any subsegment is also labelled by a proper subgraph, hence $\dist_{g\Lambda}(ga,gp_{i+1})\leq 1$ for any vertex $ga$ of $\alpha_{i}$. Thus, $\dist_{g\Lambda}(ga, \gamma) \leq 1$ and therefore $\dist_{g\Lambda}(\gamma, \rho^{h\Omega}_{g\Lambda}) \leq 2$.
\end{proof}

We can now use the bounded geodesic image axiom together with the following lemma to prove large links.

\begin{lem}\label{lem:big_gates}
Let $[g\Lambda], [h\Omega] \in \mathfrak{S}_{\Gamma}$. For any representatives $g \Lambda \in [g\Lambda]$ and $h \Omega \in [h\Omega]$, if $\diam(\pi_{g\Lambda}(h\langle\Omega\rangle)) > 2$, then $[g\Lambda] \nest [h\Omega]$.
\end{lem}

\begin{proof}
If $[g\Lambda] \trans [h\Omega]$ or $[h\Omega] \propnest [g\Lambda]$, then $\pi_{g\Lambda}\bigl(h\sub{\Omega}\bigr)  \subseteq \rho_{g\Lambda}^{h\Omega}$ , which is shown to have diameter at most $2$ in Lemmas \ref{lem:nesting} and \ref{lem:side_ways_projections_are_bounded}.  If $[g\Lambda] \perp [h\Omega]$, then  $\Lambda \subseteq \lk(\Omega)$. Let $\omega \in \sub{\Omega}$. Then $\gate_{g\Lambda}(h\omega) = g \cdot \prefix_{\Lambda}(g^{-1}h \omega)$. 
Assume without loss of generality that $$\suffix_{\Lambda}(g) = e \text{ and } \suffix_{\Omega}(h) = e.$$ By Proposition \ref{prop:syllables_of_gate_ are_syllables_of_image}, all syllables of $\prefix_{\Lambda}(g^{-1}h \omega)$ are syllables of $h\omega$. Further, since $\Lambda \subseteq \lk(\Omega)$, we have $\supp(\omega)\cap\Lambda = \emptyset$. As $\suffix_{\Omega}(h) = e$, this implies that $\prefix_{\Lambda}(g^{-1}h \omega) = \prefix_{\Lambda}(g^{-1}h)$. Thus, $\pi_{g\Lambda}(h\omega) = g \cdot \prefix_{\Lambda}(g^{-1}h)$ for all $\omega \in \sub{\Omega}$, and so $\diam\left(\pi_{g\Lambda}(h\sub{\Omega})\right) =0$.
\end{proof}

\begin{lem}[Large links]\label{lem:large_links}
Let $x,y \in G_\Gamma$ and $n = \dist_{k\Pi}(x,y)$ where $k \in G_\Gamma$ and $\Pi\subseteq \Gamma$. There exist $[h_1\Omega_1],\dots,[h_n\Omega_n] \in \mf{S}_\Gamma$ each nested into $[k\Pi]$ so that for any $[g\Lambda] \in \mf{S}_{\Gamma}$ with $[g\Lambda] \propnest [k\Pi]$, if $\dist_{g\Lambda}(x,y) > 18$ for some representative of $[g\Lambda]$, then $[g\Lambda] \nest [h_i \Omega_i]$ for some $i\in\{1,\dots,n\}$.
\end{lem}

\begin{proof}
Let $\gamma$ be a $C(k\Pi)$--geodesic connecting $\pi_{k\Pi}(x)$ and $\pi_{k\Pi}(y)$, let $\pi_{k\Pi}(x)=p_0,p_1,\dots,p_n=\pi_{k\Pi}(y)$ be the vertices of $\gamma$, and let $\lambda_i=p_{i-1}^{-1} p_{i}$ for each $i \in \{1,\dots,n\}$. 
For $i \in \{1,\dots,n\}$, define $T_i$ to be $p_{i-1}\cdot  \sub{\supp(\lambda_i)}$. 
Note that $p_i \in T_{i} \cap T_{i+1}$, and $T_{i} \subseteq k\sub{\Pi}$ since $p_{i-1} \in k\sub{\Pi}$ and $\supp(\lambda_{i}) \subsetneq \Pi$. In particular, $[T_{i}] \propnest [k\Pi]$. Note also that $\pi_{k\Pi}(T_i)$ = $T_{i}$ is contained in the closed $1$--neighbourhood of $p_i$ in $C(k\Pi)$, because $\supp(\lambda_i)$ is a proper subgraph of $\Pi$. 

Next, let $[g\Lambda] \in \mf{S}_\Gamma$ with $[g\Lambda] \propnest [k\Pi]$ and suppose $\dist_{g\Lambda}(x,y) > 18$ for some representative $g\Lambda \in [g\Lambda]$. We shall show $[g\Lambda] \nest [T_i]$ for some $i \in \{1,\dots,n\}$. Since we have established the bounded geodesic image axiom (Lemma \ref{lem:bounded_geodesic_image}), we have $\gamma \cap \mc{N}_{2}(\rho_{k\Pi}^{g\Lambda}) \neq \emptyset$, where $\mc{N}_{r}(A)$ is the closed $r$--neighbourhood of $A$ in $C(k\Pi)$.  
Let $j$ be the first number in $\{0,\dots,n\}$ so that $p_j \in \mc{N}_{4}(\rho_{k\Pi}^{g\Lambda})$, and recall that  each $\pi_{k\Pi}(T_{i})= T_{i}$ is contained in $\mc{N}_1(p_{i})$ and $\diam(\rho^{g\Lambda}_{k\Pi}) \leq 2$ (Lemma \ref{lem:nesting}). 
Therefore, if $1\leq i\leq j$ or $i\geq j+10$, then $\pi_{k\Pi}(T_i) \cap \mc{N}_2(\rho_{k\Pi}^{g\Lambda}) = \emptyset$ and the bounded geodesic image axiom says $\pi_{g\Lambda}(T_i)$ is a single point. 

Since  $T_{i-1} \cap T_i \neq \emptyset$ for $i \in \{2,\dots,n\}$ and $x \in T_{1}$, $y \in T_{n}$, we have
\[\pi_{g\Lambda} \left( \bigcup\limits_{i=1}^j T_i\right) = \pi_{g\Lambda}(x) \text{ and } \pi_{g\Lambda} \left( \bigcup\limits_{i=j+{10}}^{n} T_i\right) = \pi_{g \Lambda}(y)\] whenever $j > 0$ and $j+9<n$ respectively.
This implies 
\[ \dist_{g\Lambda}(x,y) \leq \sum _{i=j+1}^{\min\{n,j+9\}} \diam\big(\pi_{g\Lambda} \left(  T_i\right) \bigr).\] 

Since $\dist_{g\Lambda}(x,y) > 18$, there must exist $j_0 \in \{j+1,\dots,\min\{n,j+9\}\}$ so that $\diam\bigl(\pi_{g\Lambda}(T_{j_0}) \bigr) > 2$. By Lemma \ref{lem:big_gates}, this implies $[g\Lambda] \nest [T_{j_0}]$.
\end{proof}

\subsection{Partial realisation}\label{section:partial_realisation}
We now prove partial realisation, which roughly says that given a collection of pairwise orthogonal $[g_{i}\Lambda_{i}] \in \mathfrak{S}_{\Gamma}$, the hyperbolic spaces $C(g_{i}\Lambda_{i})$ give a coordinate system for $G_{\Gamma}$.

We first prove that we can  always represent $n$ mutually orthogonal elements of $\mathfrak{S}_{\Gamma}$ by the same group element, and similarly for nesting chains. This allows us to simplify arguments involving three or more orthogonal domains by working within a fixed coset.

\begin{prop}\label{prop:mutual representation}
Let $[g_{1}\Lambda_{1}], \dots, [g_{n}\Lambda_{n}] \in \mf{S}_\Gamma$. If either $[g_{1}\Lambda_{1}] \nest \dots \nest [g_{n}\Lambda_{n}]$ or  $[g_{1}\Lambda_{1}], \dots, [g_{n}\Lambda_{n}]$ are pairwise orthogonal, then there exists $g \in G_\Gamma$ so that $[g\Lambda_i] = [g_i \Lambda_i]$ for all $i \in \{1,\dots,n\}$. 
\end{prop}

\begin{proof}
We proceed by induction. The initial case $n=2$ is true by definition. 
Suppose the statement is true for all $n<m$, and consider $n=m$, that is, we have $[g_{1}\Lambda_{1}], \dots, [g_{m}\Lambda_{m}] \in \mathfrak{S}_{\Gamma}$ which are either pairwise orthogonal or nested. 
Then in particular $[g_{1}\Lambda_{1}],\dots,[g_{m-1}\Lambda_{m-1}]$ are pairwise orthogonal (respectively nested), hence there exists $g \in G_{\Gamma}$ such that $[g\Lambda_{i}] = [g_{i}\Lambda_{i}]$ for all $i<m$.  Since $[g\Lambda_{i}] = [g_{i}\Lambda_{i}]$ if and only if $[\Lambda_{i}] = [g^{-1}g_{i}\Lambda_{i}]$, we can assume $g=e$ without loss of generality.
Then $[\Lambda_{i}] \bot [g_{m}\Lambda_{m}]$ (respectively $[\Lambda_{i}] \nest [g_{m}\Lambda_{m}]$) for each $i<m$, so for each $i<m$ there exists $k_i$ such that $k_i \in \sub{\st(\Lambda_i)}$ and $g_m^{-1}k_i \in \sub{\st(\Lambda_m)}$. Let $h$ be the shortest prefix of $g_m$ such that $g_m^{-1}h \in \sub{\st(\Lambda_m)}$. 
Since  $g_m^{-1}k_i \in \sub{\st(\Lambda_m)}$ for each $i \in \{1,\dots,m-1 \}$, we know $\supp(h) \subseteq \supp(k_i) \subseteq \st(\Lambda_i)$ for each $i<m$. 
Hence $[\Lambda_i] = [h\Lambda_i]$ for each $i<m$ and $[g_m \Lambda_m] = [h \Lambda_m]$. Thus, by induction the statement is true for all $n$.
\end{proof}

\begin{lem}[Partial realisation]\label{lem:partial_realisation}
Let $\{[g_{i}\Lambda_{i}]\}_{i=1}^{n}$ be a finite collection of pairwise orthogonal elements of $\mathfrak{S}_{\Gamma}$. For each $i \in \{1,\dots,n\}$, fix a choice of representative $g_i\Lambda_i$ for $[g_i\Lambda_i]$ and let $p_i \in C(g_i \Lambda_i)$. There exists $x \in G_\Gamma$ so that:
\begin{itemize}
    \item $\dist_{g_{i}\Lambda_{i}}(x,p_{i}) = 0$ for all $i$;
    \item for each $i$ and each $[h\Omega] \in \mathfrak{S}_{\Gamma}$, if $[g_{i}\Lambda_{i}] \nestneq [h\Omega]$ or $[h\Omega] \trans [g_{i}\Lambda_{i}]$,  then for any choice of representative $h\Omega \in [h\Omega]$ we have $\dist_{h\Omega}(x,\rho^{g_{i}\Lambda_{i}}_{h\Omega}) =0$.
\end{itemize}
\end{lem}

\begin{proof} 
By Proposition \ref{prop:mutual representation} there exists some $g \in G_{\Gamma}$ such that $[g_{i}\Lambda_{i}] = [g\Lambda_{i}]$ for all $i$. Define $p'_{i} = \gate_{g\Lambda_{i}}(p_{i}) = g\lambda_{i}$, where $\lambda_{i} \in \langle\Lambda_{i}\rangle$, and let $x = g\lambda_{1}\lambda_{2}\dots\lambda_{n}$. 
Then $\pi_{g\Lambda_{i}}(x) = g\cdot\prefix_{\Lambda_i}(g^{-1}x) = g\lambda_{i} = \pi_{g\Lambda_{i}}(p_{i})$ for each $i$, since orthogonality tells us the elements $\lambda_{i}$ all commute with each other and the subgraphs $\Lambda_{i}$ are all disjoint. Therefore $\dist_{g\Lambda_{i}}(x,p_{i}) = 0$ for all $i$, and so by Lemma \ref{lem:C_spaces_are_well_defined}, we have $\dist_{g_{i}\Lambda_{i}}(x,p_{i}) = 0$ for all $i$. 

Now, suppose $[g\Lambda_{i}] \nestneq [h\Omega]$ or $[g\Lambda_{i}] \trans [h\Omega]$ for some $i \in \{1,\dots,n\}$ and $[h \Omega] \in \mf{S}_\Gamma$. Since $\Lambda_j \subseteq \lk(\Lambda_i) \subseteq \st(\Lambda_i)$ for each $j \neq i$, we have $x=g\lambda_1\dots\lambda_n \in g \sub{\st(\Lambda_i)}$. Thus, $\pi_{h\Omega}(x) \in \pi_{h\Omega}\left( g\sub{\st(\Lambda_{i})}\right) = \rho_{h\Omega}^{g\Lambda_i}$ for any choice of representative $h\Omega$ of $[h\Omega]$.  Moreover, we have $\rho_{h\Omega}^{g\Lambda_i} = \bigcup_{k \Lambda_i \parallel g \Lambda_i} \pi_{h\Omega} \bigl( k \sub{\Lambda_i}\bigr) = \rho_{h\Omega}^{g_i \Lambda_i}$, since $g_{i}\Lambda_{i} \parallel g\Lambda_{i}$. This implies $\dist_{h\Omega}(x,\rho^{g_{i}\Lambda_{i}}_{h\Omega}) =0$.
\end{proof}

\subsection{Consistency}\label{section:consistency}

Finally, we prove consistency, which says that given two transverse domains $[g\Lambda]$ and $[h\Omega]$ in $\mathfrak{S}_{\Gamma}$, each element of $G_{\Gamma}$ projects uniformly close to one of the lateral relative projections $\rho^{g\Lambda}_{h\Omega}$ and $\rho^{h\Omega}_{g\Lambda}$. 

Our proof shall proceed by contradiction. Assuming that each element of $G_{\Gamma}$ projects far from both lateral projections, we can use Lemma \ref{lem:big_gates} to show that  $[g\Lambda] \nest [h\lk(w)]$ for each vertex $w$ of $\Omega$, which will imply $[g\Lambda] \bot [hw]$ for each vertex $w$ of $\Omega$. We then obtain $[g\Lambda] \bot [h\Omega]$ by adapting the proof of Proposition \ref{prop:mutual representation} to show that we may promote orthogonality with multiple domains to orthogonality with their union. This contradicts $[g\Lambda] \trans [h\Omega]$.

\begin{lem}\label{lem:promoting_orthogonality}
Let $[g\Lambda_{1}],\dots,[g\Lambda_{n-1}],[k\Lambda_{n}] \in \mathfrak{S}_{\Gamma}$. If $[g\Lambda_{i}] \bot [k\Lambda_{n}]$ for all $i<n$, then $[g\bigcup_{i<n}\Lambda_{i}] \bot [k\Lambda_{n}]$.
\end{lem}

\begin{proof}
 Since $[g\Lambda_{i}] \bot [k\Lambda_{n}]$ if and only if $[\Lambda_{i}] \bot [g^{-1}k\Lambda_{n}]$, we may assume that $g=e$. By orthogonality, for each $i<n$ there exists $k_i$ such that $k_i \in \sub{\st(\Lambda_i)}$ and $k^{-1}k_i \in \sub{\st(\Lambda_n)}$. Following the proof of Proposition \ref{prop:mutual representation}, let $h$ be the  shortest prefix of $k$ such that $k^{-1}h \in \sub{\st(\Lambda_n)}$. Then $\supp(h) \subseteq \supp(k_{i}) \subseteq \st(\Lambda_{i})$ for all $i<n$, so $h \in \langle\st(\Lambda_{i})\rangle$ for all $i<n$. Therefore $h \in \langle\bigcap_{i<n}\st(\Lambda_{i})\rangle \subseteq \langle\st(\bigcup_{i<n}\Lambda_{i})\rangle$, hence  $[\bigcup_{i<n}\Lambda_{i}] = [h\bigcup_{i<n}\Lambda_{i}]$ and $[k\Lambda_{n}] = [h\Lambda_{n}]$. Moreover, by orthogonality, $\Lambda_{n} \subseteq \lk(\Lambda_{i})$ for all $i<n$, hence $\Lambda_{n} \subseteq \bigcap_{i<n}\lk(\Lambda_{i}) = \lk(\bigcup_{i<n}\Lambda_{i})$. We therefore have  $[\bigcup_{i<n}\Lambda_{i}] \bot [k\Lambda_{n}]$.
\end{proof}

\begin{lem}[Consistency]\label{lem:consistency}
If $[g\Lambda] \trans [h\Omega]$, then  for all $x \in G_\Gamma$ and for any choice of representatives $g\Lambda \in [g\Lambda]$ and $h\Omega \in [h\Omega]$ we have
\[\min\left\{\dist_{h\Omega} \Bigl( \pi_{h\Omega}(x),\rho^{g\Lambda}_{h\Omega}  \Bigr),\dist_{ g\Lambda} \Bigl( \pi_{g\Lambda}(x), \rho^{h\Omega}_{g\Lambda} \Bigr) \right\}\leq {2}. \tag{$*$}\label{eg:consistency}\] 
Further, if $[k\Pi] \propnest [g\Lambda]$ and either $[g\Lambda] \propnest [h\Omega]$ or $[g\Lambda] \trans [h\Omega]$ and $[h\Omega] \not\perp [k\Pi]$, then $\dist_{h\Omega}(\rho^{k\Pi}_{h\Omega},\rho^{g\Lambda}_{h\Omega}) = 0$.
\end{lem}

\begin{proof}
We prove $(*)$ by contradiction. Suppose $\dist_{h\Omega}(\pi_{h\Omega}(x),\rho^{g\Lambda}_{h\Omega}) > 2$ and $\dist_{g\Lambda}(\pi_{g\Lambda}(x),\rho^{h\Omega}_{g\Lambda}) > 2$. 
Then we also have 
\[\dist_{syl}(\gate_{h\Omega}(x), \gate_{h\Omega}(g\langle\Lambda\rangle)) > 2 \text{ and } \dist_{syl}(\gate_{g\Lambda}(x), \gate_{g\Lambda}(h\langle\Omega\rangle)) > 2.\]  Thus $\gate_{h\Omega}(x)$ and $\gate_{h\Omega}(g\langle\Lambda\rangle)$ are separated by some hyperplane $H_{w}$ labelled by a vertex $w$ of $\Omega$. By Proposition \ref{prop:gates_to_subgroups}(\ref{gate:hyperplanes_separating_pairs}), $H_w$ also separates $x$ and $g\langle\Lambda\rangle$. In particular, $H_w$ crosses  any $S(\Gamma)$--geodesic segment $\gamma$ connecting $x$ and $g \sub{\Lambda}$. 
Because of Proposition \ref{prop:gates_to_subgroups}(\ref{gate:hyperplanes_separating_image}), $H_w$ cannot separate $ g\sub{\Lambda}$ and $\gate_{h\Omega}(g\sub{\Lambda})$ as $H_w$ crosses $h \sub{\Omega}$. Thus, there exists a combinatorial hyperplane of $H_w$ contained in the same component of $S(\Gamma) \smallsetminus H_w$ as both $g\sub{\Lambda}$ and $\gate_{h\Omega}(g \sub{\Lambda})$. Let $H'_w$ be this particular combinatorial hyperplane of $H_w$; see Figure \ref{fig:consistency}.

 \begin{figure}[ht]
     \centering
     \def\svgscale{.7}
     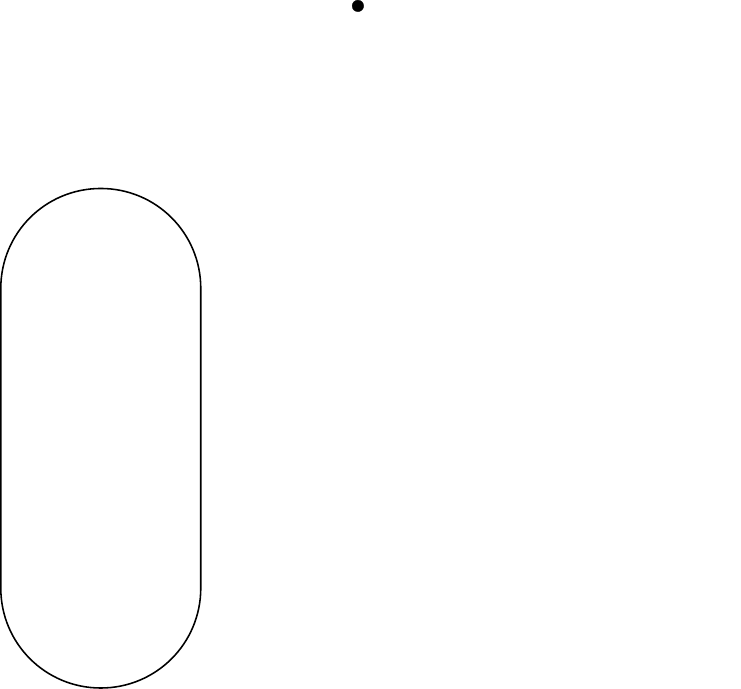
     \caption{The combinatorial hyperplane $H'_w$ of $H_w$ that is in the same component of $S(\Gamma) \smallsetminus H_w$ as both $g\sub{\Lambda}$ and $\gate_{h\Omega}(g \sub{\Lambda})$.}
     \label{fig:consistency}
 \end{figure}

We claim that $\diam(\pi_{g\Lambda}(H'_w)) > 2$. By construction, $H'_{w}$ contains both a vertex of $h\langle\Omega\rangle$ and a vertex of $\gamma$. Thus, $\pi_{g\Lambda}(H'_w)$ contains points from both $\pi_{g\Lambda}(h \sub{\Omega})$ and $\pi_{g\Lambda}(\gamma)$.    
Since $\gate_{g\Lambda}(x)$ is the unique point in $g \sub{\Lambda}$ that minimises the $S(\Gamma)$--distance from $x$ to $g\sub{\Lambda}$, we have $\gate_{g\Lambda}(\gamma) = \gate_{g\Lambda}(x) \in \pi_{g\Lambda}(H'_w)$. 
Since $\dist_{g\Lambda}(\pi_{g\Lambda}(x), \pi_{g\Lambda}(h\sub{\st(\Omega)}))  = \dist_{g\Lambda}(\pi_{g\Lambda}(x), \rho^{h\Omega}_{g\Lambda})>2$, and $\pi_{g\Lambda}(H'_w)$ must contain points from both $\pi_{g\Lambda}(x)$ and $\pi_{g\Lambda}(h \sub{\Omega})$, we must have $\diam(\pi_{g\Lambda}(H'_w)) > 2$.

By Remark \ref{rem:labeling_hyperplanes}, $H'_w \subseteq h \sub{\lk(w)}$. Thus, $\diam(\pi_{g\Lambda}(H'_w)) >2$ implies $\diam(\pi_{g\Lambda}(h\sub{\lk(w)})) > {2}$.  Lemma \ref{lem:big_gates} then forces $$[g \Lambda] \nest [h\lk(w)] \nest [h\st(w)].$$ 
This implies $\Lambda \subseteq \lk(w)$ and that there exists $k \in G_\Gamma$ such that $[k \Lambda ] = [g \Lambda]$ and $[k\st(w) ] = [h\st (w)]$. Since $\st(\st(w))= \st(w)$, $[k\st(w) ] = [h\st(w)]$ implies $[kw] = [hw]$. Thus $[g \Lambda] = [k \Lambda] \perp [k w] = [hw]$. 
Moreover, since $\dist_{h\Omega}(\pi_{h\Omega}(x),\rho^{g\Lambda}_{h\Omega}) > 2$, every vertex of $\Omega$ must appear as an edge label for the $S(h\Omega)$--geodesic  connecting $\gate_{h\Omega}(x)$ and $\gate_{h\Omega}(g\sub{\Lambda})$. 
Therefore such a hyperplane $H_{w}$ exists for every vertex $w$ of $\Omega$, and so $[g\Lambda] \perp [hw]$ for all $w \in V(\Omega)$. Lemma \ref{lem:promoting_orthogonality} then tells us $[g\Lambda] \perp [h\Omega]$, contradicting transversality. Hence inequality (\ref{eg:consistency}) must hold.

Now suppose $[k\Pi] \propnest [g\Lambda]$ and either $[g\Lambda] \propnest [h\Omega]$ or $[g\Lambda] \trans [h\Omega]$ and $[h\Omega] \not\bot [k\Pi]$. Then there exists some element $a$ such that $[k\Pi] = [a\Pi]$ and $[g\Lambda] = [a\Lambda]$.  Therefore  $\pi_{h\Omega}(a\langle\Pi\rangle) \subseteq  \rho^{k\Pi}_{h\Omega}$ and  $ \pi_{h\Omega}(a\langle\Lambda\rangle) \subseteq  \rho^{g\Lambda}_{h\Omega}$.
But $a\langle\Pi\rangle \subseteq a\langle\Lambda\rangle$, so $\dist_{h\Omega}(\rho^{k\Pi}_{h\Omega}, \rho^{g\Lambda}_{h\Omega}) = 0$.
\end{proof}

\subsection{Compatibility of the group structure}\label{section:compatibility}
The results so far show that a graph product $G_{\Gamma}$ can be given the structure of a relatively hierarchically hyperbolic space. It remains to show that this structure agrees with the group structure of $G_{\Gamma}$.

\begin{lem}\label{lem:HHG_relations}
 The map $\phi: G_{\Gamma} \times \mathfrak{S}_{\Gamma} \rightarrow \mf{S}_{\Gamma}$ where $\phi(a,[g\Lambda]) = [ag\Lambda]$ defines  a  $\nest$--, $\perp$--, and $\trans$--preserving action of $G_{\Gamma}$ on $\mf{S}_{\Gamma}$ by bijections such that $\mf{S}_{\Gamma}$ contains finitely many $G_{\Gamma}$--orbits.
\end{lem}

\begin{proof}
 Let $\phi_{a} = \phi(a,\cdot)$. This is well-defined, since $[g\Lambda] = [k\Lambda]$ if and only if $[ag\Lambda] = [ak\Lambda]$. Further, since $\phi_a$ does not alter the subgraph $\Lambda$, it preserves the orthogonality, nesting, and transversality relations. Each $\phi_{a}$ is also a bijection: if $[ag\Lambda] = [ah\Omega]$, then $\Lambda = \Omega$ and $(ag)^{-1}(ah) = g^{-1}h \in \langle\st(\Lambda)\rangle$, hence $[g\Lambda] = [h\Omega]$, proving injectivity. Surjectivity holds since we can always write $[g\Lambda] = \phi_{a}([a^{-1}g\Lambda])$. Finally, there are finitely many $G_{\Gamma}$--orbits; one for each subgraph $\Lambda \subseteq \Gamma$.
\end{proof}

\begin{lem}\label{lem:HHG_conditions}
  For each subgraph $\Lambda \subseteq \Gamma$  and elements $a,g\in G_{\Gamma}$, there exists an isometry $a_{g\Lambda} \colon C(g\Lambda) \rightarrow C(ag\Lambda)$ satisfying the following for all  subgraphs $\Lambda,\Omega \subseteq \Gamma$ and elements $a,b,g,h \in G_\Gamma$.
    \begin{itemize}
            \item The isometry $(ab)_{g\Lambda} \colon C(g\Lambda) \to C(abg\Lambda)$ is equal to the composition $a_{bg\Lambda} \circ b_{g\Lambda} \colon C(g\Lambda) \to C(abg\Lambda)$.
            \item For each $x \in G_{\Gamma}$, we have $a_{g\Lambda}(\pi_{g\Lambda}(x))=\pi_{ag\Lambda}(ax)$.
            \item If $[h\Omega] \trans [g\Lambda]$ or $[h\Omega] \propnest [g\Lambda]$, then $a_{g\Lambda}(\rho_{g\Lambda}^{h\Omega})=\rho_{ag\Lambda}^{ah\Omega}$.
    \end{itemize}
\end{lem}

\begin{proof}
Let the isometry $a_{g\Lambda}$ be left-multiplication by $a$, that is for any $gx \in C(g\Lambda)$, let $a_{g\Lambda}(gx) = agx$. Then: 
\begin{itemize}
    \item The equality $(ab)_{g\Lambda} = a_{bg\Lambda} \circ b_{g\Lambda}$ is immediate from our definition.
    \item We have $a_{g\Lambda}(\pi_{g\Lambda}(x)) = \pi_{ag\Lambda}(ax)$ by Proposition \ref{prop:gates_to_subgroups}(\ref{gate:equivariance}).
    \item The final property follows as an immediate consequence of the previous one and the definition of the relative projections.\qedhere
\end{itemize}
\end{proof}

\subsection{Graph products are relative HHGs}
We now compile the results from Section \ref{section:rel_HHG} to obtain the main result of this paper, that any graph product of finitely generated groups is a relative HHG.

\begin{thm}\label{thm:graph_products_are_rel_HHGs}
Let $G_\Gamma$ be a graph product of finitely generated groups.   The proto-hierarchy structure $\mf{S}_\Gamma$ from Theorem \ref{thm:proto-structure} is a relatively hierarchically hyperbolic group structure for $G_\Gamma$  with hierarchy constant $\max\{18,|V(\Gamma)|\}$ and uniqueness function \[\theta(r) = (2^{|V(\Gamma)|}r+2)^{|V(\Gamma)|}.\] 
\end{thm}

\begin{proof}

Let $\mf{S}_\Gamma$ be the proto-hierarchy structure for $(G_\Gamma,\dist)$ from Theorem \ref{thm:proto-structure}. The work of this section has shown that $\mf{S}_\Gamma$ is a relative HHS structure for $(G_\Gamma,\dist)$.

\begin{itemize}
    \item[(1)] We proved that the spaces associated to the non-$\nest$--minimal domains of $\mf{S}_\Gamma$ are  $\frac{7}{2}$--hyperbolic in Lemma \ref{lem:hyperbolicity}.
    \item[(2)] We proved finite complexity in Lemma \ref{lem:finite_complexity}.
    \item[(3)] We proved the container axiom in Lemma \ref{lem:containers}.
    \item[(4)] The proof of the uniqueness axiom  follows from Lemma \ref{lem:uniqueness}, since if $\dist_{C([g\Lambda])}(x,y)$ is uniformly bounded for all $[g\Lambda] \in \mf{S}_\Gamma$, then Lemma \ref{lem:C_spaces_are_well_defined} implies that $\dist_{g\Lambda}(x,y)$ has the same uniform bound for all $g \in G_\Gamma$ and  $\Lambda \subseteq \Gamma$.
    \item[(5)] We proved the bounded geodesic image axiom in Lemma \ref{lem:bounded_geodesic_image}.
    \item[(6)]  We proved the large links axiom in Lemma \ref{lem:large_links}. 
    \item[(7)] We proved the consistency axiom in Lemma \ref{lem:consistency}.
    \item[(8)] We proved the partial realisation axiom in Lemma \ref{lem:partial_realisation}.
\end{itemize}

We now verify the remaining axioms required for $(G_{\Gamma},\dist)$ to be a relative HHG, as laid out in Definition \ref{defn:hierarchically hyperbolic groups}.

Let $\phi: G_{\Gamma} \times \mathfrak{S}_{\Gamma} \rightarrow \mf{S}_{\Gamma}$ be the map $\phi(a,[g\Lambda]) = [ag\Lambda]$. By Lemma \ref{lem:HHG_relations}, this is a well-defined $G_\Gamma$--action by bijections that preserves the nesting, orthogonality, and transversality relations and has finitely many orbits.

For each $[g\Lambda] \in \mf{S}_\Gamma$, let $\overline{g} \Lambda$ denote the fixed representative of $[g\Lambda]$ such that $C([g\Lambda]) = C(\overline{g}\Lambda)$; see the proto-hierarchy structure in Theorem \ref{thm:proto-structure}.
Left multiplication by $a \in G_\Gamma$ gives an isometry $a_{g\Lambda} \colon C(g\Lambda) \to C(ag\Lambda)$ for each $g \in G_\Gamma$ and each subgraph $\Lambda \subseteq \Gamma$. For each $a\in G_\Gamma$ and $[g\Lambda] \in \mf{S}_\Gamma$, define $\mathbf{a}_{[g\Lambda]} \colon C(\overline{g}\Lambda) \to C(\overline{ag}\Lambda)$ by $\mathbf{a}_{[g\Lambda]} = \gate_{\overline{ag}\Lambda} \circ a_{\overline{g}\Lambda}$. 

Let $a,b \in  G_\Gamma$ and $[g\Lambda],[h\Omega] \in \mf{S}_\Gamma$. We now verify the remaining axioms of a relatively hierarchically hyperbolic group (Definition \ref{defn:hierarchically hyperbolic groups}).
\begin{itemize}
    \item Let $\lambda \in \sub{\Lambda}$. To show $(\mathbf{ab})_{[g\Lambda]} = \mathbf{a}_{[bg\Lambda]} \circ \mathbf{b}_{[g\Lambda]}$ we will show  \[(\mathbf{ab})_{[g\Lambda]}(\overline{g}\lambda) = (\mathbf{a}_{[bg\Lambda]} \circ \mathbf{b}_{[g\Lambda]})(\overline{g}\lambda).\] 
    Using the last clause of Lemma \ref{lem:C_spaces_are_well_defined}, we have \[(\mathbf{ab})_{[g\Lambda]}(\overline{g}\lambda) =\gate_{\overline{abg}\Lambda}(ab\overline{g}\lambda) = \overline{abg} \cdot p_{ab} \lambda\] where $p_{ab} = \prefix_{\Lambda}((\overline{abg})^{-1} \cdot ab\overline{g})$. 
    Similarly, we have 
    \[(\mathbf{a}_{[bg\Lambda]} \circ \mathbf{b}_{[g\Lambda]})(\overline{g}\lambda)= \mathbf{a}_{[bg\Lambda]}(\overline{bg} \cdot p_b \lambda) = \overline{abg} \cdot p_a p_b \lambda \] where $p_b = \prefix_\Lambda\bigl( (\overline{bg})^{-1} \cdot b\overline{g}\bigr)$ and  $p_a = \prefix_\Lambda\bigl( (\overline{abg})^{-1} \cdot a\overline{bg}\bigr)$.  
    Thus, it suffices to prove $p_ap_b = p_{ab}$. 
    
    Since $\overline{bg}$ and $b\overline{g}$ are both representatives of the parallelism class $[bg \Lambda]$, we have $(\overline{bg})^{-1} \cdot b\overline{g} \in \sub{\st(\Lambda)}$. 
    Therefore $ (\overline{bg})^{-1} \cdot b\overline{g} = p_bl_b$ where $l_b \in \sub{\lk(\Lambda)}$. 
    Similarly, $ (\overline{abg})^{-1} \cdot a \overline{bg} = p_a l_a$ where $l_a \in \sub{\lk(\Lambda)}$. Hence  the following calculation concludes our argument:
    \begin{align*}
 (\overline{abg})^{-1} \cdot ab \overline{g} =& (\overline{abg})^{-1} \cdot a \overline{bg} \cdot p_b l_b\\
 \prefix_\Lambda\left((\overline{abg})^{-1} \cdot ab \overline{g}\right) =& \prefix_\Lambda\left( (\overline{abg})^{-1} \cdot a \overline{bg} \cdot p_b l_b \right)\\
 p_{ab} =& \prefix_\Lambda( p_a l_a p_b l_b )\\
 p_{ab} =& p_a p_b. 
    \end{align*}

    \item Let $x \in G_\Gamma$. Since $a\overline{g}\Lambda \parallel \overline{ag} \Lambda$, we can use Lemma \ref{lem:C_spaces_are_well_defined} and equivariance of the gate map (Proposition \ref{prop:gates_to_subgroups}(\ref{gate:equivariance})) to conclude: 
    \begin{align*}
       \gate_{\overline{ag}\Lambda}\left(\gate_{a \overline{g}\Lambda}(ax)\right) =& \gate_{\overline{ag}\Lambda}(ax) \\
       \gate_{\overline{ag}\Lambda} \left(a \cdot \gate_{\overline{g}\Lambda}(x)\right) =& \gate_{\overline{ag}\Lambda}(ax)\\
       \left(\gate_{\overline{ag}\Lambda}\circ a_{\overline{g}\Lambda}\right) \left( \pi_{\overline{g}\Lambda}(x) \right) =& \pi_{\overline{ag}\Lambda}(ax)\\
       \mathbf{a}_{[g\Lambda]}\left( \pi_{[g\Lambda]}(x)\right) =& \pi_{[ag\Lambda]}(ax).
    \end{align*} 

    \item Suppose $[h\Omega] \trans [g\Lambda]$ or $[h\Omega] \propnest [g\Lambda]$.  Lemmas \ref{lem:C_spaces_are_well_defined}, \ref{lem:HHG_relations}, and \ref{lem:HHG_conditions} imply $\mathbf{a}_{[g\Lambda]} \left( \rho_{[g\Lambda]}^{[h\Omega]} \right) = \rho_{[ag\Lambda]}^{[ah\Omega]}$:
    \begin{flalign*}
        && \mathbf{a}_{[g\Lambda]} \left( \rho_{[g\Lambda]}^{[h\Omega]} \right) 
        =& \left( \gate_{\overline{ag}\Lambda}\circ a_{\overline{g}\Lambda} \right) \left( \rho_{\overline{g}\Lambda}^{\overline{h}\Omega}\right) && \text{(Definition of } \mathbf{a}_{[g\Lambda]})\\
        && =&  \gate_{\overline{ag}\Lambda} \left( \rho_{a \overline{g} \Lambda}^{a \overline{h} \Omega}  \right) && \text{(Lemma \ref{lem:HHG_conditions})}\\
        && =&  \gate_{\overline{ag}\Lambda} \left(  \gate_{a\overline{g} \Lambda}(a \overline{h} \sub{\st(\Omega)}) \right) && \text{(Definition of } \rho)\\
        && =&  \gate_{\overline{ag} \Lambda} ( a \overline{h} \sub{\st(\Omega)}) && \text{(Lemma \ref{lem:C_spaces_are_well_defined})}\\
        && =&  \gate_{\overline{ag} \Lambda} (  \overline{ah} \sub{\st(\Omega)}) && (a \overline{h} \Omega \parallel \overline{ah} \Omega)\\
        && =& \rho_{\overline{ag}\Lambda}^{\overline{ah}\Omega}. && \qedhere 
    \end{flalign*}
\end{itemize}
\end{proof}

Behrstock, Hagen, and Sisto show that any relatively hierarchically hyperbolic space has a distance formula, which expresses distances in the space as a sum of distances in the projections \cite[Theorem 6.10]{BHS_HHSII}. As a result, we now have such a distance formula for graph products of finitely generated groups.

\begin{cor}[Distance formula for graph products]\label{cor:distance_formula}
Let $G_\Gamma$ be a graph product of finitely generated groups. There exists $\sigma_0 >0$  such that for all $\sigma \geq \sigma_0$ there exist $K\geq 1$ and $L \geq 0$ such that for all $g,h \in G_\Gamma$ \[ \frac{1}{K} \sum_{[k\Lambda] \in \mf{S}_\Gamma}\threshold{\dist_{[k\Lambda]}(g,h)}{\sigma} -L \leq  \dist(g,h) \leq K \sum_{[k\Lambda] \in \mf{S}_\Gamma}\threshold{\dist_{[k\Lambda]}(g,h)}{\sigma} +L\]
where we define $\threshold{N}{\sigma}= N$ if $N \geq \sigma$ and $0$ if $N < \sigma$.
\end{cor}

Another key consequence of relative hierarchical hyperbolicity for a group is that the action of the group on the $\nest$--maximal space is acylindrical. Thus, we have that the action of $G_\Gamma$ on $C(\Gamma)$ is acylindrical. Recall, the action of a group $G$ on a metric space $X$ is \emph{acylindrical} if for all $\epsilon \geq 0$, there exist $R,N \geq 0$ so that if $x,y \in X$ satisfy $\dist_X(x,y) \geq R$, then there are at most $N$ elements $g \in G$ such that $\dist_X(x,gx) \leq \epsilon$  and $\dist_X(y,gy) \leq \epsilon$.

\begin{cor}[The action on $C(\Gamma)$ is acylindrical]\label{cor:acyl}
Let $G_\Gamma$ be a graph product of finitely generated groups. The action of $G_\Gamma$ on $C(\Gamma)$ by left multiplication is acylindrical.
\end{cor}

\begin{proof}
Behrstock, Hagen, and Sisto proved that if  $(G,\mf{S})$ is a (non-relative) hierarchically hyperbolic group and $T \in \mf{S}$ is the $\nest$--maximal element, then the action of $G$ on $C(T)$ is acylindrical \cite[Theorem 14.3]{BHS_HHSI}. However, the argument they employ only uses the hyperbolicity of the space $C(T)$ and not the hyperbolicity of any of the other spaces in the HHG structure. Thus, their argument carries through verbatim if $(G,\mf{S})$ is a relative HHG provided $\mf{S} \neq \{T\}$. In the case when $\mf{S} = \{T\}$, then $C(T)$ is equivariantly quasi-isometric to a Cayley graph of $G$ with respect to some finite generating set. Thus, $G$ acts on $C(T)$ properly, and hence acylindrically.  Applying this to the graph product $G_\Gamma$ with relative HHG structure $\mf{S}_\Gamma$, we have that $G_\Gamma$ acts on $C(\Gamma)$ acylindrically.
\end{proof}

\subsection{The syllable metric is an HHS}\label{section:syllable}
Since nearly every argument used in the proof of Theorem \ref{thm:graph_products_are_rel_HHGs} factors through the syllable metric on the graph product $G_\Gamma$, the same arguments show that the syllable metric on $G_{\Gamma}$ is itself a hierarchically hyperbolic space. This proves Corollary \ref{intro_cor:syllable_is_HHG} stated in the introduction and answers a question of Behrstock, Hagen, and Sisto about the syllable metric on a right-angled Artin group. Note that since we are not working with a word metric on $G_\Gamma$ in this situation,  we do not require the vertex groups to be finitely generated. As the only use of the finite generation of the vertex groups in Theorem \ref{thm:graph_products_are_rel_HHGs} is to ensure that $G_\Gamma$ has a word metric, this does not create any additional difficulty.

\begin{cor}
Let $\Gamma$ be a finite simplicial graph, with each vertex $v$ labelled by a group $G_{v}$. Then the graph product $G_{\Gamma}$ endowed with the syllable metric is a hierarchically hyperbolic space.
\end{cor}

\begin{proof}
Define the proto-hierarchy structure for $G_{\Gamma}$ as before, except whenever $v \in V(\Gamma)$ and $g \in G_{\Gamma}$, define $C(gv)$ to be the graph whose vertices are elements of $gG_{v}$ and where every pair of vertices is joined by an edge (that is, we endow $gG_{v}$ with the syllable metric rather than the word metric). The proofs of the HHG axioms then follow as before, with any instance of `word metric' replaced with `syllable metric', and with trivial $\nest$--minimal case for the majority of axioms due to such $C(gv)$ having diameter $1$.
\end{proof}

\section{Some applications of hierarchical hyperbolicity}\label{section:applications}
We now give some  applications of the relative hierarchical hyperbolicity of graph products. Our main result of this section is Theorem \ref{thm:graph_products_of_HHGs}, which shows that if the vertex groups of a graph product $G_{\Gamma}$ are HHGs, then $G_{\Gamma}$ is itself a (non-relative) HHG.

We then give a new proof of a theorem of Meier, classifying when a graph product $G_{\Gamma}$ with hyperbolic vertex groups is itself hyperbolic. We do this using the relative HHS structure that we just obtained, noting that when the vertex groups are hyperbolic, this is in fact a (non-relative) HHS structure.

Finally, we answer two questions of Genevois regarding the \emph{electrification} $\mathbb{E}(\Gamma)$ of a graph product $G_{\Gamma}$ of finite groups \cite[Questions 8.3, 8.4]{Gen_Rigidity}. 
The similarity of Genevois' definition of $\mathbb{E}(\Gamma)$ to our own subgraph metric $C(\Gamma)$ allows us to leverage properties of $C(\Gamma)$ to prove statements about $\mathbb{E}(\Gamma)$. In particular, we use $\Gamma$ to  classify when $\mathbb{E}(\Gamma)$ has bounded diameter (Theorem \ref{thm:electrification}) and when it is a quasi-line (Theorem \ref{thm:quasi-line}). As Genevois proved that any quasi-isometry between graph products of finite groups induces a quasi-isometry between their electrifications {\cite[Proposition 1.4]{Gen_Rigidity}}, these two theorems provide us with tools for studying quasi-isometric rigidity of graph products of finite groups. 

\subsection{Graph products of HHGs}\label{section:graph_products_of_HHG}

\begin{thm}\label{thm:graph_products_of_HHGs}
Let $G_\Gamma$ be a graph product of finitely generated groups. If for each $v \in V(\Gamma)$, the vertex group $G_v$ is a hierarchically hyperbolic group, then $G_\Gamma$ is a hierarchically hyperbolic group.
\end{thm}

\begin{proof}
 For each $v \in V(\Gamma)$, let $\mf{R}_{[v]}$ be the HHG structure for $G_v$ and let $\mf{S}_\Gamma$ be the relative HHG structure for $G_\Gamma$ coming from Theorem \ref{thm:graph_products_are_rel_HHGs}. Fix $E_0 >0$ to be  the maximum of the hierarchy constants for $\mf{S}_\Gamma$ and for each $\mf{R}_{[v]}$. For each $[g\Lambda] \in \mf{S}_\Gamma$, let $\overline{g}\Lambda$ be the fixed representative of $[g\Lambda]$ so that $C([g\Lambda]) = C(\overline{g}\Lambda)$. If $[g\Lambda] = [\Lambda]$, then we choose $\overline{g} = e$.

Let $\mf{S}_\Gamma^{min} = \{ [g\Lambda] \in \mf{S}_\Gamma : \Lambda \text{ is a single vertex of } \Gamma \}$.  If $\Lambda$ is a single vertex $v$ of $\Gamma$, then $C([v])$ is the Cayley graph of the vertex group $G_v$ with respect to a finite generating set. Thus, $\mf{R}_{[v]}$ is an HHG  structure for $C([v])$. For each $[gv] \in \mf{S}_\Gamma^{min}$, $\mf{R}_{[v]}$ is also an $E_0$--HHS structure for $C([gv])$, since $C([gv])$ is isometric to $C([v])$. Let $\mf{R}_{[gv]}$ denote the HHS structure for $C([gv])$ induced by $\mf{R}_{[v]}$. If $U \in \mf{R}_{[v]}$, then we will denote the corresponding element of $\mf{R}_{[gv]}$ by $\overline{g}U$ where $\overline{g}$ is the chosen fixed representative of $[gv]$. Let $\overline{\mf{R}} = \bigcup_{[gv] \in \mf{S}_\Gamma^{min}}\mf{R}_{[gv]}$, then let  $\mf{T}_0 = (\mf{S}_\Gamma \smallsetminus \mf{S}_\Gamma^{min}) \cup \overline{\mf{R}}$. 

 We shall use $\Snest$, $\Sperp$, and $\Strans$ to denote the nesting, orthogonality, and transversality relations between elements of $\mf{S}_\Gamma$, and $\Rnest$, $\Rperp$, and $\Rtrans$ to denote the relations between elements of a fixed $\mf{R}_{[gv]}$.
 
 The bulk of our proof of Theorem \ref{thm:graph_products_of_HHGs} does not use the specifics of the relative HHG structure $\mf{S}_\Gamma$ and instead relies on more general relative HHS properties. Thus, to simplify notation, we will use the capital letters $V$ or $V'$ to denote elements of $\mf{S}_{\Gamma}^{min}$ and use $\mf{R}_V$ or $\mf{R}_{V'}$ to denote the corresponding HHS structure on $C(V)$ or $C(V')$. That is, if $V = [gv]$ for a vertex $v \in V(\Gamma)$, then $\mf{R}_V = \mf{R}_{[gv]}$. We will use the capital letters $U$, $W$, and $Q$ to denote elements of $\mf{T}_{0}$. For $U,W \in \mf{S}_\Gamma \smallsetminus\mf{S}_{\Gamma}^{min}$ or $U,W \in \mf{R}_{V}$ we shall denote the relative projection from $U$ to $W$ in $\mf{S}_\Gamma$ or $\mf{R}_{V}$ as $\oldRel_W^U$.  We shall use $\oldProj_W$ to denote the projection $G_\Gamma \to 2^{C(W)}$ if $W \in \mf{S}_\Gamma$ and  $\oldProj_W^{V}$ to denote the projection $C(V) \to 2^{C(W)}$ if $W \in \mf{R}_{V}$.

Our proof of Theorem \ref{thm:graph_products_of_HHGs} proceeds via four claims. First we prove that the structure $\mf{S}_{\Gamma}$ can be combined with all of the $\mf{R}_{V}$ structures in a natural way to produce a proto-hierarchy structure for $G_\Gamma$ with index set $\mf{T}_0$ (Claim \ref{claim:graph_products_of_HHGs_proto}). This proto-hierarchy structure is not quite a hierarchically hyperbolic space structure, as it  satisfies every axiom  except the container axiom (Claim \ref{claim:graph_product_of_HHG_almost_HHS}). However, we show that this proto-hierarchy structure has the property that any set of pairwise orthogonal elements of $\mf{T}_0$ has uniformly bounded cardinality (Claim
\ref{claim:graph_products_of_HHGs_finite_rank}). This allows us to use the results of the appendix  of \cite{ABD} to upgrade $\mf{T}_0$ to a genuine HHS structure $\mf{T}$. Since the proto-structure will satisfy the equivariance properties of a hierarchically hyperbolic group structure for $G_\Gamma$ (Claim \ref{claim:graph_products_of_HHGs_group_action}), this HHS structure will also be a hierarchically hyperbolic group structure.

\begin{claim}\label{claim:graph_products_of_HHGs_proto}
$G_\Gamma$ admits an  $E_1$--proto-hierarchy structure with index set $\mf{T}_0$, where $E_{1} = E_{0}^{2}+E_{0}$.
\end{claim}

\begin{proof}
For $U \in \mf{T}_0$, the associated hyperbolic space $C(U)$ will be the same as the space associated to $U$ in either $\mf{S}_\Gamma$ or $\overline{\mf{R}}$.

 \textbf{Projections:} For all $W \in \mf{T}_0$, the projection map will be denoted $\newProj_W \colon G_\Gamma \to 2^{C(W)}$. If $W \in \mf{S}_\Gamma \smallsetminus \mf{S}_\Gamma^{min}$, then $\newProj_W = \oldProj_W$ and if $W \in \mf{R}_{V}$, then  $\newProj_W = \oldProj_W^{V} \circ \oldProj_V$. Each $\newProj_{W}$ is $(E_0^{2}, E_0^{2}+E_0)$--coarsely Lipschitz.

\textbf{Nesting:} Let $W,U \in \mf{T}_0$. We define $U\nest W$  if one of the following holds:
    \begin{itemize}
        \item  $W,U \in \mf{S}_\Gamma \smallsetminus \mf{S}_\Gamma^{min}$ and $U \Snest W$;
        \item $W,U \in \mf{R}_V$  and  $U \Rnest W$;
        \item $W \in  \mf{S}_\Gamma \smallsetminus \mf{S}_\Gamma^{min}$ and $U \in \mf{R}_{V}$ with $V \Snest W$.
    \end{itemize}
  
    This definition makes $[\Gamma]$, the $\Snest$--maximal element of $\mf{S}_\Gamma$, also the $\nest$--maximal element of $\mf{T}_0$. For  $U,W \in \mf{T}_0$ with $U \propnest W$ we denote the relative projection from $U$ to $W$ by $\newRel_W^U$ and define it as follows.
        \begin{itemize}
        \item  If $W,U \in \mf{S}_\Gamma \smallsetminus \mf{S}_\Gamma^{min}$ and  $U \Snest W$  or $W,U \in \mf{R}_{V}$  and  $U \Rnest W$, then  $\newRel_W^U$ is $\oldRel_W^U$, the relative projection from $U$ to $W$ in $\mf{S}_\Gamma$ or $\mf{R}_{V}$ respectively.
        \item If $W \in  \mf{S}_\Gamma \smallsetminus \mf{S}_\Gamma^{min}$ and $U \in \mf{R}_{V}$ with $V \Snest W$, then $\newRel_W^U$ is $\oldRel_W^{V}$,the relative projection from $V$ to $W$ in $\mf{S}_\Gamma$.
    \end{itemize}
    The diameter of $\newRel_W^U$ is bounded by $E_0$ in all cases as it is always coincides with a relative projection ($\oldRel_W^U$ or $\oldRel_W^{V}$) from an existing hierarchy structure with constant $E_0$.

    \textbf{Orthogonality:} Let $W,U \in \mf{T}_0$. We define $U\perp W$  if one of the following holds:
    \begin{itemize}
        \item  $W,U \in \mf{S}_\Gamma \smallsetminus \mf{S}_\Gamma^{min}$ and  $U \Sperp W$;
        \item $W,U \in \mf{R}_{V}$  and  $U \Rperp W$;
        \item $W \in  \mf{S}_\Gamma \smallsetminus \mf{S}_\Gamma^{min}$ and $U \in \mf{R}_{V}$ with $V \Sperp W$;
        \item $W \in \mf{R}_{V'}$ and $U \in \mf{R}_{V}$ where $V \Sperp V'$.
    \end{itemize}

   \textbf{Transversality:} Let $U, W \in \mf{T}_0$. We define $U \trans W$ whenever they are not orthogonal or nested in $\mf{T}_0$. This arises in three different situations, which  determine the definition of the relative projections $\newRel_U^W$ and $\newRel_W^U$.
    \begin{itemize}
        \item   Either $U,W \in \mf{S}_\Gamma$ or $U,W \in \mf{R}_V$ and $U \Strans W$ or $U \Rtrans W$ respectively. In this case, $\newRel_W^U$ is $\oldRel_W^U$, the relative projection from $U$ to $W$ in $\mf{S}_\Gamma$ or $\mf{R}_V$ respectively, and $\newRel^{W}_{U}$ is $\oldRel^{W}_{U}$.
        \item $W \in \mf{S}_\Gamma$ and $U \in \mf{R}_V$ where $W \Strans V$. In this case, $\newRel_W^U$ is $\oldRel_W^V$, the relative projection from $V$ to $W$ in $\mf{S}_\Gamma$, and $\newRel^{W}_{U} = \oldProj^{V}_{U}(\oldRel^{W}_{V})$.
        \item $W \in \mf{R}_{V'}$ and $U \in \mf{R}_V$ where $V \Strans V'$. In this case, $\newRel_W^U = \pi_W^{V'} ( \oldRel_{V'}^V)$ and $\newRel_U^W = \pi_U^{V} ( \oldRel_{V}^{V'})$.
     \end{itemize}
  The projection and transversality axioms of $\mf{R}_V$ and $\mf{S}_\Gamma$ ensure that $\newRel_W^U$ has diameter at most $E_0^2+E_0$ in all cases.
  \end{proof}

\begin{claim}\label{claim:graph_product_of_HHG_almost_HHS}
$\mf{T}_0$ satisfies all of the axioms of a hierarchically hyperbolic space except for the container axiom.
\end{claim}

\begin{proof} Recall, $E_1>0$ is the hierarchy constant from the proto-hierarchy structure $\mf{T}_0$. Note $E_1$ is larger than $E_0$, which in turn is larger than the hierarchy constants for $\mf{S}_\Gamma$ and  each $\mf{R}_V$.

    \textbf{Hyperbolicity:} For all $W \in \mf{T}_0$, the space $C(W)$ is $E_1$--hyperbolic.

    \textbf{Uniqueness:} Let $\kappa \geq 0$ and $\theta \colon [0,\infty) \to [0,\infty)$ be the maximum of the uniqueness functions for $\mf{S}_{\Gamma}$ and each $\mf{R}_{V}$.  If $x,y \in G_\Gamma$ and $\dist(x,y) \geq \theta(\theta(\kappa)+\kappa)$, then there exists $W \in \mf{S}_\Gamma$ such that $\dist_{W}(x,y) \geq \theta(\kappa) +\kappa$ by the uniqueness axiom in $(G_\Gamma,\mf{S}_\Gamma)$. If $W \not \in \mf{S}_\Gamma^{min}$, then $W$ is in $\mathfrak{T}_0$ and the uniqueness axiom is satisfied. If $W \in \mf{S}_\Gamma^{min}$, then the uniqueness axiom in $(C(W),\mf{R}_W)$ provides $U \in \mf{R}_W$ so that $\dist_U(x,y) \geq \kappa$. The uniqueness function for $(G_\Gamma,\mf{T}_0)$ is therefore $\phi(\kappa) = \theta(\theta(\kappa) +\kappa)$.
    
    \textbf{Finite complexity:} The length of a $\nest$--chain in $\mf{T}_0$ is at most $2E_1$. 
    
    \textbf{Bounded geodesic image:} Let $x,y \in G_\Gamma $ and $U,W \in \mathfrak{T}_0$ with $U \propnest W$. If $U,W \in \mf{S}_\Gamma$ or $U,W \in \mf{R}_V$, then the bounded geodesic image axiom from $(G_\Gamma,\mf{S}_\Gamma)$ or $(C(V),\mf{R}_V)$ implies the bounded geodesic image axiom for $(G_\Gamma,\mf{T}_0)$. 
    Suppose, therefore, that $U \in \mf{R}_V$ and $W \in \mf{S}_\Gamma \smallsetminus \mf{S}_\Gamma^{min}$.
    By definition, $V \Snest W$  and $\newRel_W^U$ coincides with $\oldRel_W^V$, the relative projection of $V$ to $W$ in $\mf{S}_\Gamma$.  If $\dist_U(x,y) >E_1^2 +E_1$, then we have 
    \begin{align*}
        E_1^2+E_1 &< \dist_{U}(x,y) \\
        &= \dist_{U}(\pi^{V}_{U}(\pi_{V}(x)),\pi^{V}_{U}(\pi_{V}(y))) \\ 
        &\leq E_1\dist_{V}(\pi_{V}(x),\pi_{V}(y)) +E_1,
    \end{align*}
    which implies $E_1 < \dist_{V}(\pi_V(x),\pi_V(y))$. 
    Now, the bounded geodesic image axiom in $(G_\Gamma,\mf{S}_\Gamma)$ says every geodesic in $C(W)$ from $\newProj_W(x) = \oldProj_W(x)$ to $\newProj_W(y) = \oldProj_W(y)$ must pass through the $E_1$--neighbourhood of $\oldRel_W^V=\newRel_W^U$. Thus, the bounded geodesic image axiom is satisfied for $(G_\Gamma,\mf{T}_0)$. 
    
    \textbf{Large links:} Let $W \in \mathfrak{T}_0$ and $x,y \in G_\Gamma$. If $W \in \mf{R}_V$ for some $V \in \mf{S}_\Gamma^{min}$, then all elements of $\mf{T}_0$ that are nested into $W$ are also elements of $\mf{R}_V$. Thus, the large links axiom in $(C(V),\mf{R}_V)$ immediately implies the large links axiom for $(G_\Gamma,\mf{T}_0)$. 
    
    Assume $W \in \mathfrak{S}_\Gamma \smallsetminus \mf{S}_\Gamma^{min}$. 
    The large links axiom for $(G_\Gamma,\mathfrak{S}_\Gamma)$ gives  a collection $\mathfrak{L} = \{U_{1}, \dots, U_{m}\}$ of elements of $\mathfrak{S}_\Gamma$ nested into $W$ such that $m$ is at most $E_1\dist_{W}(\pi_{W}(x),\pi_{W}(y))+E_1$, and for all $V \in \mathfrak{S}_W$, either $V \Snest U_{i}$ for some $i$ or $\dist_{V}(\pi_{V}(x),\pi_{V}(y)) < E_1$. 
    For each $i \in \{1,\dots,m\}$, define  $\overline{U_i}$ to be the $\Rnest$--maximal element of $\mf{R}_{U_i}$ if $U_i \in \mf{S}_\Gamma^{min}$ and define $\overline{U}_i$ to be $U_i$ if $U_i \not\in \mf{S}_{\Gamma}^{min}$. Let $\overline{ \mf{L}} = \{\overline{U_1}, \dots, \overline{U_m}\}$.  
    
    If $V \in \mf{S}_{\Gamma}^{min}$ is nested into $W$, but is not nested into an element of  $\mf{L}$, then $\dist_{V}(\oldProj_{V}(x),\oldProj_{V}(y)) < E_1$ and so \[\dist_{Q}(\newProj_{Q}(x),\newProj_{Q}(y)) < E_1^2+E_1\] for all $Q \in \mf{R}_{V}$. Thus, if $\dist_{Q}(\newProj_{Q}(x),\newProj_{Q}(y)) \geq E_1^2+E_1$ and $Q$ is nested into $W$, then either $Q \in \mf{S}_{\Gamma} \smallsetminus \mf{S}_{\Gamma}^{min}$ or $Q \in \mf{R}_V$ where  $V$ is nested into an element of $\mf{L}$ (and so $Q$ is nested into an element of $\overline{\mf{L}}$). If $Q \in \mf{S}_\Gamma \smallsetminus \mf{S}_\Gamma^{min}$, then $Q$ must be nested into an element of $\mf{L}$ that is not in $\mf{S}_\Gamma^{min}$ by the large links axiom of $(G_\Gamma,\mf{S}_\Gamma)$, and hence must be nested into an element of $\overline{\mf{L}}$. Thus, $Q \nest W$ is nested into an element of $\overline{\mf{L}}$ whenever $\dist_{Q}(\newProj_Q(x),\newProj_Q(y)) \geq E_1^2+E_1$. 
    
   \textbf{Consistency:} Let $U,W \in \mf{T}_0$ with $U \trans W$ and $x \in G_\Gamma$. Since the relative projections are inherited from $\mf{S}_\Gamma$ and the $\mf{R}_V$, we only need to consider the case where either $W \in \mf{S}_\Gamma$ and $U \in \mf{R}_V$, or $W \in \mf{R}_{V'}$ and $U \in \mf{R}_V$ with $V'\neq V$. 
    Define $ Q = W$ if $W \in \mf{S}_\Gamma$ and  $Q=V'$ if $W \in \mf{R}_{V'}$. In either case  $Q \Strans V$. 
    
    First assume  $Q = W$ so that $\newRel_W^U = \oldRel_Q^V$ and $\newRel_U^W = \oldProj^V_U(\oldRel_V^Q)$. If $\dist_W(x,\newRel_W^U) = \dist_Q(x,\oldRel_Q^V) > E_1$, then the consistency axiom for $(G_\Gamma,\mf{S}_\Gamma)$ says $\dist_V(x,\oldRel_V^Q)\leq E_1$. The coarse Lipschitzness of the projections then implies $\dist_U(x, \oldProj^V_U(\oldRel_V^Q) ) = \dist_U(x,\newRel_U^W) \leq E_1^2 +E_1$. 
    
    Now assume $Q =V'$ so that $\newRel_W^U = \oldProj_{W}^{Q}(\oldRel_Q^V)$ and  $\newRel_U^W = \oldProj_U^{V}(\oldRel_V^Q)$.
    If  $\dist_W(x, \newRel_W^U) > E_1^2 +E_1$, then $\dist_{Q}(x,\oldRel_Q^V) > E_1$. The consistency axiom for $(G_\Gamma,\mf{S}_\Gamma)$ then says $\dist_V(x,\oldRel_V^Q)\leq E_1$ and we again have $$\dist_U(x,\newRel_U^W) = \dist_V(x, \oldProj_U^V(\oldRel_V^Q))\leq E_1^2 +E_1.$$
    For the last clause of the consistency axiom, let $Q,U,W \in \mf{T}_0$ with $Q \propnest U$. If $U\propnest W$, the definition of nesting and relative projection in $\mf{T}_0$ and the consistency axioms in $(G_\Gamma,\mf{S}_\Gamma)$ and the  $(C(V),\mf{R}_V)$ ensure that $\dist_W(\newRel_W^Q,\newRel_W^U) \leq E_1^{2}+E_1$. Similarly, if $W \in \mf{S}_\Gamma$ with $W \trans U$ and $W \not \perp Q$, then $\dist_W(\newRel_W^Q,\newRel_W^U) \leq E_1^{2}+E_1$. Assume $W \in \mf{R}_V$ for some $V \in \mf{S}_\Gamma^{min}$, $W \trans U$, and $W \not \perp Q$. If $U,Q \in \mf{R}_{V'}$, then $V' \Strans V$ and $\newRel_W^U = \newRel_W^Q$. If $U,Q \in \mf{S}_{\Gamma}$, then  $U \Strans V$ and  $Q \Strans V$. Thus, the consistency axiom for $(G_\Gamma,\mf{S}_\Gamma)$ provides $\dist_{V}(\oldRel_V^U, \oldRel_V^Q ) \leq E_1$. Similarly, if $U \in \mf{S}_\Gamma$ and $Q \in \mf{R}_{V'}$, then $U \Strans V$,  $V' \Strans V$, and $\dist_{V}(\oldRel_V^U, \oldRel_V^{V'} ) \leq E_1$.  Hence in both cases $\dist_W(\newRel_W^U,\newRel_W^Q) \leq E_1^2+E_1$.
    
    \textbf{Partial realisation:} Let $W_1,\dots, W_n$ be pairwise orthogonal elements of $\mf{T}_0$ and $p_i \in C(W_i)$ for each $i \in \{1,\dots,n\}$. Since $(G_\Gamma,\mf{S}_\Gamma)$ satisfies the partial realisation axiom, we can assume at least one $W_i$ is not an element of $\mf{S}_\Gamma$. There  exist $V_1,\dots,V_r \in \mf{S}_\Gamma^{min}$ so that for each $i \in \{1,\dots,n\}$, either $W_i \in \mf{S}_\Gamma$ or there exists a unique $j \in \{1,\dots, r\}$ such that $W_i \in \mf{R}_{V_j}$. For each $j \in \{1,\dots ,r\}$, let $\{W_1^j,\dots, W_{k_j}^j\}$ be the elements of $\{W_1,\dots, W_{n}\}$ that are also elements of $\mf{R}_{V_j}$ and let $\{p^j_1,\dots, p^j_{k_j}\}$ be the subset of $\{p_1,\dots,p_n\}$ satisfying $p_i^j \in C(W_i^j)$ for all $j \in \{1,\dots,r\}$ and $i \in \{1, \dots, k_j\}$.  For each $j\in \{1,\dots,r\}$,  use partial realisation in $(C(V_j),\mf{R}_{V_j})$  on the points $p^j_1,\dots, p^j_{k_j}$ to produce a  point $y_j \in C(V_j)$ so that: 
 
 \begin{itemize}
 \item $\dist_{W_i^j}(y_j,p^j_{i})\leq E_1$ for all $i \in \{1,\dots,k_j\}$;
 \item for each $i \in \{1,\dots, k_j\}$ and 
 each  $U\in \mf{R}_{V_j}$, if $W^j_i\propnest U$ or $W_j^i\trans U$, we have 
 $\dist_{U}(y_j,\oldRel^{W_i^j}_{U})\leq E_1$.
  \end{itemize}
 
Assume, without loss of generality, that $W_m,W_{m+1},\dots,W_n$ are all of the $W_i$ that are not contained in any of the $\mf{R}_{V_j}$ (it is possible the set of such $W_i$ is empty). Now, applying partial realisation for $(G_\Gamma,\mf{S}_\Gamma)$ to $y_1,\dots,y_r, p_m,\dots,p_n$ produces  a point $x \in G_\Gamma$ so that $\psi_{W_{i}}(x)$  is uniformly close to $p_i$ for each $i \in \{1,\dots,n\}$ and $\psi_{U}(x)$ is uniformly close to $\newRel_{U}^{W_i}$ whenever $ W_i \propnest U$ or $U \trans W_i$, for any $U \in \mf{T}_0$. Note, if the set of $W_i$ that are not elements of any of the $\mf{R}_{V_j}$ is empty, then the above applies just to $y_1,\dots, y_r$, but the conclusion still holds.
\end{proof}

\begin{claim}\label{claim:graph_products_of_HHGs_finite_rank}
The $E_1$--proto-hierarchy structure $\mf{T}_0$ has the following property: if $W_1,\dots,W_n \in \mf{T}_0$ are pairwise orthogonal, then $n \leq  4E_1^2+2E_1$.
\end{claim}

\begin{proof}
We first note the following basic lemma from the theory of hierarchically hyperbolic spaces.
\begin{lem}[{\cite[Lemma 1.5]{HHS_Boundary}}] \label{lem:finite_rank}
If $(\mc{X},\mf{S})$ is an $E$--HHS, then any set of pairwise orthogonal elements of $\mf{S}$ has cardinality at most $2E$.
\end{lem}

Now, let $W_1,\dots, W_n \in \mf{T}_0$ be pairwise orthogonal. Without loss of generality, let $W_1,\dots W_k$ be the elements of $\{W_1,\dots, W_n\}$ that are elements of $\mf{S}_\Gamma$. Since $W_1,\dots,W_k$ is a pairwise orthogonal collection of elements of $\mf{S}_\Gamma$, Lemma \ref{lem:finite_rank} says $k \leq 2E_1$.

Let $V_1,\dots, V_m$ be the minimal collection of elements of $\mf{S}_\Gamma^{min}$ such that if $i \in \{k+1,\dots,n\}$ (i.e. $W_i \not \in \mf{S}_\Gamma$), then $W_i \in \mf{R}_{V_j}$ for some $j \in \{1,\dots,m\}$. Minimality implies that for each $j \in \{1,\dots,m\}$, there exists $i \in \{k+1,\dots,n\}$ such that $W_i \in \mf{R}_{V_j}$. Suppose  $W_i \in \mf{R}_{V_j}$ and $W_\ell \in \mf{R}_{V_r}$ with $j \neq r$. Since $W_i \perp W_\ell$ in $\mf{T}_0$, the definition of orthogonality in $\mf{T}_0$ implies that $V_j \Sperp V_r$. Thus, $V_1,\dots, V_m$ is a pairwise orthogonal collection of {elements of} $\mf{S}_\Gamma$ and $m \leq 2E_1$ by Lemma \ref{lem:finite_rank}. Similarly, for each $j \in \{1,\dots,m\}$ the set $\{W_i : W_i \in \mf{R}_{V_j}\}$  is a pairwise orthogonal collection of {elements of} $\mf{R}_{V_j}$ and must have cardinality at most $2E_1$. Putting this together, we have that $n \leq k + 2E_1 m \leq 2E_1 + 4E_1^2$.
\end{proof}

\begin{claim}\label{claim:graph_products_of_HHGs_group_action}
The action of $G_\Gamma$ on $\mf{S}_\Gamma$ induces an action of $G_\Gamma$ on $\mf{T}_0$ that satisfies axioms (\ref{item:HHG_index_action}) and (\ref{item:HHG conditions}) of the definition of a hierarchically hyperbolic group (Definition \ref{defn:hierarchically hyperbolic groups}).
\end{claim}

\begin{proof}
\textbf{The action of $G_\Gamma$ on $\mf{T}_{0}$:} Let $\sigma \in G_\Gamma$ and $W \in \mf{T}_0$. Define $\Phi \colon G \times \mf{T}_{0} \to \mf{T}_{0}$ as follows.
\begin{itemize}
    \item If $W=[g\Lambda] \in \mf{S}_\Gamma \smallsetminus \mf{S}_\Gamma^{min}$, then $\Phi(\sigma,[g\Lambda]) = [\sigma g\Lambda]$, i.e., the action is the same as the action of $G_\Gamma$ on $\mf{S}_\Gamma$.
    \item \label{item:peripheral_action_graph_product} If $W= \overline{g}R \in \mf{R}_{[gv]}$ for some $[gv] \in \mf{S}_\Gamma^{min}$, then $(\overline{\sigma g})^{-1} \sigma \overline{g} \in \Stab_{G_\Gamma}([v]) $, where $\overline{\sigma g}$ is the chosen fixed  representative of $[\sigma g v] = [\sigma\overline{g} v] $. Since $\Stab_{G_\Gamma}([v]) = \sub{\st(v)}$, there exists $l \in \sub{\lk(v)}$ and $\hat{\sigma} \in \sub{v}$ such that $ l \hat{\sigma} = (\overline{\sigma g})^{-1} \sigma \overline{g}$. Because $\mf{R}_{[v]}$ is an HHG structure for $\sub{v}=G_v$ there exists $R_\sigma = \hat{\sigma} R \in \mf{R}_{[v]}$ determined by $\sigma$ and $\overline{g}R$.
    Define $\Phi(\sigma, \overline{g}R) = {\overline{\sigma g} R_\sigma \in \mf{R}_{[\sigma g v]}}$.  The following  diagram summarises how $\sigma$ takes elements of $\mf{R}_{[gv]}$ to elements of $\mf{R}_{[\sigma g v]}$. 
\end{itemize}

\[
\begin{tikzcd}
 \mf{R}_{[gv]}   \arrow[d, "\overline{g}^{-1}" left] \arrow[r, "\sigma"]  & \mf{R}_{[\sigma gv]}\\
 \mf{R}_{[v]} \arrow[ur, "\hspace{.5cm}  \overline{\sigma g}" below ] \arrow[loop left, "\hat{\sigma}"] &
\end{tikzcd}
\]

We now verify that $\Phi$ preserves the relations in $\mf{T}_0$. Let $W,U \in \mf{T}_0$. If $W,U \in \mf{S}_\Gamma \smallsetminus \mf{S}_\Gamma^{min}$ or $W,U \in \mf{R}_{[gv]}$ for some $[gv] \in \mf{S}_\Gamma^{min}$, then $\Phi$ preserves the relation between $W$ and $U$, since the actions of $G_\Gamma$ on $\mf{S}_\Gamma$ and $G_v = \sub{v}$ on $\mf{R}_{[v]}$ preserve the relations in their respective hierarchy structures.  If $W \in \mf{S}_\Gamma \smallsetminus \mf{S}_\Gamma^{min}$ and $U \in \mf{R}_{[gv]}$, then $W = [h\Omega]$ and the relation between $W$ and $U$ in $\mf{T}_0$ is the same as the relation between $[h\Omega]$ and $[gv]$ in $\mf{S}_\Gamma$. Thus, $\Phi$  preserves the relation between $W$ and $U$, since the action of $G_\Gamma$ preserves the relations in $\mf{S}_\Gamma$. Similarly, the same is true in the case where $W \in \mf{R}_{[gv]}$ and $U \in \mf{R}_{[hw]}$ for $[gv] \neq [hw]$ as the relation between $W$ and $U$ in $\mf{T}_0$ is the same as the relation between $[gv]$ and $[hw]$ in $\mf{S}_\Gamma$.

The definition of $\Phi$ implies that $\overline{g} R \in \mf{R}_{[gv]}$ is in the $G_\Gamma$--orbit of $\overline{h} R' \in  \mf{R}_{[hw]}$ if and only if  $v = w$ and $R$ is in the $G_v$--orbit of $R'$. Thus, the action of $G_{\Gamma}$ on $\mf{T}_0$ has finitely many orbits since the actions of $G_\Gamma$ on $\mf{S}_\Gamma$ and $G_v$ on $\mf{R}_{[v]}$ contain finitely many orbits.

For the remainder of the proof we shall use $\sigma W$ to denote $\Phi(\sigma, W)$ for all $W \in \mf{T}_0$. This does not conflict with previous use of the notation as the action of $G_\Gamma$ on $\mf{T}_0$ agrees with the action of $G_\Gamma$ on $\mf{S}_\Gamma$ or the action of $G_v$ on $\mf{R}_{[v]}$, when $W  \in \mf{S}_\Gamma$ or $\sigma \in \sub{v}$ and $W \in \mf{R}_{[v]}$ respectively.

\textbf{Associated isometries and equivariance with the projection maps:} Let $\sigma,\tau \in G_\Gamma$ and $W \in \mf{T}_{0}$.   Since the action of $G_\Gamma$ on $\mf{T}_0$ agrees with the action of $G_\Gamma$ on $\mf{S}_\Gamma$ for the elements of $\mf{T}_0$ in $\mf{S}_\Gamma$, we can define the isometry $\sigma_{[g\Lambda]} \colon C([g\Lambda]) \to C([\sigma g \Lambda])$ to be the same as the original isometry in $(G_\Gamma,\mf{S}_\Gamma)$; this guarantees the HHG axioms are satisfied in this case.

If $W \in \mf{R}_{[gv]}$, then $W = \overline{g} R$ for some $R \in \mf{R}_{[v]}$.   Now $\sigma W = \overline{ \sigma g} R_\sigma$, where $R_\sigma$ is defined as above. In this case, define the isometry $\sigma_W \colon C(W) \to C(\sigma W)$ to be the composition \[C(W) \xrightarrow{(\overline{g}_R)^{-1}} C(R) \xrightarrow{\hat{\sigma}_R} C(R_\sigma) \xrightarrow{\overline{ \sigma g}_{R_\sigma}} C(\sigma W) \]
where $\hat{\sigma}_R \colon C(R) \to C(R_\sigma)$ is the isometry in $\mf{R}_{[v]}$ induced by $\hat{\sigma} \in G_v$, and $\overline{g}_R$ and $\overline{\sigma g}_{R_\sigma}$ are the isometries resulting from identifying $\mf{R}_{[v]}$ with $\mf{R}_{[gv]}$ and $\mf{R}_{[\sigma g v]}$ respectively.

Now, if $\tau \in G_{\Gamma}$, then $(G_v, \mf{R}_{[v]})$ being an HHG implies $\hat{\tau}_{R_{\sigma}} \circ \hat{\sigma}_{R} = \widehat{\tau\sigma}_{R}$.    Thus the isometry $(\tau\sigma)_W$ equals the isometry $\tau_{\sigma W} \circ \sigma_W$ for any $W \in \mf{T}_0$.  We continue to use the notation set out before Claim \ref{claim:graph_products_of_HHGs_proto}: $\newProj_*$ and $\newRel_*^*$ denote the projections and relative projections in $\mf{T}_0$, while $\oldProj_*^*$ and $\oldRel_*^*$ denote the projections and relative projections in $\mf{S}_\Gamma$ and $\mf{R}_{[gv]}$. Since the projection map $\newProj_W \colon G_\Gamma \to 2^{C(W)}$ is equal to   $\oldProj^{[gv]}_{W} \circ \pi_{[gv]}$, the uniform bound on the distance between $\psi_{\sigma W}(\sigma x)$ and $\sigma_W(\psi_{W}(x))$ follows from the HHG axioms of $(G_\Gamma,\mf{S}_\Gamma)$ and $(G_v,\mf{R}_{[v]})$. Similarly, since the relative projection $\newRel_W^U$ (where $U \propnest W$ or $U \trans W$ in $\mf{T}_0$) is defined using the coarsely equivariant projections and relative projections of $\mf{S}_\Gamma$ and $\mf{R}_{[v]}$, we have that $\sigma_W( \newRel_W^U)$ is uniformly close to  $\newRel^{\sigma U}_{\sigma W}$ whenever  $U \propnest W$ or $U \trans W$ .
\end{proof}

 We now finish the proof of Theorem \ref{thm:graph_products_of_HHGs} using the following result. 

\begin{thm}[{\cite[Theorem A.1]{ABD}}]\label{thm:almost_HHGs_are_HHGs}
Let $G$ be a finitely generated group and let $\mf{T}_0$ be a proto-hierarchy structure for the Cayley graph of $G$ with respect to some finite generating set. If $\mf{T}_0$ satisfies the following:
\begin{itemize}
    \item all of the axioms of a hierarchically hyperbolic space except the container axiom;
    \item any set of pairwise orthogonal elements of $\mf{T}_0$ has uniformly bounded cardinality;
    \item axioms (\ref{item:HHG_index_action}) and (\ref{item:HHG conditions}) of a hierarchically hyperbolic group structure (Definition \ref{defn:hierarchically hyperbolic groups});
\end{itemize}
then there exists a hierarchically hyperbolic group structure $\mf{T}$ for the group $G$ such that $\mf{T}_0 \subsetneq \mf{T}$ and for all $W \in \mf{T} \smallsetminus \mf{T}_0$, the associated hyperbolic space $C(W)$ is a single point.
\end{thm}

Claims \ref{claim:graph_product_of_HHG_almost_HHS}, \ref{claim:graph_products_of_HHGs_finite_rank}, and \ref{claim:graph_products_of_HHGs_group_action} show that the proto-hierarchy structure $\mf{T}_0$ satisfies the requirements of Theorem \ref{thm:almost_HHGs_are_HHGs}. Thus, there exists an HHG structure $\mf{T}$ for $G_\Gamma$.
\end{proof}

\begin{remark}[The HHG structure from Theorem \ref{thm:almost_HHGs_are_HHGs}]
The proof of Theorem \ref{thm:almost_HHGs_are_HHGs} produces an explicit HHG structure given the proto-structure $\mf{T}_0$. We will describe that structure briefly now, and direct the reader to the appendix of \cite{ABD} for full details.

Let $\mc{U}$ denote a non-empty set of pairwise orthogonal elements of $\mf{T}_0$ and let $W \in\mf{T}_0$. We say the pair $(W,\mc{U})$ is a \emph{container pair} if the following are satisfied:
\begin{itemize}
    \item $U
\nest W$  for all $U \in \mc{U}$;
    \item there exists $Q \nest W$ such that $Q \perp U$ for all $U \in \mc{U}$.
\end{itemize}
Let $\mf{D}$ denote  the set of all container pairs. We will denote a pair $(W,\mc{U}) \in \mf{D}$ by $D_W^\mc{U}$. The crux of   Theorem \ref{thm:almost_HHGs_are_HHGs} is that the elements of $\mf{D}$ will serve as containers for the elements of $\mf{T}_0$, while the rest of the proto-structure is set up in the minimal way  that satisfies all the other axioms.

The HHG structure produced by Theorem \ref{thm:almost_HHGs_are_HHGs} has index set $\mf{T}_0 \cup \mf{D}$. The hyperbolic spaces, projection maps, relations, and relative projections for elements of $\mf{T}_0$ remain unchanged. The hyperbolic spaces for elements of $\mf{D}$ are single points and the projection maps are the constant maps to these points.  The nesting relation involving elements of $\mf{D}$ is defined as follows. 
    \begin{itemize}
        \item     Define $Q \nest D_W^\mc{U}$ if $Q \nest W$ in $\mf{T}_0$ and $Q \perp U$ for all $U \in \mc{U}$. 
        \item  Define  $D_W^\mc{U} \nest Q$ if  $W \nest Q$ in $\mf{T}_0$. 
        \item  Define $D_W^\mc{U} \nest D_T^\mc{R}$ if $W \nest T$ in $\mf{T}_0$  and for all $R \in\mc{R}$ either $R \perp W$ or there exists $U \in \mc{U}$ with $R \nest U$.
    \end{itemize}

Two elements $D_W^\mc{U}, D_T^\mc{R} \in \mf{D}$ are orthogonal if $W \perp T$ in $\mf{T}_{0}$.    An element $Q \in \mf{T}_0$ is orthogonal to  $D_W^\mc{U} \in \mf{D}$ if, in $\mf{T}_{0}$, either $W \perp Q$  or $Q \nest U$ for some $U \in \mc{U}$. Two elements of $\mf{T}$ are transverse if they are not orthogonal and neither is nested into the other.

Since the associated hyperbolic spaces for elements of $\mf{D}$ are single points, the relative projections onto  these elements are just these single points. If $D_W^\mc{U} \propnest Q$ or $Q \trans D_W^\mc{U}$, then the relative projection $\rho_Q^{D_W^\mc{U}}$ is defined in one of two ways.
\begin{enumerate}
    \item If there exists $U \in \mc{U}$ such that $U \propnest Q$ or $U \trans Q$, then $\rho_Q^{D_W^\mc{U}}$ is the union of all $\rho_Q^U$ for $U \in \mc{U}$ with $U \propnest Q$ or $U \trans Q$.
    \item If there does not exist $U \in \mc{U}$ such that $U \propnest Q$ or $U \trans Q$, then the definition of the relations given above forces  $Q \trans D_W^{\mc{U}}$ and $W \trans Q$. In this case $\rho_Q^{ D_W^{\mc{U}}}=\rho_Q^W$. 
\end{enumerate}
\end{remark}

\subsection{Meier's condition for hyperbolicity}\label{section:hyperbolic_graph_products}
We now recover a theorem of Meier classifying hyperbolicity of graph products. We do this by applying Behrstock, Hagen and Sisto's bounded orthogonality condition for hierarchically hyperbolic spaces.

\begin{thm}[{\cite[Corollary 2.16]{BHS_HHS_Quasiflats}}]
Let $(\mathcal{X},\mathfrak{S})$ be an HHS. The following are equivalent: 
\begin{itemize}
    \item $\mathcal{X}$ is hyperbolic.
    \item (Bounded orthogonality.)  There exists a constant $D \geq 0$  such that \[\min(\diam(C(U)),\diam(C(V))) \leq D\] for all $U,V \in \mathfrak{S}$ satisfying $U \bot V$.
\end{itemize}
\end{thm}

\begin{thm}[Meier's criterion for hyperbolicity of graph products; \cite{Meier}]
Let $\Gamma$ be a finite simplicial graph with hyperbolic groups associated to its vertices. Let $\Gamma_{F}$ be the subgraph spanned by the vertices associated with finite groups. Then $G_{\Gamma}$ is hyperbolic if and only if the following conditions hold. 
\begin{enumerate}[(i)]
    \item There are no edges connecting two vertices of $\Gamma \smallsetminus \Gamma_{F}$.  \label{item:meier_1}
    \item If $v$ is a vertex of $\Gamma \smallsetminus \Gamma_{F}$ then $\lk(v)$ is a complete graph. \label{item:meier_2}
    \item $\Gamma_{F}$ does not contain any induced squares. \label{item:meier_3}
\end{enumerate}
\end{thm}

\begin{proof}
We show hyperbolicity via the bounded orthogonality condition, noting that since each of the vertex groups is hyperbolic, the graph product $G_{\Gamma}$ is an HHS. We call the vertices of $\Gamma_F$ the \emph{finite vertices} of $\Gamma$ and the vertices of $\Gamma \smallsetminus \Gamma_F$ the \emph{infinite vertices} of $\Gamma$.

($\Rightarrow$) Suppose we have bounded orthogonality. Then: 
\begin{enumerate}[(i)]
    \item Suppose two infinite vertices $v,w$ are connected by an edge. Then $[v] \bot [w]$ and $C(v),C(w)$ have infinite diameter as they are the infinite groups $G_{v},G_{w}$ with the word metric. This contradicts bounded orthogonality.
    \item Suppose $\lk(v)$ is incomplete for some vertex $v$ of $\Gamma \smallsetminus \Gamma_{F}$. 
    Then there exist some vertices $x,y$ in $\lk(v)$ with no edge between them. Moreover, $[v] \bot [x \cup y]$, $C(v)$ has infinite diameter as $v$ is an infinite vertex, and $C(x \cup y)$ has infinite diameter since $\dist_{x \cup y}(e,(g_{x}g_{y})^{n}) = 2n$ for elements $g_{x} \in G_{x}\smallsetminus\{e\}, g_{y} \in G_{y}\smallsetminus\{e\}$. 
    This again contradicts bounded orthogonality.
    \item Suppose $\Gamma_{F}$ contains a square with vertices $v,x,w,y$, where $v,w$ and $x,y$ are non-adjacent. Then $[v \cup w] \bot [x \cup y]$ and both $C(v \cup w)$ and $C(x \cup y)$ have infinite diameter as in case (ii). Once again, this contradicts bounded orthogonality.
\end{enumerate}

($\Leftarrow$) Conversely, suppose conditions (\ref{item:meier_1}), (\ref{item:meier_2}), and (\ref{item:meier_3}) are satisfied, and set   $D = \max\{2,|G_{v}| : v \in V(\Gamma_{F})\}.$ Moreover, suppose $[g\Lambda],[h\Omega] \in \mathfrak{S}$ satisfy $[g\Lambda] \bot [h\Omega]$.

Suppose $\diam(C(g\Lambda))>D$. Then Theorem \ref{thm:diameter} tells us that either $\Lambda$ consists of a single infinite vertex or $\Lambda$ contains at least $2$ vertices and does not split as a join. 

If $\Lambda$ consists of a single infinite vertex, then conditions (\ref{item:meier_1}) and (\ref{item:meier_2}) tell us that $\lk(\Lambda) \supseteq \Omega$ is a complete graph consisting of finite vertices, hence either $\Omega$ is a single finite vertex or $\Omega$ splits as a join. In both cases, $\diam(C(h\Omega)) \leq D$. 

If $\Lambda$ contains at least $2$ vertices and does not split as a join, then in particular it contains two non-adjacent vertices $v$ and $w$. As $\Omega \subseteq \lk(\Lambda)$, every vertex of $\Omega$ is connected to both $v$ and $w$. Since $v$ and $w$ are non-adjacent, condition (\ref{item:meier_2}) implies $\Omega \subseteq \Gamma_{F}$. If either $v$ or $w$ is an infinite vertex, condition (\ref{item:meier_2}) implies $\Omega$ is a complete graph, and if both $v$ and $w$ are finite vertices, condition (\ref{item:meier_3}) implies $\Omega$ is a complete graph. That is, $\Omega$ either consists of a single finite vertex or splits as a join. In both cases, $\diam(C(h\Omega)) \leq D$. Thus the bounded orthogonality condition holds.
\end{proof}

\subsection{Genevois' minsquare electrification.} \label{section:electrification}
 We now use our characterisation of when $C(g\Lambda)$ has infinite diameter (Theorem \ref{thm:diameter}) to answer two questions of Genevois \cite[Questions 8.3, 8.4]{Gen_Rigidity} regarding the \emph{electrification} of $G_{\Gamma}$, defined as follows.

\begin{defn}
Let $\Gamma$ be a simplicial graph. An induced subgraph $\Lambda \subseteq \Gamma$ is called \emph{square-complete} if every induced square in $\Gamma$ sharing two non-adjacent vertices with $\Lambda$ is a subgraph of $\Lambda$. A subgraph is \emph{minsquare} if it is a minimal square-complete subgraph containing at least one induced square.

The \emph{electrification} $\mathbb{E}(\Gamma)$ of a graph product $G_{\Gamma}$ is the graph whose vertices are elements of $G_{\Gamma}$ and where two vertices $g$ and $h$ are joined by an edge if $g^{-1}h$ is an element of a vertex group or $g^{-1}h \in \langle\Lambda\rangle$ for some minsquare subgraph $\Lambda$ of $\Gamma$. We use $\dist_{\mathbb{E}}(g,h)$ to denote the distance  in $\mathbb{E}(\Gamma)$ between $g,h \in G_\Gamma$.
\end{defn}

Genevois' interest in the electrification arises from the fact that it forms a quasi-isometry invariant whenever the vertex groups of a graph product are all finite, as is the case for right-angled Coxeter groups.
\begin{thm}[{\cite[Proposition 1.4]{Gen_Rigidity}}]\label{thm:electrification_is_QI_invariant}
Let $G_\Gamma$ and $G_\Lambda$ be graph products of finite groups. Any quasi-isometry  $G_\Gamma \to G_\Lambda$ induces a quasi-isometry between $\mathbb{E}(\Gamma)$ and $\mathbb{E}(\Lambda)$.
\end{thm}

For graph products of finite groups, we classify when $\mathbb{E}(\Gamma)$ has bounded diameter and when $\mathbb{E}(\Gamma)$ is a quasi-line. 
These classifications answer Questions 8.3 and 8.4 of \cite{Gen_Rigidity} in the affirmative. The core idea behind both proofs is the same: when $\Gamma$ is not minsquare,  the electrification $\mathbb{E}(\Gamma)$  sits between the syllable metric $S(\Gamma)$ and the subgraph metric $C(\Gamma)$, that is,  we obtain $\mathbb{E}(\Gamma)$ from $S(\Gamma)$ by adding edges and then obtain $C(\Gamma)$ from $\mathbb{E}(\Gamma)$ by adding more edges. 
This means large distances in $C(\Gamma)$, which we can  detect with Theorem \ref{thm:diameter}, will persist in $\mathbb{E}(\Gamma)$. We start with a lemma that we use in both classifications to reduce to the case where $\Gamma$ does not split as a join.

\begin{lem}\label{lem:joins_with_minsquare}
If $\Gamma$ splits as a  join and contains a proper minsquare subgraph, then $\Gamma$ splits as a join $\Gamma = \Gamma_1 \bowtie \Gamma_2$ where $\Gamma_1$ contains every minsquare subgraph of $\Gamma$ and $\Gamma_2$ is a complete graph. In this case, $\mathbb{E}(\Gamma)$ is the $1$--skeleton of $\mathbb{E}(\Gamma_1) \times \mathbb{E}(\Gamma_2)$.
\end{lem}

\begin{proof}
Suppose  $\Gamma$ contains a proper minsquare subgraph $\Lambda$ and splits as a  join $\Gamma = \Omega_1 \bowtie \Omega_2$. We first show $\Gamma$ splits as a (possibly different) join $\Gamma_1 \bowtie \Gamma _2$, where $\Gamma_1$ contains the minsquare subgraph $\Lambda$. If $\Lambda$ is a subgraph of either $\Omega_1$ or $\Omega_2$ we are done. Otherwise, $\Lambda$ contains vertices of both $\Omega_{1}$ and $\Omega_{2}$. By minimality of $\Lambda$, there must exist a square of $\Lambda$ containing vertices of both $\Omega_{1}$ and $\Omega_{2}$. Moreover, since $\Omega_{1}$ and $\Omega_{2}$ form a join, this square must arise in the form of two pairs of disjoint vertices $v_{i},w_{i} \in V(\Omega_{i})$, $i=1,2$. Then any vertex $v$ of $\Omega_{1} \smallsetminus \Lambda$ must be connected to every vertex $w$ of $\Lambda \cap \Omega_{1}$, else $v,w,v_{2},w_{2}$ form an induced square, contradicting square-completeness of $\Lambda$. Similarly, any vertex of $\Omega_{2} \smallsetminus \Lambda$ must be connected to every vertex of $\Lambda \cap \Omega_{2}$. This then gives a decomposition of $\Gamma$ as a join of the minsquare subgraph $\Lambda$ and the graph $\Gamma \smallsetminus \Lambda$.

We have  shown that $\Gamma$ splits as a join $\Gamma_1 \bowtie \Gamma_2$ with $\Lambda  \subseteq \Gamma_1$. We now show that $\Gamma_2$ must be a complete graph. Since $\Lambda$ is  minsquare, there exists an induced square $S$ in $\Lambda \subseteq\Gamma_{1}$. Let $v_{1}, w_{1}$ be two disjoint vertices of $S$, and suppose there exists a pair of disjoint vertices $v_{2}, w_{2}$ in $\Gamma_{2}$. Since $\Gamma$ is a join of $\Gamma_1$ and $\Gamma_2$ and $\Lambda \subseteq \Gamma_1$, the vertices $v_{1}, w_{1}, v_{2}, w_{2}$, define an induced square that shares two opposite vertices with $\Lambda$, but is not contained in $\Lambda$. This would contradict square-completeness of $\Lambda$. Therefore, $\Gamma_{2}$ must be complete. 

Finally we show that every other minsquare subgraph of $\Gamma$ must also be contained in $\Gamma_1$. Let $\Omega \subseteq \Gamma$ be minsquare. If four vertices $v_1,v_2,v_3,v_4$ of $\Omega$ form an induced square of $\Gamma$, then each $v_i$ must be contained in $\Gamma_1$, since any $v_{i}$ that $\Gamma_{2}$ contains must be connected to all $v_{j}$ in $\Gamma_{1}$, but $\Gamma_{2}$ cannot contain a pair of disjoint vertices since it is complete. Thus the minimality of $\Omega$ implies $\Omega$ must be contained in $\Gamma_1$ (otherwise $\Omega \cap \Gamma_{1}$ would be a proper square-complete subgraph of $\Omega$).

Since $\Gamma$ splits as a join $\Gamma_1 \bowtie \Gamma_2$, it follows that $S(\Gamma)$ is the $1$--skeleton of $S(\Gamma_1) \times S(\Gamma_2)$ and since the only minsquare subgraphs of $\Gamma$ are the minsquare subgraphs of $\Gamma_1$, $\mathbb{E}(\Gamma)$ is the $1$--skeleton  of $\mathbb{E}(\Gamma_1) \times \mathbb{E}(\Gamma_2)$ by construction.
\end{proof}

We now show that $\mathbb{E}(\Gamma)$ is bounded only in the obvious cases.

\begin{thm}\label{thm:electrification}
The electrification $\mathbb{E}(\Gamma)$ is bounded if and only if $\Gamma$  either is minsquare, complete, or splits as a join of a minsquare subgraph and a complete graph.
\end{thm}

\begin{proof}
We first show that if $\Gamma$ is minsquare, complete, or splits as the join of a minsquare subgraph and a complete graph then the electrification is bounded.
If $\Gamma$ is minsquare, then $\mathbb{E}(\Gamma)$ has diameter $1$ by definition. Let $x, y$ be vertices of $\mathbb{E}(\Gamma)$, so that $x^{-1}y \in G_{\Gamma}$. If $\Gamma$ is a complete graph on $n$ vertices, then all vertex groups of $\Gamma$ commute, so we can write $x^{-1}y = s_{1} \dots s_{n}$ where $\supp(s_{i}) = v_{i} \in V(\Gamma)$ and $v_{i} \neq v_{j}$ for all $i \neq j$. Thus $\dist_{\mathbb{E}}(x,y) \leq n$, hence $\mathbb{E}(\Gamma)$ is bounded. If $\Gamma$ splits as a join of a minsquare subgraph $\Gamma_{1}$ and a complete graph $\Gamma_{2}$ on $n$ vertices, then $G_{\Gamma} \cong \langle\Gamma_{1}\rangle \times \langle\Gamma_{2}\rangle$ and so we can write $x^{-1}y = g_{1}g_{2}$ where $g_{i} \in \langle\Gamma_{i}\rangle$. Therefore $\dist_{\mathbb{E}}(x,y) \leq n+1$, hence $\mathbb{E}(\Gamma)$ is bounded.

We now assume $\mathbb{E}(\Gamma)$ is bounded and prove this implies $\Gamma$ either is complete, minsquare, or splits as a join of a minsquare subgraph and a complete graph. The proof will proceed by induction on the number of vertices of $\Gamma$. The base case is immediate as  $\Gamma$ is complete and $\mathbb{E}(\Gamma)$ has diameter 1 when $\Gamma$ is a single vertex. Assume the conclusion holds whenever the defining graph has at most $n-1$ vertices. Let $G_\Gamma$ be a graph product of groups where $\Gamma$ contains $n \geq 2$ vertices.

\begin{claim}\label{claim:electrification}
If $\mathbb{E}(\Gamma)$ is bounded and $\Gamma$ is neither complete nor minsquare, then $\Gamma$ must split as a  join and must contain a proper minsquare subgraph.
\end{claim}
\begin{proof}
Suppose $\Gamma$ does not split as a join. By Theorem \ref{thm:diameter}, $C(\Gamma)$ is therefore unbounded. Since $\Gamma$ is not minsquare, $\mathbb{E}(\Gamma)$  can be obtained from $C(\Gamma)$ by removing some edges. In particular, if $C(\Gamma)$ has infinite diameter then so does $\mathbb{E}(\Gamma)$.  This implies that if $\Gamma$ is not minsquare and does not split as a join then $\mathbb{E}(\Gamma)$ is unbounded, contradicting our assumption. 

Now suppose $\Gamma$ does not contain any proper minsquare subgraphs. Then $\mathbb{E}(\Gamma)$ is simply  $S(\Gamma)$. Since $\Gamma$ is not complete, there exist two disjoint vertices $v,w \in V(\Gamma)$. Therefore  $\dist_{\mathbb{E}}(e,(g_{v}g_{w})^{m}) = \dist_{syl}(e,(g_{v}g_{w})^{m}) = 2m$ for any $g_{v} \in G_{v}\smallsetminus\{e\}$ and $g_{w} \in G_{w}\smallsetminus\{e\}$, hence $\mathbb{E}(\Gamma)$ is unbounded, a contradiction. 
\end{proof}

Assume that $\Gamma$ is neither complete nor minsquare, so that $\Gamma$ must contain a strict minsquare subgraph $\Lambda$ and splits as a  join by Claim \ref{claim:electrification}. By Lemma \ref{lem:joins_with_minsquare}, $\Gamma$ must split as a join of $\Gamma_1$ and $\Gamma_2$ where $\Gamma_2$ is complete and $\mathbb{E}(\Gamma)$ is the $1$--skeleton of $\mathbb{E}(\Gamma_1) \times \mathbb{E}(\Gamma_2)$. Thus, $\mathbb{E}(\Gamma)$ having bounded diameter implies $\mathbb{E}(\Gamma_1)$  must also have bounded diameter. Since $\Gamma_1$ contains at most $n-1$ vertices, the induction hypothesis then implies $\Gamma_1$  either is minsquare, complete, or splits as a join of a minsquare subgraph and a complete graph. Since $\Lambda \subseteq \Gamma_1$ contains a square, $\Gamma_1$ cannot be complete. Thus, $\Gamma_1$ is either minsquare or a join of $\Lambda$ with a  complete graph $\Omega$. Hence, $\Gamma$  splits either as a join of the minsquare subgraph $\Gamma_{1}$ and the complete graph $\Gamma_{2}$, or as a join of the minsquare subgraph $\Lambda$ and the complete graph $\Omega \bowtie \Gamma_2$. 
\end{proof}

Finally, we show that $\mathbb{E}(\Gamma)$ being a quasi-line coincides with $G_\Gamma$ being virtually cyclic. The key step of the proof is to produce two elements of $G_\Gamma$ that act as independent loxodromic elements on $C(\Gamma)$. This creates more than two directions to escape to infinity in $C(\Gamma)$, which then gives more than two directions to escape to infinity in $\mathbb{E}(\Gamma)$.

\begin{thm}\label{thm:quasi-line}
Let $G_\Gamma$ be a graph product of finite groups. The electrification $\mathbb{E}(\Gamma)$ is a quasi-line if and only if $G_\Gamma$ is virtually cyclic.
\end{thm}

\begin{proof}
A graph product of finite groups $G_\Gamma$ is virtually cyclic if and only if either $\Gamma$ is a pair of disjoint vertices each with vertex group $\Z_2$ or $\Gamma$ splits as a join $\Gamma_1 \bowtie \Gamma_2$, where $\Gamma_1$ is a pair of disjoint vertices each with vertex group $\Z_2$ and $\Gamma_2$ is a complete graph (this follows from \cite[Lemma 3.1]{BPR_v_cyclic_RACG}).  Thus, if $G_\Gamma$ is virtually cyclic, then $\mathbb{E}(\Gamma) = S(\Gamma)$ is a quasi-line by construction.

 Let us now assume $G_\Gamma$ is not virtually cyclic. If $\Gamma$   is either  minsquare, complete, or the join of a minsquare graph and a complete graph, then  $\mathbb{E}(\Gamma)$ has bounded diameter by Theorem \ref{thm:electrification} and is therefore not a quasi-line.  Let us therefore assume that $\Gamma$ is not minsquare, not complete, and does not split as a join of a minsquare graph  and a complete graph.

First assume $\Gamma$ does not split as a  join at all. 
Since the action of $G_\Gamma$ on $C(\Gamma)$ by left multiplication is acylindrical (Corollary \ref{cor:acyl}), a result of Osin \cite[Theorem 1.1]{Osin_acyl_hyp} says $G_\Gamma$ must satisfy exactly one of the following: $G_\Gamma$ has bounded orbits in $C(\Gamma)$, $G_\Gamma$ is virtually cyclic, or $G_\Gamma$ contains two elements that act loxodromically and independently on $C(\Gamma)$. Since $\Gamma$ does not split as a join, the proof of Theorem \ref{thm:diameter} implies that $G_\Gamma$ does not have bounded orbits in $C(\Gamma)$. Further, $G_\Gamma$ is not virtually cyclic by assumption.
Thus, there exist $g,h \in G_\Gamma$ such that $n \mapsto \pi_\Gamma(g^n)$ and $n \mapsto \pi_\Gamma(h^n)$ are bi-infinite quasi-geodesics in $C(\Gamma)$ whose images, $\pi_\Gamma(\sub{g})$ and $\pi_\Gamma(\sub{h})$, have infinite Hausdorff distance from each other.
Now, since $\Gamma$ is not minsquare, $C(\Gamma)$ is obtained from $\mathbb{E}(\Gamma)$ by adding edges and therefore $\dist_\Gamma(x,y) \leq \dist_{\mathbb{E}}(x,y)$ for all $x,y \in G_\Gamma$. Hence, the subsets $\sub{g}$ and $\sub{h}$  in $\mathbb{E}(\Gamma)$ are also the images of bi-infinite quasi-geodesics that have infinite Hausdorff distance from each other. This implies $\mathbb{E}(\Gamma)$ is not a quasi-line, as any two bi-infinite quasi-geodesics in a quasi-line have finite Hausdorff distance.

Now assume $\Gamma$ splits as a join. If $\Gamma$ contains no minsquare subgraph, then $\mathbb{E}(\Gamma) = S(\Gamma)$. Since the vertex groups are all finite, $S(\Gamma)$ is quasi-isometric to the word metric on $G_\Gamma$ and hence $S(\Gamma) = \mathbb{E}(\Gamma)$ is  not a quasi-line, because we assumed $G_\Gamma$ is not virtually cyclic.  Thus we can assume $\Gamma$ contains a minsquare subgraph $\Lambda$. By applying Lemma \ref{lem:joins_with_minsquare} iteratively, we have that $\Gamma$ splits as a join $\Gamma = \Gamma_1 \bowtie \Gamma_2$ such that: 
\begin{itemize}
    \item $\Gamma_1$ either does not split as a join or is minsquare;
    \item $\Gamma_2$ is a complete graph;
    \item $\mathbb{E}(\Gamma)$ is the 1--skeleton of $\mathbb{E}(\Gamma_1) \times \mathbb{E}(\Gamma_2)$.
\end{itemize}
Recall our assumption that $\Gamma$ does not split as a join of a minsquare graph and a complete graph, hence $\Gamma_{1}$ cannot be minsquare and thus must not split as a join by the first item above. Further, $\langle\Gamma_{1}\rangle$ is not virtually cyclic since it is a finite index subgroup of $G_{\Gamma}$, which is not virtually cyclic. Thus, we can apply the previous case to conclude that $\mathbb{E}(\Gamma_1)$ is not a quasi-line and hence $\mathbb{E}(\Gamma)$ is not a quasi-line.
\end{proof}

\bibliography{GraphProducts}{}

\begin{thebibliography}{10}
\providecommand{\url}[1]{\texttt{#1}}
\providecommand{\urlprefix}{\relax}
\providecommand{\selectlanguage}[1]{\relax}
\providecommand{\eprint}[2][arXiv]{\textt{#1:#2}}
\providecommand*{\MR}[1]{MR~#1}
\providecommand*{\ZBL}[2][Zbl]{#1~#2}
\providecommand*{\Zbl}[1]{\ZBL[Zbl]{#1}}
\providecommand*{\JFM}[1]{\ZBL[JFM]{#1}}
\providecommand*{\ERAM}[1]{\ZBL[ERAM]{#1}}

\bibitem{ABD}
C.~Abbott, J.~Behrstock, and M.~G. Durham, Largest acylindrical actions and
  stability in hierarchically hyperbolic groups. {W}ith an appendix by {D}.
  {B}erlyne and {J}. {R}ussell. \emph{Trans. Amer. Math. Soc. Ser. B}
  \textbf{8} (2021), \mbox{66--104}.   \Zbl{07319368} \MR{4215647}

\bibitem{BPR_v_cyclic_RACG}
H.~Baik, B.~Petri, and J.~Raimbault, Subgroup growth of virtually cyclic
  right-angled Coxeter groups and their free products. \emph{Combinatorica}
  \textbf{39} (2019), no.~4, \mbox{779--811}.   \Zbl{1449.20036} \MR{4015351}

\bibitem{BHS_HHS_AsDim}
J.~Behrstock, M.~F. Hagen, and A.~Sisto, Asymptotic dimension and
  small-cancellation for hierarchically hyperbolic spaces and groups.
  \emph{Proc. London Math. Soc. (3)} \textbf{114} (2017), no.~5,
  \mbox{890--926}.   \Zbl{1431.20028} \MR{3653249}

\bibitem{BHS_HHSI}
J.~Behrstock, M.~F. Hagen, and A.~Sisto, Hierarchically hyperbolic spaces, {I}:
  {C}urve complexes for cubical groups. \emph{Geom. Topol.} \textbf{21} (2017),
  no.~3, \mbox{1731--1804}.   \Zbl{1439.20043} \MR{3650081}

\bibitem{BHS_HHSII}
J.~Behrstock, M.~F. Hagen, and A.~Sisto, Hierarchically hyperbolic spaces {II}:
  Combination theorems and the distance formula. \emph{Pacific J. Math.}
  \textbf{299} (2019), \mbox{257--338}.   \Zbl{07062864} \MR{3956144}

\bibitem{BHS_HHS_Quasiflats}
J.~Behrstock, M.~F. Hagen, and A.~Sisto, Quasiflats in hierarchically
  hyperbolic spaces. \emph{Duke Math. J.} \textbf{170} (2021), no.~5,
  \mbox{909--996}.   \Zbl{07369844} \MR{4255047}

\bibitem{BKMM_Rigidity}
J.~Behrstock, B.~Kleiner, Y.~N. Minsky, and L.~Mosher, Geometry and rigidity of
  mapping class groups. \emph{Geom. Topol.} \textbf{16} (2012), no.~2,
  \mbox{781--888}.   \Zbl{1281.20045} \MR{2928983}

\bibitem{BR_Combination}
F.~Berlai and B.~Robbio, A refined combination theorem for hierarchically
  hyperbolic groups. \emph{Groups Geom. Dyn.} \textbf{14} (2020), no.~4,
  \mbox{1127--1203}.   \Zbl{07362555} \MR{4186470}

\bibitem{Brock_WeilPetersson}
J.~F. Brock, The {W}eil-{P}etersson metric and volumes of 3-dimensional
  hyperbolic convex cores. \emph{J. Amer. Math. Soc.} \textbf{16} (2003),
  no.~3, \mbox{495--535}.   \Zbl{1059.30036} \MR{1969203}

\bibitem{Durham_Combinatorial_Teich}
M.~G. Durham, The augmented marking complex of a surface. \emph{J. London Math.
  Soc.} \textbf{94} (2016), no.~3, \mbox{933}.    \MR{3614935}

\bibitem{HHS_Boundary}
M.~G. Durham, M.~F. Hagen, and A.~Sisto, Boundaries and automorphisms of
  hierarchically hyperbolic spaces. \emph{Geom. Topol.} \textbf{21} (2017),
  no.~6, \mbox{3659--3758}.   \Zbl{1439.20044} \MR{3693574}

\bibitem{DHS_corrigendum}
M.~G. Durham, M.~F. Hagen, and A.~Sisto, Correction to the article Boundaries
  and automorphisms of hierarchically hyperbolic spaces. \emph{Geom. Topol.}
  \textbf{24} (2020), no.~2, \mbox{1051--1073}.   \Zbl{1443.20071} \MR{4153656}

\bibitem{EMR_Teich_Rank}
A.~Eskin, H.~A. Masur, and K.~Rafi, Large-scale rank of {T}eichm\"uller space.
  \emph{Duke Math. J.} \textbf{166} (2017), no.~8, \mbox{1517--1572}.
  \Zbl{1373.32012} \MR{3659941}

\bibitem{GenNeg}
A.~Genevois, Automorphisms of graph products of groups and acylindrical
  hyperbolicity. Preprint, 2020. arXiv:1807.00622 [math.GR].

\bibitem{Gen_Thesis}
A.~Genevois, \emph{Cubical-like geometry of quasi-median graphs and
  applications to geometric group theory}. Ph.D. thesis. Universit\'{e}
  Aix-Marseille, 2017.

\bibitem{Gen_Rigidity}
A.~Genevois, Quasi-isometrically rigid subgroups in right-angled {C}oxeter
  groups. \emph{Algebr. Geom. Topol.} \textbf{22} (2022), no.~2,
  \mbox{657--708}.    \MR{4464462}

\bibitem{Genevois_Martin}
A.~Genevois and A.~Martin, Automorphisms of graph products of groups from a
  geometric perspective. \emph{Proc. London Math. Soc.} \textbf{119} (2019),
  no.~6, \mbox{1745--1779}.   \Zbl{1461.20008}

\bibitem{Green}
E.~Green, \emph{Graph products of groups}. Ph.D. thesis. University of Warwick,
  Coventry, 1990.

\bibitem{gromov1}
M.~Gromov, Hyperbolic groups. In \emph{Essays in group theory}, Math. Sci. Res.
  Inst. Publ., 8. Springer, New York, 1987, \mbox{75--263}.   \Zbl{0634.20015}
  \MR{919829}

\bibitem{KimKoberda}
S.-H. Kim and T.~Koberda, The geometry of the curve graph of a right-angled
  Artin group. \emph{Internat. J. Algebra Comput.} \textbf{24} (2014), no.~2,
  \mbox{121--169}.   \Zbl{1342.20042} \MR{3192368}

\bibitem{ManningBottleneck}
J.~F. Manning, Geometry of pseudocharacters. \emph{Geom. Topol.} \textbf{9}
  (2005), no.~2, \mbox{1147--1185}.   \Zbl{1083.20038} \MR{2174263}

\bibitem{MMI}
H.~A. Masur and Y.~N. Minsky, Geometry of the complex of curves. {I}.
  {H}yperbolicity. \emph{Invent. Math.} \textbf{138} (1999), no.~1,
  \mbox{103--149}.   \Zbl{0941.32012} \MR{1714338}

\bibitem{MMII}
H.~A. Masur and Y.~N. Minsky, Geometry of the complex of curves. {II}.
  {H}ierarchical structure. \emph{Geom. Funct. Anal.} \textbf{10} (2000),
  no.~4, \mbox{902--974}.   \Zbl{0972.32011} \MR{1791145}

\bibitem{Meier}
J.~Meier, When is the graph product of hyperbolic groups hyperbolic?
  \emph{Geom. Dedicata} \textbf{61} (1996), no.~1, \mbox{29--41}.
  \Zbl{0874.20026} \MR{1389635}

\bibitem{Osin_acyl_hyp}
D.~V. Osin, Acylindrically hyperbolic groups. \emph{Trans. Amer. Math. Soc.}
  \textbf{368} (2016), no.~2, \mbox{851--888}.   \Zbl{1380.20048} \MR{3430352}

\bibitem{Rafi_Combinatorial_Teich}
K.~Rafi, A combinatorial model for the {T}eichm\"uller metric. \emph{Geom.
  Funct. Anal.} \textbf{17} (2007), no.~3, \mbox{936--959}.   \Zbl{1129.30031}
  \MR{2346280}

\end{thebibliography}
\bibliographystyle{ggd2021}
\end{document}